\documentclass{article}

\usepackage[T1]{fontenc}
\usepackage{lmodern}

\usepackage{mathtools}
\usepackage{amssymb}
\usepackage{amsthm}

\usepackage{caption} 

\usepackage{xcolor}
\usepackage{csquotes}
\usepackage{booktabs}
\usepackage{makecell}
\usepackage{graphicx}

\usepackage{cite}

\usepackage{tikz}
\usetikzlibrary{arrows.meta,positioning,calc}

\usepackage{hyperref}
\usepackage{geometry}
\usepackage[english]{babel}
\usepackage{tcolorbox}
\usepackage{orcidlink}

\mathtoolsset{showonlyrefs}

\newcommand{\FundingLogos}{%
  \raisebox{0pt}{\includegraphics[height=1.5cm]{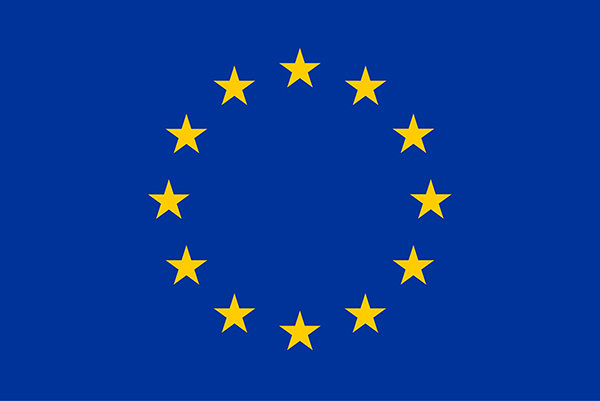}}%
  \hspace{1em}%
  \raisebox{0pt}{\includegraphics[height=1.5cm]{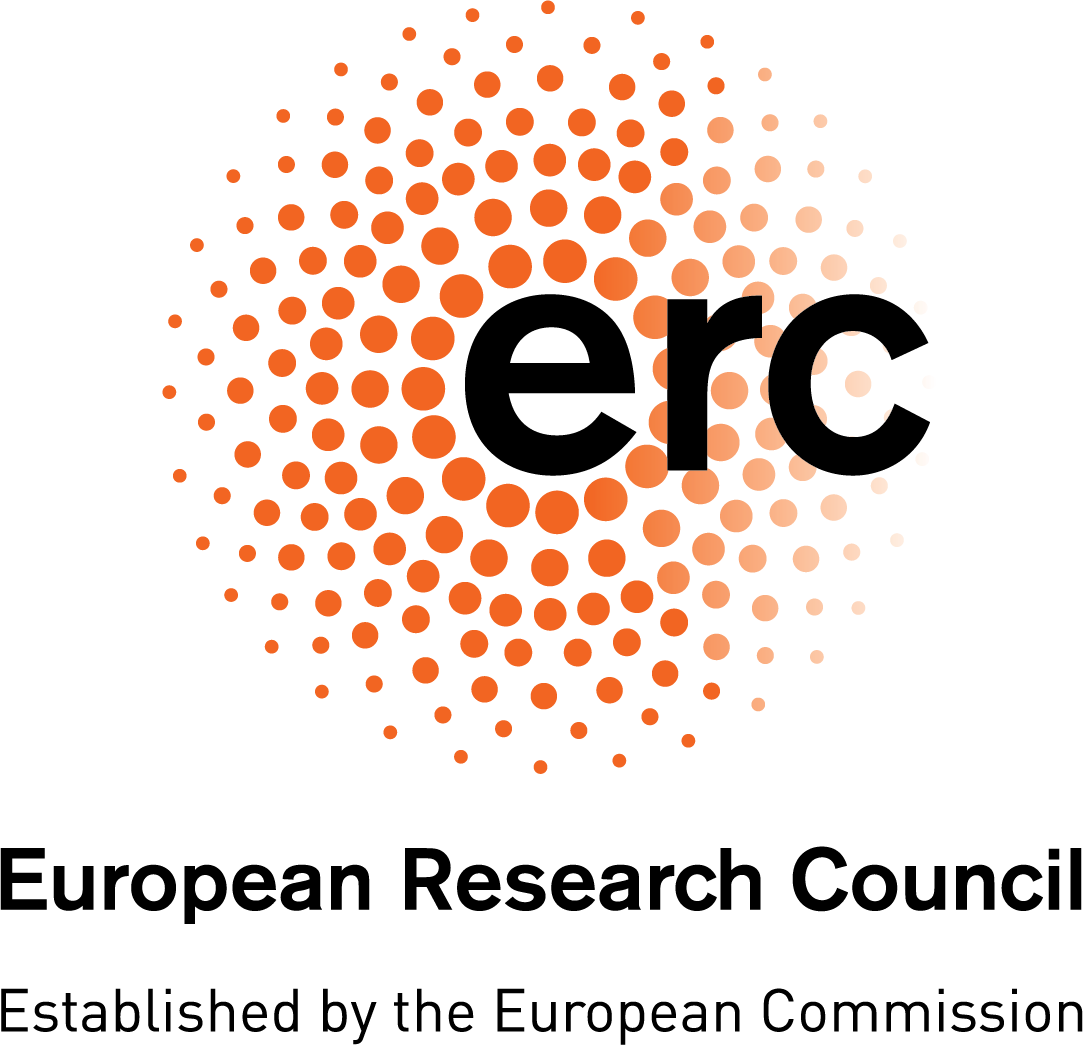}}%
}

\tcbset{
  colback=white,
  boxrule=0.5pt,
  arc=2pt,
  left=4pt,
  right=4pt,
  top=4pt,
  bottom=4pt
}

\newcounter{dummy}
\numberwithin{dummy}{section}

\newtheorem{remark}[dummy]{Remark}
\newtheorem{theorem}[dummy]{Theorem}
\newtheorem{lemma}[dummy]{Lemma}
\newtheorem{example}[dummy]{Example}

\newtheorem{cor}[dummy]{Corollary}
\newtheorem{proposition}[dummy]{Proposition}

\newcommand{\R}{\mathbb R}

\let\P\relax
\newcommand{\P}{\mathcal P}

\newcommand{\F}{\mathcal F}
\newcommand{\NN}{\mathcal N}

\DeclareMathOperator{\lin}{lin}
\DeclareMathOperator{\diag}{diag}
\DeclareMathOperator{\loc}{loc}
\DeclareMathOperator{\tr}{tr}
\DeclareMathOperator{\Attn}{Attn}
\DeclareMathOperator{\LN}{LN}
\DeclareMathOperator{\reg}{reg}
\DeclareMathOperator{\MLP}{MLP}
\DeclareMathOperator{\GL}{GL}
\DeclareMathOperator{\ran}{ran}
\DeclareMathOperator{\Sym}{Sym}
\DeclareMathOperator{\eff}{eff}
\DeclareMathOperator{\SM}{SM}
\DeclareMathOperator{\grad}{grad}

\renewcommand{\d}{\mathop{}\!\mathrm{d}}

\allowdisplaybreaks

\mathtoolsset{showonlyrefs}

\title{SympFormer: Inertial Attention Blocks from Damped Hamiltonian Dynamics on Spaces of Probability Measures}
\author{%
Viktor Stein\orcidlink{0000-0001-9040-820X}\thanks{Equal contribution.}\  
\thanks{Department of Mathematics, Technical University of Munich, and the Munich Center for Machine Learning. {\ttfamily viktor.stein@tum.de}, \url{https://neitalg.github.io/NEITALG/}.} {},
Levin Maier%
\protect\footnotemark[1]%
\ 
\thanks{School of Mathematics, Institute for Advanced Study, 1 Einstein Drive, Princeton, NJ 08540, USA, {\ttfamily lmaier@ias.edu}.}{},
Wuchen Li\,\orcidlink{0000-0002-2218-5734}%
\thanks{%
Department of Mathematics,
University of South Carolina,
Columbia.
1523 Greene St, Columbia, SC 29208, USA, {\ttfamily wuchen@mailbox.sc.edu}.
},
Gabriele Steidl\,\orcidlink{0000-0001-7486-6315}%
\thanks{Institute of Mathematics,
Technische Universit\"at Berlin,
Stra{\ss}e des 17.\ Juni 136,
10623 Berlin, Germany,
{\ttfamily steidl@math.tu-berlin.de},
\url{https://tu.berlin/imageanalysis}.
}%
}
\date{\today}

\begin{document}
\maketitle

\begin{abstract}
The attention mechanism underpins modern large language models, yet their dynamics remain only partially understood. Recent perspectives interpret attention blocks as interacting particle systems, whose mean-field limits correspond to gradient flows of energy functionals on spaces of probability densities on $\R^d$ equipped with Wasserstein-type metrics. We extend this viewpoint by deriving accelerated attention blocks from inertial, Nesterov-type damped Hamiltonian dynamics on these density spaces, developing the geometry and differential calculus needed to make the construction rigorous.
In particular, we introduce a local-mobility geometry for the Otto–Wasserstein and softmax–Wasserstein settings, and a quotient Stein geometry that resolves the degeneracy of the linear-attention kernel to yield a well-defined metric and differential calculus. We provide rigorous foundations when softmax attention admits a gradient-flow formulation and when its inertial counterpart admits a damped Hamiltonian on explicitly characterized domains. For linear attention, we obtain a finite-dimensional solution theory with sharp conditions for global existence versus finite-time second-moment blow-up. We further prove a quantitative link between continuous moment blow-up and the discrete-layer depth at which token norms cross a prescribed threshold. Finally, motivated by the finite-token geometry, we propose geometric corrections to token and momentum updates. In decoder-only, nanoGPT-style experiments on TinyStories and OpenWebText, these corrections lower validation negative log-likelihood on next-token prediction relative to momentum-only controls.
\end{abstract}


\section{Introduction}

This paper develops foundations of infinite-dimensional Nesterov acceleration and connects it to the mathematical theory of deep transformers—two fields that at first glance appear largely unrelated. To motivate this connection, we first consider the role of attention in transformers.

The transformer architecture introduced in \emph{Attention Is All You Need}~\cite{VSPUJGKP2017}
forms the backbone of modern large language models.
Its attention layers update token representations by
allowing tokens to exchange information across a sequence.
The widespread use and societal importance of these systems
make a mathematical understanding of their behavior important. 
Although attention has become an active subject of mathematical research,
with contributions to approximation theory, stability,
learning, and the dynamics of token interactions \cite{YBRRK2020,FHP2025,DCL2021,VBC2020,KPM2021,CAP2024,LiEtAl2026},
even the behavior of the original transformer architecture remains
only partially understood. 

Attention layers can be viewed as discrete dynamical systems for token representations
\cite{VBC2020,DCL2021}
and, in a continuous-depth limit, as interacting particle systems
\cite{LLHSDQWL2019,SABP2022,GLPR2025}.
Under suitable assumptions, a subsequent mean-field limit describes
the evolution of the token distribution through a non-linear transport equation,
allowing PDE methods to be used to study collective behavior
\cite{BKKRW2025,CACP2025}.
We provide further details in Section~\ref{sec:trans}
and illustrate this passage in Figure~\ref{fig:attention_dynamics}.
In their recent work, Castin--Ablin--Carrillo--Peyr\'e~\cite{CACP2025}
establish well-posedness of the mean-field PDE for softmax attention
for compactly supported initial data and study Gaussian initial data separately.
General non-Gaussian initial densities with unbounded support remain
outside their analysis.
As in their work, we consider probability densities on $\R^d$,
allowing unbounded support, which lets us study growth in token magnitudes,
which sphere-constrained models exclude
\cite{GLPR2023,GKPR2024,GLPR2025,BKKRW2025,MDCBS2026}.

The transformer PDE arising in the mean-field limit can also be
interpreted formally as a gradient flow of an interaction energy
on a space of probability densities equipped with a suitable
Riemannian metric.
For softmax attention, Castin et al.~\cite[Sec.~5]{CACP2025}
develop this interpretation, but explicitly leave its rigorous
geometric justification open.
Before moving on, we highlight several obstacles to making this
gradient-flow interpretation rigorous.
The first concerns differentiation itself: the admissible density
classes on $\R^d$ are not assumed to carry a smooth manifold structure.
Therefore the relevant tangent and cotangent spaces and notions of differentiation
must be specified explicitly.
A further difficulty comes from weak Riemannian metrics, for which
even smooth energies need not admit a gradient vector field; see
Section~\ref{ss:background_inf_dim_geometry}.
To define the Riemannian metrics considered here, we follow
a guiding principle: associate a differential operator, called the
\emph{mobility operator}, with the attention kernel and use its
inverse to define the metric.
Making this construction rigorous requires identifying suitable
function spaces on which this operator is invertible.
Figure~\ref{fig:attention_geometry} illustrates this passage from
attention kernels to Riemannian metrics.

A natural extension of this viewpoint is inspired by Nesterov acceleration~\cite{N1983}
in finite-dimensional optimization, which, roughly speaking, introduces
momentum to accelerate gradient descent; see Section \ref{sec:Nesterov}.
On spaces of probability densities, this leads us to study
infinite-dimensional damped Newton systems, a particular class
of damped Hamiltonian systems.
Newton dynamics on such spaces also arise in geometric hydrodynamics;
see Khesin--Misio{\l}ek--Modin~\cite{khesin2021geometric}.
Wang and Li~\cite{WL2022} study such damped Newton dynamics
as accelerated gradient flows on spaces of probability densities.
However, their geometric construction has serious analytic gaps
even in the Otto--Wasserstein and Stein settings;
see Section~\ref{sec3} and
Appendix~\ref{app:Non-surjectivity of the Stein operator on compact domains}.
The authors of~\cite{SLS2026} use this framework to derive inertial PDEs
for linear and softmax attention, but leave the underlying geometric
construction at a formal level. 

These connections motivate the following question:
\emph{Can PDE methods, functional analysis, and geometry
applied to the mean-field limit yield new insights into
the dynamics of transformers with finitely many tokens
and layers?}

\paragraph{Main results.}
Our paper makes the following main contributions, whose connections we summarize in Figure~\ref{fig:intro-roadmap}.

\medskip\noindent
\textbf{1. Foundations for infinite-dimensional Nesterov acceleration.}
For linear and softmax attention, we develop the geometric foundations
needed to make the formal gradient-flow and damped Hamiltonian
formulations rigorous on specified domains
(Corollary~\ref{thm:sm-gradient},
Theorem~\ref{thm:sm-hamiltonian},
Proposition~\ref{prop:lin-covariance-geometry}, and
Lemma~\ref{lem:lin-canonical-hamiltonian}).
This provides the rigorous justification left open
in~\cite[Sec.~5.4]{CACP2025}.
For softmax attention, we introduce a \emph{local-mobility geometry},
which includes both Otto--Wasserstein and softmax--Wasserstein
geometry on smooth probability densities over $\R^d$
(Theorem~\ref{prop:LM_operator_invertible}).
For linear attention, we introduce a \emph{quotient Stein geometry}
adapted to the bilinear attention kernel
(Theorem~\ref{prop:isomorphism}).
Since a differentiable structure on the spaces of probability
densities involved is not given a priori, we develop the associated
differential calculus directly on these spaces
(Appendices~\ref{app:diif_in_in_local_mobility_calculus}
and~\ref{app:diff_stein_geom}).

\medskip\noindent
\textbf{2. Well-posedness and blow-up for the linear-attention
transformer PDE and its accelerated counterpart.}
By choosing coordinates adapted to the evolving density, we reduce
the infinite-dimensional dynamics of both equations to finite-dimensional
matrix ODEs.
The resulting two-way PDE--ODE correspondences yield existence
and uniqueness of strong solutions for admissible non-Gaussian initial
data with potentially unbounded support
(Theorems~\ref{thm:lin-ordinary-correspondence}
and~\ref{thm:lin-correspondence}).
In the accelerated case, we require a quadratic initial momentum potential.
These reductions also give conditions for global-in-time existence
and finite-time second-moment blow-up
(Proposition~\ref{prop:lin-ordinary-commuting} and
Theorem~\ref{thm:lin-bounded-damping-blowup}).

\medskip\noindent
\textbf{3. From PDE blow-up to growth of token norms across layers.}
For accelerated linear attention at fixed token count, we use
continuous moment growth to predict the first explicit-Euler layer
crossing a prescribed directional second-moment threshold.
The prediction has a bounded error in layer count as the step size
tends to zero (Theorem~\ref{thm:lin-transformer-depth}).
As a corollary, continuum moment blow-up implies that every prescribed
token-norm bound is exceeded in the discrete layers for sufficiently
small step sizes
(Corollary~\ref{cor:lin-transformer-amplification}).
To the best of our knowledge, these are the first results linking
continuum moment blow-up to quantitative threshold-depth predictions
and arbitrarily large token norms in discrete attention networks
at fixed token count.

\medskip\noindent
\textbf{4. Numerical evaluation of geometric attention updates.}
We derive corrections to token and momentum updates from the
finite-token softmax metric
(Lemma~\ref{lem:sm-particle-connection}).
On TinyStories and OpenWebText, we compare practical versions of
these updates with matched momentum-only controls.
The corrections give small reductions in mean validation loss
in some settings, at increased computational cost
(Section~\ref{sec:numerics}).

\medskip\noindent
We discuss the relation to previous work in
Section~\ref{sec:Related results}.

\color{black}

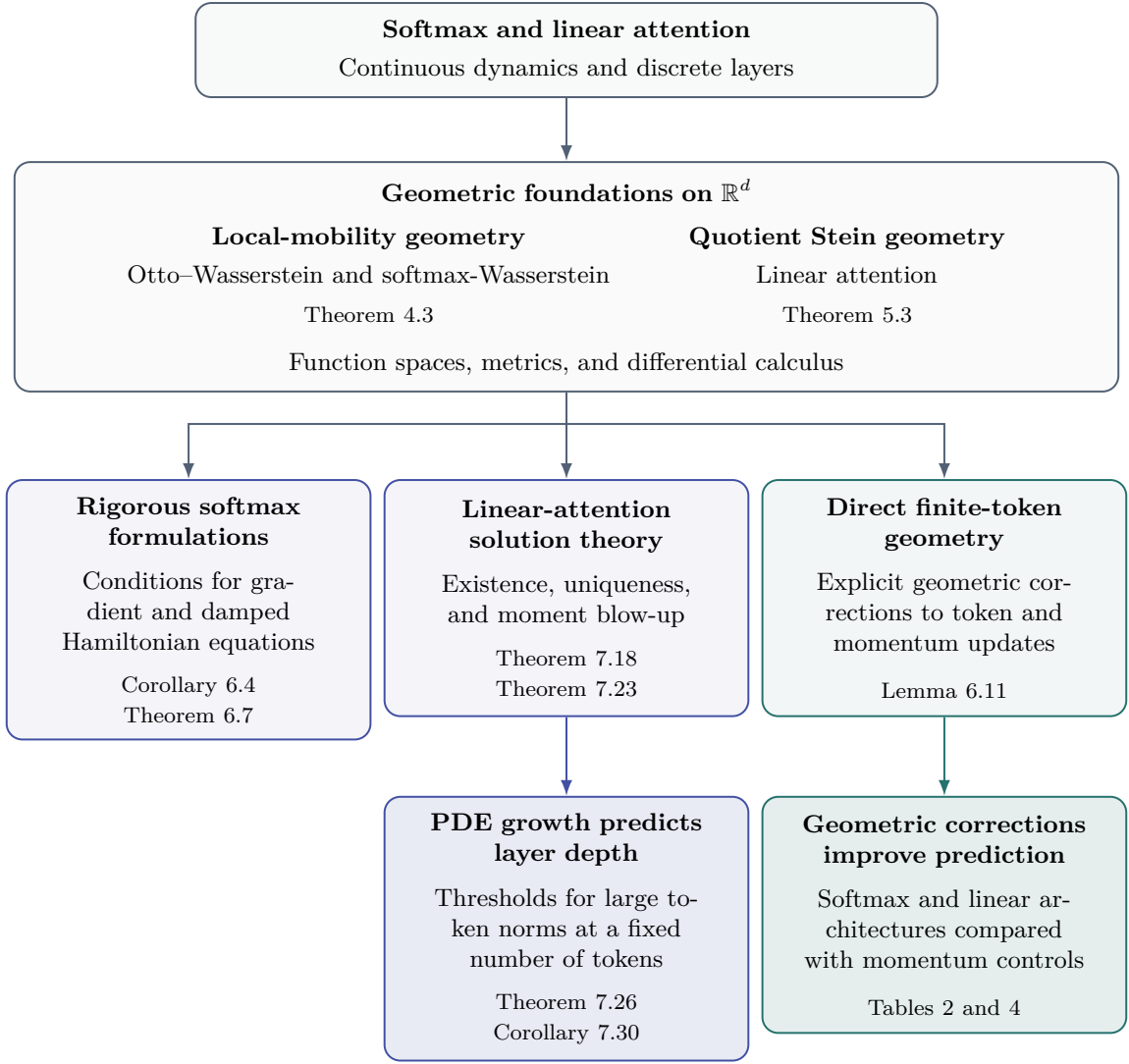
\begin{figure}[!htbp]
  \centering
  \definecolor{introBlue}{HTML}{4250A3}
  \definecolor{introTeal}{HTML}{24716D}
  \definecolor{introGray}{HTML}{526074}

  \resizebox{\linewidth}{!}{%
  \begin{tikzpicture}[
    font=\small,
    introbox/.style={
      draw,
      rounded corners=2mm,
      line width=.7pt,
      align=center,
      text width=4.2cm,
      minimum height=3cm,
      inner sep=6pt
    },
    theory/.style={
      introbox,
      draw=introBlue,
      fill=introBlue!6
    },
    design/.style={
      introbox,
      draw=introTeal,
      fill=introTeal!6
    },
    flow/.style={
      -{Latex[length=2mm]},
      line width=.7pt,
      draw=introGray
    }
  ]

  \node[
    introbox,
    text width=9cm,
    minimum height=1.1cm,
    draw=introGray,
    fill=introGray!4
  ] (attention) {
    \textbf{Softmax and linear attention}\\[3pt]
    Continuous dynamics and discrete layers
  };

  \node[
    introbox,
    text width=13.6cm,
    minimum height=2.6cm,
    below=8mm of attention,
    draw=introGray,
    fill=introGray!3
  ] (foundation) {
    \textbf{Geometric foundations on $\R^d$}\\[7pt]
    \begin{tabular}{@{}c@{\hspace{1cm}}c@{}}
      \textbf{Local-mobility geometry}
      &
      \textbf{Quotient Stein geometry}
      \\[3pt]
      Otto--Wasserstein and softmax-Wasserstein
      &
      Linear attention
      \\[3pt]
      {\footnotesize
        Theorem~\ref{prop:LM_operator_invertible}}
      &
      {\footnotesize
        Theorem~\ref{prop:isomorphism}}
    \end{tabular}\\[7pt]
    Function spaces, metrics, and differential calculus
  };

  \draw[flow] (attention) -- (foundation);

  \node[
    theory,
    anchor=north
  ] (softmax)
    at ($(foundation.south)+(-4.8cm,-1.1cm)$) {
    \textbf{Rigorous softmax\\formulations}\\[5pt]
    Conditions for gradient and
    damped Hamiltonian equations\\[6pt]
    {\footnotesize
      Corollary~\ref{thm:sm-gradient}\\
      Theorem~\ref{thm:sm-hamiltonian}}
  };

  \node[
    theory,
    anchor=north
  ] (linear)
    at ($(foundation.south)+(0,-1.1cm)$) {
    \textbf{Linear-attention\\solution theory}\\[5pt]
    Existence, uniqueness,\\
    and moment blow-up\\[6pt]
    {\footnotesize
      Theorem~\ref{thm:lin-correspondence}\\
      Theorem~\ref{thm:lin-bounded-damping-blowup}}
  };

  \node[
    design,
    anchor=north
  ] (tokens)
    at ($(foundation.south)+(4.8cm,-1.1cm)$) {
    \textbf{Direct finite-token\\geometry}\\[5pt]
    Explicit geometric corrections
    to token and momentum updates\\[6pt]
    {\footnotesize
      Lemma~\ref{lem:sm-particle-connection}}
  };

  \coordinate (fork)
    at ($(foundation.south)+(0,-4mm)$);

  \draw[line width=.7pt,draw=introGray]
    (foundation.south) -- (fork);
  \draw[flow] (fork) -| (softmax.north);
  \draw[flow] (fork) -- (linear.north);
  \draw[flow] (fork) -| (tokens.north);

  \node[
    theory,
    fill=introBlue!12,
    below=10mm of linear
  ] (depth) {
    \textbf{PDE growth predicts\\layer depth}\\[5pt]
    Thresholds for large token norms
    at a fixed number of tokens\\[6pt]
    {\footnotesize
      Theorem~\ref{thm:lin-transformer-depth}\\
      Corollary~\ref{cor:lin-transformer-critical-depth}}
  };

  \node[
    design,
    fill=introTeal!12,
    below=10mm of tokens
  ] (evaluation) {
    \textbf{Geometric corrections\\improve prediction}\\[5pt]
    Softmax and linear architectures
    compared with momentum controls\\[6pt]
    {\footnotesize
      Tables~\ref{tab:softmax_fiveway_uncertainty}
      and~\ref{tab:other_attention_endpoints}}
  };

  \draw[flow,draw=introBlue]
    (linear) -- (depth);
  \draw[flow,draw=introTeal]
    (tokens) -- (evaluation);

  \end{tikzpicture}%
  }

  \caption{From attention geometry to analysis and architecture.
  The local-mobility and quotient Stein constructions provide
  function spaces and a differential calculus for attention
  dynamics on $\R^d$.
  For softmax attention, they identify when the equations
  admit regular gradient and Hamiltonian formulations.
  For inertial linear attention under the stated assumptions,
  continuous blow-up predicts large token norms in discrete
  layers, even at a fixed number of tokens.
  A direct finite-token geometric construction motivates
  corrections to token and momentum updates.
  The reported experiments evaluate their predictive gains
  and additional computational cost.}
  \label{fig:intro-roadmap}
\end{figure}
\paragraph{Organization.}
Section~\ref{sec:trans} recalls the interpretation of
attention layers as interacting particle systems and
introduces their continuum descriptions.
Section~\ref{sec:Nesterov} reviews Nesterov acceleration
and motivates its extension to attention through geometric
formulations of density dynamics.
The obstruction to using the full formal tangent space
in Stein geometry is established in
Appendix~\ref{app:Non-surjectivity of the Stein operator on compact domains}.
Sections~\ref{sec3} and~\ref{sec:stein} develop the
local-mobility and quotient Stein frameworks, respectively.
They identify suitable function spaces, establish
invertibility of the corresponding mobility operators,
and define the induced metrics and gradients.
The notions of differentiability needed for the gradient
and Hamiltonian formulations are developed in
Appendices~\ref{app:diif_in_in_local_mobility_calculus}
and~\ref{app:diff_stein_geom}.
Section~\ref{sec:softmax-regular-dynamics} treats softmax
attention. It distinguishes pointwise well-definedness
from regular gradient and Hamiltonian formulations,
identifies their admissibility conditions, and derives
the finite-token geometric correction used in the
numerical section.
Section~\ref{sec:lin-accelerated-linear-attention}
applies the quotient Stein framework to linear attention.
It introduces covariance coordinates, establishes the
ODE--PDE correspondence and strong well-posedness, and
studies global existence and finite-time moment blow-up.
It then relates continuous growth to activation thresholds
and effective depth in discrete attention layers.
Section~\ref{sec:Discrete SympFormer architectures} introduces the
discretizations and resulting architectures, and
Section~\ref{sec:numerics} compares their linear and
softmax variants in language-modeling experiments.
We discuss the relation to previous work in
Section~\ref{sec:Related results}.

\part{Attention Models and Acceleration}

\section{Transformers as Dynamical Systems} \label{sec:trans}
In this section, we recall the mean-field description of
transformer dynamics and discuss possible geometric interpretations
of the resulting PDE; see, e.g., \cite{CACP2025,GLPR2025}.
We first specify the attention weights and residual layers, then
describe the formal passage to continuous-depth particle dynamics
and their mean-field continuity equation, as illustrated in
Figure~\ref{fig:attention_dynamics}. Finally, we discuss the geometric requirements for a gradient-flow
interpretation and the obstacles to making it rigorous.

\begin{figure}[ht]
\centering
\begin{tikzpicture}[
    box/.style={
        draw=blue!55!black, fill=blue!10, text=black,
        rounded corners=3pt, line width=0.7pt,
        minimum height=1.8cm, text width=3.4cm,
        align=center, inner sep=8pt, font=\small
    },
    arrow/.style={-{Stealth[length=2.2mm]}, line width=0.7pt},
    node distance=1.3cm
]
\node[box] (layers)
    {\textbf{Residual attention}\\[5pt]
     $X_i^{(k)}\mapsto X_i^{(k+1)}$};

\node[box, right=of layers] (particles)
    {\textbf{Particle dynamics}\\[5pt]
     $\dot X_i=\Gamma_{\mathbf X}(X_i)$};

\node[box, right=of particles] (density)
    {\textbf{Density evolution}\\[5pt]
     $\partial_t\rho=-\nabla\cdot(\rho\Gamma_\rho)$};

\draw[arrow] (layers)
    -- node[above, font=\scriptsize] {$h\to0$} (particles);
\draw[arrow] (particles)
    -- node[above, font=\scriptsize] {$N\to\infty$} (density);
\end{tikzpicture}
\caption{From attention layers to transport dynamics:
with shared weights, the formal continuous-depth limit turns
residual updates into an interacting particle system.
A suitably normalized mean-field limit then describes the
token distribution through a continuity equation.}
\label{fig:attention_dynamics}
\end{figure}
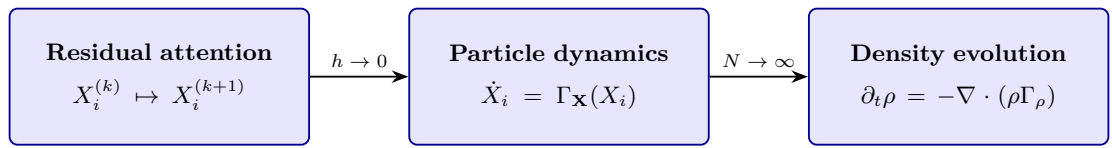
\paragraph{Token configurations and attention weights.}
Transformers act on token configurations
$$\mathbf{X} \coloneqq (X_1,\ldots, X_N) \in (\R^d)^N.$$ 
The literature also considers embeddings in $(\mathbb S^d)^N$; see, e.g., \cite{BKKRW2025,CKS2026,GLPR2023,CLPR2025}, which this paper does not handle.
Self-attention combines information
across tokens, whereas MLP layers apply {token-wise}
non-linear transformations. Let
$Q, K \in \R^{m\times d}$ and $V \in \R^{d\times d}$, $m\leq d$,
denote the query, key, and value matrices, and set
$A \coloneqq K^{\top} Q \in \R^{d\times d}$.
For the geometric constructions in
Sections~\ref{sec:lin-accelerated-linear-attention} and \ref{sec:softmax-regular-dynamics}, we assume that $m = d$
and that $A$ is a symmetric matrix.
In the following, we ignore the causal masking, layer normalization, the standard attention scaling, the output projection, and the multi-head structure usually used in transformer architectures.

Recall that the softmax function is defined for $w\in\R^d$ by
$$
\text{softmax} (w) \coloneqq \big(\exp(w_i) /\sum_{j=1}^d \exp(w_j) \big)_{i=1}^d.
$$ 
We consider the following two choices of attention
weights $\kappa = \kappa_{\mathbf{X}} \colon \R^d \times \R^d \to \R$:
\begin{equation} \label{eq:self_attn_kernel}
\kappa_{\mathbf{X}} (X_i,X_j) = 
\left\{
\begin{array}{ll}
   \frac{\exp\left((X_j)^{\top} A X_i \right)}{\sum_{\ell=1}^N \exp\left((X_{\ell})^{\top} A X_i\right)}
   =
   \text{softmax}_j\big((X_\ell^{\top} A X_i)_{\ell = 1}^{N}\big) &\quad  \text{softmax self-attention},\\[3ex]
 (X_j)^{\top} A X_i &\quad \text{linear self-attention}.
\end{array}
\right.
\end{equation}
We use the term \enquote{linear self-attention} for
unnormalized dot-product attention, to distinguish it from other linear
self-attention mechanisms \cite{KVPF2020}.
Note that the evaluation cost of transformers with softmax attention scales quadratically with sequence length. Kernel attention addresses this by replacing softmax with a more general kernel function. For example, \cite{RHS2026} aims to identify kernels that retain the expressivity of attention while enabling quasi-linear computation.

\paragraph{Residual attention and MLP layers.}
Let $(X^0_1,\cdots, X^0_N)$ be the input tokens, let $h>0$ be the step
size, and let $k\in\mathbb{N}$ denote the layer index. The attention
matrices usually differ between layers. For the continuous-depth
description below, we instead keep $Q$, $K$, and $V$ fixed across layers
(\enquote{weight sharing}), as in architectures such as ALBERT
\cite{Lan2020}. The \emph{residual attention sublayer} is
\begin{equation} \label{eq:self-attention}
X^{(k + 1)}_i = X^{(k)}_i + h \sum_{j=1}^N \kappa_{\mathbf{X}}(X^{(k)}_i,X^{(k)}_j) V X^{(k)}_j,
\end{equation}
where we have to replace $h$ by $h/N$ for linear attention.
For weight matrices $W,\widetilde W\in\R^{d\times d}$, a bias
$b\in\R^d$, and a non-linear function $\sigma\colon \R\to\R$ applied
componentwise, the \emph{residual MLP sublayer} is
\begin{equation} \label{eq:MLP}
    X^{(k + 1)}_i= X^{(k)}_i +  hW \sigma (\widetilde W X^{(k)}_i + b),     
\end{equation}
where the MLP parameters may differ between layers. The updates
\eqref{eq:self-attention} and \eqref{eq:MLP} describe the two sublayers
separately. In the following passage to particle and density dynamics,
we consider only the attention update \eqref{eq:self-attention}.

\paragraph{Particle dynamics.}
Given input tokens
$\mathbf{X}(0) = (X_1(0), \dots, X_N(0)) \in (\R^d)^N$,
we interpret the attention iteration as a discretization in depth.
Formally, as the step size $h\to 0$, this gives
\begin{equation}\label{eq:attention}
\dot{X}_i(t) = \Gamma_{{\mathbf X}(t)}(X_i(t)), 
\qquad 1 \le i \le N,
\end{equation}
where $\Gamma_{\mathbf X}\colon \R^d\to\R^{d}$ is the attention vector
field determined by $Q$, $K$, $V$, and the current token configuration
$\mathbf{X}(t)$. For softmax attention, it is given by
\begin{equation*}
 \Gamma_{\mathbf X}(x)\coloneqq 
\sum_{j=1}^N
\frac{\exp\big(Qx \cdot KX_j\big)}
{\sum_{\ell=1}^N \exp\big(Qx \cdot KX_\ell\big)}
\, VX_j,\qquad  \textrm{for }x\in\R^d.
\end{equation*}
The velocity of each token in \eqref{eq:attention} is a similarity-weighted average of the value vectors. The resulting dynamics form a non-linear interacting particle system.

\paragraph{Density continuity equation.}
For a probability density $\rho$, the corresponding attention field is
\begin{equation*}
    \Gamma_{\rho} \colon \R^d \to \R^d, \qquad
    x \mapsto \int_{} V y \, \kappa_{\rho}(x, y) \rho(y)  \d y,
\end{equation*}
for all $\rho$ such that this integral is finite.
The two choices in \eqref{eq:self_attn_kernel} lead to the density-dependent
weights
\begin{equation*}
    \kappa_{\rho}(x, y)
    \coloneqq \begin{cases} 
    \frac{\exp(y^{\top} A x)}{\int_{} \exp(z^{\top} A x) \rho(z) \d z} & \text{softmax self-attention},
    \\[2ex]
    y^{\top} A x, & \text{linear self-attention}.
    \end{cases}
\end{equation*}
In the formal mean-field limit, the self-attention dynamics
\eqref{eq:attention} are described on $\R^d$ by the continuity equation
\begin{equation}\label{eq:transformer_PDE}
    \partial_t \rho_t
    = -\nabla\cdot(\rho_t\Gamma_{\rho_t}), \qquad t>0,
\end{equation}
see \cite{CACP2025,GLPR2025} and also
\cite{BKKRW2025,GLPR2025,SABP2022,ZZCD2022}.

\paragraph{Geometry of the mean-field limit.}
Under suitable assumptions on the attention matrices, the mean-field
equation \eqref{eq:transformer_PDE} admits a formal gradient-flow
interpretation. Recall that a Riemannian gradient flow on a manifold
$\mathcal M$ is specified by a pair $(\F,G)$: a smooth functional
$\F \colon \mathcal M\to\R$ and a Riemannian metric $G$. Thus, identifying
an energy is only part of the construction; the geometry determines
the corresponding direction of steepest descent.

Making this interpretation rigorous requires specifying the underlying
space of probability densities. A central distinction is whether these
densities have compact support or extend over all of $\R^d$.
In the latter case, additional integrability and coercivity conditions
are needed to justify the formal geometric operators. For softmax
attention, a general rigorous justification of this geometry for
non-compactly supported densities was left open in
\cite[Sec.~5]{CACP2025}.

This geometric viewpoint provides a way to introduce momentum into
attention dynamics, following the principle of Nesterov acceleration,
as discussed in the next section. Discretizing the resulting
inertial equations yields attention layers that propagate both
token positions and a momentum variable across network depth.

We close this subsection by recalling the guiding question
of this paper:
\begin{quote}
\emph{Can PDE methods, functional analysis, and geometry
applied to the mean-field limit yield new insights into
the dynamics of transformers with finitely many tokens
and layers?}
\end{quote}

\section{Nesterov's Acceleration Method and Attention Models} \label{sec:Nesterov}
The aim of this section is to explain how Nesterov acceleration
provides a way to introduce momentum into attention dynamics
and to identify the geometric constructions needed to justify
this approach.

We first recall Nesterov acceleration in Euclidean space and
its continuous-time formulation as the damped Hamiltonian
system \eqref{hamil}. We then describe the corresponding
construction for smooth probability densities on compact
manifolds, leading to the accelerated Otto--Wasserstein
equations \eqref{eq:accelerated_W2_gradient_flow}.
This construction relies on the mobility operator being
invertible between suitable cotangent and tangent spaces.
We explain why this requirement presents an obstruction
in Stein geometry.

Finally, we connect this viewpoint to the linear and softmax
attention models introduced in Section~\ref{sec:trans}.
Their kernels give rise to different mobility operators
and hence different formal geometric descriptions.
We state the resulting accelerated systems in
\eqref{eq:nesterov-softmax-system} and \eqref{eq:nesterov-stein-system},
and identify the
function-space and differentiability questions that must
be resolved to make them rigorous.
These questions motivate the local-mobility calculus
developed in Section~\ref{sec3} and the quotient Stein
geometry developed in Section~\ref{sec:stein}.

\subsection{Nesterov Acceleration on \texorpdfstring{$\R^d$}{Rd}}
Nesterov acceleration~\cite{N1983} is a well-known technique for improving the convergence rate of gradient descent methods. For an objective function $F\in C^1(\R^d;\R)$, we consider the optimization problem
\begin{equation}\label{problem}
\min_{x \in \R^d} F(x).
\end{equation}
The \emph{gradient descent} sequence $(x^{(k)})_{k \in \mathbb N}$ starting at $x_0 \in \R^d$ with
step size $\tau$ is given by
$$
\begin{cases}    
x^{(k+1)} = x^{(k)} - \tau \nabla F(x^{(k)}), \\ x^{(0)} = x_0.
\end{cases}$$ 
Nesterov's \emph{accelerated gradient descent} is obtained by choosing an extrapolation coefficient as
\begin{equation} \label{low-d}
\begin{cases}
x^{(k + 1)} = y^{(k)} - \tau \nabla F(y^{(k)}),  
\\[2mm]
y^{(k + 1)} = x^{(k + 1)} + (1 - \frac{3}{k+3}) (x^{(k + 1)}-x^{(k)}), \\ 
x^{(0)} = x_0, \ y^{(0)} = x_0.
\end{cases}
\end{equation}
Substituting $t= k \sqrt{\tau}$ and letting $\tau \to 0$, Nesterov’s method can be interpreted through the following continuous-time dynamical system; see~\cite{su2016differential,WWJ2016}:
\begin{equation} \label{low-c}
\begin{cases}
\ddot{x}_t + \dfrac{3}{t} \dot{x}_t + \nabla F(x_t) = 0,  
\\[2mm]
 x_t|_{t = 0} = x_0,\ \dot{x}_t|_{t = 0} = 0.
\end{cases}
\end{equation}
For a convex function $F$ possessing a minimizer, and any step size $\tau \in (0,\frac1L]$, 
where $L$ is the Lipschitz constant of $\nabla F$, the discrete scheme \eqref{low-d} exhibits a convergence rate 
$F(x^{(k)}) - F^* \le \mathcal O(\frac{\|x_0-x^*\|^{2}}{\tau k^2})$, 
where $x^*$ is a minimizer of $F$ and $F^* = F(x^*)$ \cite[p.~2]{su2016differential}.
This rate is optimal among all methods having only information about the gradient of $F$ at consecutive iterates \cite{Nesterov2004}. 
For the continuous dynamics \eqref{low-c}, we have similarly
$F(x(t)) - F^* \le \mathcal O(\frac{\|x_0-x^*\|^{2}}{t^2})$.

Introducing a momentum variable $p_t$ and $\alpha_t \coloneqq \frac3t$, the second-order ODE \eqref{low-c} can  be rewritten as the first-order system
\begin{equation} \label{first_order}
\begin{cases}
\dot x_t = p_t, \\ 
\dot p_t +\alpha_t p_t  = - \nabla F(x_t), \\ 
x_t|_{t = 0} = x_0, \ p_0 = 0.
\end{cases}
\end{equation}
It was observed, e.g., in \cite{TM2019}, that with the \emph{Hamiltonian} and its derivatives
$$H(x,p) = \frac12 \|p\|^2 + F(x), 
\quad \nabla_p H(x,p) = p, 
\quad \nabla_x H(x,p) = \nabla F(x),$$
the first-order system \eqref{first_order} becomes a \emph{damped Hamiltonian} flow
\begin{equation} \label{hamil}
\begin{cases}
\dot x_t
= \nabla_p H(x_t, p_t), & x_t|_{t = 0} = x_0,\\
\dot p_t + \alpha_t p_t
= - \nabla_x H(x_t, p_t), & p_0 = 0.
\end{cases}
\end{equation}

\subsection{Nesterov Acceleration on  Probability Measures on Compact Sets}
In \cite{WL2022} (see \cite{CLTW2025,TLG2026} for
alternative approaches), Nesterov acceleration is
formulated on the space of smooth, strictly positive
probability densities on a closed, connected manifold $\Omega$.
In this subsection, we briefly recall this approach
to get a sense of the path we must follow for our
transformer setting on the space of smooth probability
densities on $\R^d$.
We focus on Otto--Wasserstein geometry from optimal
transport \cite{O2001} and Stein geometry associated
with Stein variational gradient descent
\cite{LW2016,DNS2023}; see also the references therein.

We begin by fixing notation for the space of smooth,
strictly positive probability densities on $\Omega$:
\[
\P^\infty(\Omega)
\coloneqq
\left\{
\rho\in C^\infty(\Omega):
\rho>0,\quad
\int_\Omega\rho(x)\,\d x=1
\right\}.
\]
This space carries a smooth Fr\'echet manifold structure
\cite{H1982,L1988}. The tangent space of $\P^\infty(\Omega)$ at a smooth probability density $\rho$ is given by 
\begin{equation} \label{eq:formal_tangent_space}
    T_\rho \P^{\infty}(\Omega) = C_0^\infty(\Omega) \coloneqq
    \{ \dot \rho  \in C^\infty (\Omega) : \int_\Omega \dot \rho \d x = 0\} . 
\end{equation} 
The cotangent space of $\P^\infty(\Omega)$ at a smooth probability density  $\rho$, to be precise its smooth dual, reads as
$
    T_\rho^* \P^{\infty} (\Omega)\coloneqq C^\infty (\Omega)/\R
$.
The class of a representative $\Phi$ modulo constants is denoted by
$[\Phi]$.

\paragraph{Nesterov acceleration in Otto--Wasserstein geometry.}
For all $\rho\in \P^\infty(\Omega)$, the \emph{Otto--Wasserstein mobility operator}
\begin{equation} \label{eq:LM}
    L_{\rho}^W:
    T_\rho^*\P^{\infty}(\Omega)
    \to
    T_\rho\P^{\infty}(\Omega),
    \qquad
    [\Phi]\mapsto
    -\nabla\cdot(\rho \nabla\Phi),
\end{equation}
is a linear isomorphism, and the corresponding bilinear form induces a well-defined \emph{Riemannian metric} on $T_\rho \P^\infty(\Omega)$, given by
$G_{\rho}^W \colon T_\rho \P^\infty(\Omega)  \times T_\rho\P^\infty(\Omega) \to \R$
with
\begin{equation}\label{eq:LM_metric_compact}
   G_{\rho}^W(\dot\rho,\dot\eta)
    \coloneqq
    \int_{\Omega}
    \Phi_{\dot\rho} \dot\eta\d x,
\end{equation}
where $\Phi_{\dot \rho}$ is uniquely defined for any given $\dot \rho$ via
\begin{equation}\label{helperL}
\dot\rho =  L_\rho^W([\Phi_{\dot \rho}]) = 
     - \nabla\cdot
    \left(
        \rho \nabla\Phi_{\dot\rho}
    \right), 
\end{equation}
 see \cite{L1988,O2001,V2008}. Note that the invertibility of $L_\rho^W$ is ensured as long as we work
with measures on compact, connected spaces. An important ingredient is that on these spaces the so-called Poincaré inequality holds.
In our transformer setting, we have to establish such a relation for a more general operator related to the softmax function
as well as for the whole space $\R^d$.

Let $0\leq t_0<T\leq\infty$.  For every accelerated evolution considered on
$[t_0,T)$, we assume that the damping coefficient satisfies
\begin{equation}\label{eq:assump_damping_factor}
     \alpha\in C\big([t_0,T);[0,\infty)\big).
\end{equation}
For schedules singular at the origin, such as $\alpha(t)=r/t$ \cite{AGR2000,APR2016} or
$\alpha(t)=r/t+m$, we always choose $t_0>0$ or translate time and use $\alpha(t_0+t)$ with $t_0>0$.
Then, instead of \eqref{problem}, we are looking for a minimizer of $\F \colon  \P^\infty (\Omega) \to \R$, i.e.,
\[
\min_{\rho \in \P^\infty(\Omega)} \F(\rho).
\]
The gradient flow of  $\mathcal{F}(\rho)$ in the probability density space  $(\P^\infty(\Omega), G^W)$ satisfies 
\begin{equation} \label{gradd_Omega}
\partial_t \rho_t = - L_{\rho_t}^W \Big[\frac{\delta \F(\rho_t)}{\delta \rho_t} \Big],
\end{equation}
where \(\frac{\delta\F}{\delta\rho}\) denotes the first variation of the functional \(\F\) at the density \(\rho\).
With the  \emph{Hamiltonian} defined by
\begin{equation} \label{hamilt_compact}
 \mathcal{H}(\rho,\Phi) \coloneqq \frac12 \int _{\Omega} \Phi L_\rho^W\Phi \d x + \F(\rho), \qquad \rho \in \P^\infty(\Omega),\,\Phi\in T_\rho^*\P^\infty(\Omega)
\end{equation}
the accelerated dynamics by the \emph{damped Hamiltonian} equations suggested in \cite{WL2022} satisfy 
\begin{equation} \label{sys_comp}
\begin{cases}
\partial_t \rho_t 
= \frac{\delta}{\delta \Phi_t}  \mathcal{H}(\rho_t,\Phi_t), & \rho_t|_{t = 0} = \rho_0\\
\partial_t \Phi_t
+ \alpha_t \Phi_t
= - \frac{\delta}{\delta \rho_t}  \mathcal{H}(\rho_t,\Phi_t), & \Phi_0 = 0,
\end{cases}
\end{equation}
where $\rho_0\in\P^\infty(\Omega)$ is a given initial density function, see also \cite{R2012,khesin2021geometric}.
The system \eqref{sys_comp} can be rewritten as
\begin{equation}\label{sys_1}
\begin{cases}
 \partial_t \rho_t= L_{\rho_t}^W   (\Phi_t), & \rho_t|_{t = 0}=\rho_0, 
 \\[1ex]
 \partial_t \Phi_t + \alpha_t \Phi_t = - \frac12 \frac{\delta}{\delta \rho_t}
 \left(\int_{\Omega} \Phi_t L_{\rho_t}^W(\Phi_t) \d x \right)  - \frac{\delta \F}{\delta \rho_t}, & \Phi_0 = 0,
 \end{cases}
\end{equation} 
and plugging in \eqref{helperL}, we obtain the \emph{Nesterov acceleration in the Otto--Wasserstein geometry}
\begin{equation} \label{eq:accelerated_W2_gradient_flow}
\begin{cases}
\partial_t \rho_t = - \nabla \cdot \left(\rho_t \nabla  \Phi_t  \right), & \rho_t|_{t = 0} = \rho_0, \\[1ex]
\partial_t \Phi_t + \alpha_t  \Phi_t = - \frac12 \| \nabla \Phi_t\|^2 - \frac{\delta \F(\rho_t)}{\delta \rho_t} , & \Phi_0 = 0.
\end{cases}    
\end{equation}
This PDE is related to damped Euler equations, whose long-time
behavior under constant damping was studied in \cite{CCT2019}
for suitable free energies, assuming the existence of global
energy-dissipating solutions. That work does not establish
general well-posedness. For a related kinetic formulation,
see \cite{M2026}.

To connect this construction with attention dynamics, we next
consider the mobility operators associated with the attention
kernels. For softmax attention, the operator in
\eqref{eq:softmax-attention-mobility} gives a modification of
Otto--Wasserstein geometry that remains closely related to it
on compact domains. Linear attention instead suggests a
Stein-type geometry through its bilinear kernel.
Here, an additional obstruction arises: even for symmetric,
bounded, integrally strictly positive definite kernels, the
Stein mobility operator need not be invertible between the
full formal cotangent and tangent spaces. The following
paragraph explains why the preceding Otto--Wasserstein
construction therefore requires further modifications in
the Stein setting, even on compact domains.

\paragraph{Nesterov acceleration in Stein geometry and its failure.}
In \cite{WL2022} and follow-up papers, it was assumed that the
\emph{Stein mobility operator}
\begin{equation}\label{eq:stein_mob_operator}
L_{\rho}^S:
    T_\rho^*\P^\infty(\Omega)
    \to
    T_\rho\P^\infty(\Omega),
    \qquad
    [\Phi]\mapsto
- \nabla \cdot \left(\rho \int_\Omega K(\cdot,y) \, \nabla \Phi(y) \, \rho(y) \d y\right)
\end{equation}
with a symmetric, bounded, integrally strictly positive definite kernel $K$ is
also a linear isomorphism, so that
$G_{\rho}^S \colon T_\rho \P^\infty(\Omega) \times T_\rho\P^\infty(\Omega) \to \R$
given by
\begin{equation}\label{eq:Stein_metric_compact}
G_{\rho}^S(\dot\rho,\dot\eta)
\coloneqq
\int_{\Omega} \Phi_{\dot\rho}(x) \dot\eta(x) \d x =
\int_{\Omega}\int_{\Omega}
K(x,y)\,
\big\langle \nabla\Phi_{\dot\rho}(x),
\nabla\Phi_{\dot\eta}(y) \big\rangle
\,\rho(x)\rho(y) \d y \d x
\end{equation}
with
\[
\dot\rho = L_\rho^S([\Phi_{\dot\rho}]),
\qquad
\dot\eta = L_\rho^S([\Phi_{\dot\eta}]).
\]
imposes a well-defined Riemannian metric on $T_\rho \P^\infty(\Omega)$.
Unfortunately, this is not true, since $L_{\rho}^S$ fails in general to be surjective, and a restriction of the formal tangent space is necessary to ensure surjectivity; see Appendix~{\ref{app:Non-surjectivity of the Stein operator on compact domains}}.

We close this paragraph by noting an additional geometric obstruction on general compact manifolds $\Omega$. The kernel in \eqref{eq:stein_mob_operator} has to do more than just weighting the gradients: it must map tangent vectors at the integration point to tangent vectors at the evaluation point, so that they can be averaged in a common tangent space. Thus, an intrinsic formulation requires an operator-valued kernel acting between the respective tangent spaces. A scalar kernel alone does not provide this identification, which is implicit in the Euclidean setting. For one possible way to overcome this obstruction, we refer to \cite{LZ2018}.

\subsection{Nesterov Acceleration and Attention Models on \texorpdfstring{$\R^d$}{Rd}}

This subsection connects the previously discussed interpretation of transformers as dynamical systems with Nesterov acceleration on $\R^d$, motivating the theory developed later in this paper. At a high level, each choice of a self-attention kernel in \eqref{eq:self_attn_kernel} corresponds to a choice of geometry. More precisely, the softmax attention kernel corresponds to softmax Otto--Wasserstein geometry, while the linear attention kernel corresponds to a geometry reminiscent of Stein geometry. We begin by stating at this point formally the analog of \eqref{eq:accelerated_W2_gradient_flow} in these geometries.
At a formal level, the self-attention kernel determines a mobility
operator, whose inverse on suitable tangent and cotangent spaces
induces the corresponding metric. This is the mechanism through
which the choice of an attention kernel determines a geometry.

\begin{figure}[ht]
\centering
\begin{tikzpicture}[
    box/.style={
        draw=blue!55!black, fill=blue!10, text=black,
        rounded corners=3pt, line width=0.7pt,
        minimum height=1.6cm, text width=3.3cm,
        align=center, inner sep=8pt, font=\small
    },
    arrow/.style={-{Stealth[length=2.2mm]}, line width=0.7pt},
    node distance=1.1cm
]
\node[box] (kernel)
    {\textbf{Attention kernel}\\[5pt]$K(x,y)$};

\node[box, right=of kernel, text width=3.9cm] (mobility)
    {\textbf{Mobility operator}\\[5pt]
     $L_\rho:T_\rho^*\mathcal P\to T_\rho\mathcal P$};

\node[box, right=of mobility] (metric)
    {\textbf{Riemannian metric}\\[5pt]$G_\rho$};

\draw[arrow] (kernel)
    -- node[above, font=\scriptsize] {induces} (mobility);
\draw[arrow] (mobility)
    -- node[above, font=\scriptsize] {invert} (metric);
\end{tikzpicture}
\caption{From attention to geometry: an attention kernel
gives rise to a mobility operator. The key step is to
identify suitable function spaces on which
this operator is invertible. Its inverse then defines
the metric underlying a rigorous analog of Otto calculus.}
\label{fig:attention_geometry}
\end{figure}
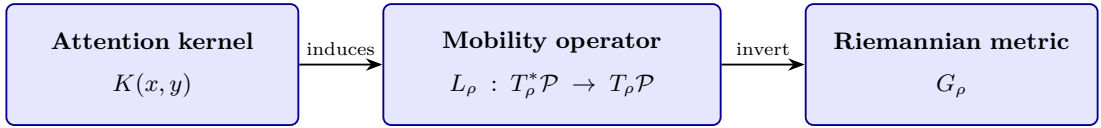

\paragraph{Accelerated softmax attention.}
Let $A=A^\top$ be invertible and let
$B=VA^{-1}=B^\top\succ0$. The softmax attention kernel induces
the \emph{formal local-mobility operator}
\begin{equation}\label{eq:softmax-attention-mobility}
L_\rho^{\SM}[\Phi]
\coloneqq
-\nabla\cdot\left(\frac{\rho}{Z_\rho}B\nabla\Phi\right),
\qquad
Z_\rho(x)\coloneqq
\int_{\R^d}\exp(y^\top Ax)\rho(y) \d y.
\end{equation}
As in the Otto--Wasserstein case, this formally induces a
\emph{Riemannian metric} through
\[
G_\rho^{\SM}\bigl(L_\rho^{\SM}[\Phi],L_\rho^{\SM}[\Psi]\bigr)
\coloneqq
\int_{\R^d}\Phi\,L_\rho^{\SM}[\Psi] \d x
=
\int_{\R^d}\frac{\rho}{Z_\rho}
\nabla\Phi^\top B\nabla\Psi \d x.
\]
For a functional $\F$, the associated \emph{Hamiltonian} is
\begin{equation}\label{eq:nesterov-softmax-hamiltonian}
\mathcal H^{\SM}(\rho,\Phi)
\coloneqq
\frac12 G_\rho^{\SM}
\bigl(L_\rho^{\SM}[\Phi],L_\rho^{\SM}[\Phi]\bigr)
+\F(\rho).
\end{equation}
The \emph{damped Hamiltonian} equations associated with
\eqref{eq:nesterov-softmax-hamiltonian} formally read
{\small
\begin{equation}\label{eq:nesterov-softmax-system}
\left\{
\begin{aligned}
&\partial_t\rho_t
= -\nabla\cdot\left(
\frac{\rho_t}{Z_{\rho_t}}B\nabla\Phi_t
\right),\\[1ex]
&\partial_t\Phi_t(x)+\alpha_t\Phi_t(x)
= -\frac{\nabla\Phi_t(x)^\top B\nabla\Phi_t(x)}
       {2Z_{\rho_t}(x)}
-\frac{\delta\F}{\delta\rho}(\rho_t)(x)
+\frac12\int_{\R^d}\exp(y^\top Ax)
\frac{\nabla\Phi_t(y)^\top B\nabla\Phi_t(y)}
     {Z_{\rho_t}(y)^2}
\,\rho_t(y) \d y.
\end{aligned}
\right.
\end{equation}
}
The second equation is understood modulo spatial constants.
The integral term accounts for the dependence of $Z_\rho$
on the density.

To make this construction rigorous, we must identify suitable
spaces of probability densities and corresponding tangent and
cotangent spaces on which the mobility operator is well-defined
and invertible. On $\R^d$, smooth positivity and finite second
moments alone do not suffice: the normalization $Z_\rho$ may be
infinite, and the required coercivity of the mobility operator
is not automatic. Section~\ref{sec3} develops the corresponding
functional framework using weighted Sobolev spaces and a
weighted Poincar\'e inequality for the mobility $\rho/Z_\rho$.

\paragraph{Accelerated linear attention.}
Let $A=A^\top\succ0$. The linear
attention kernel $K_A(x,y)=x^\top Ay$ induces the \emph{formal mobility operator}
\begin{equation}\label{eq:linear-attention-mobility}
L_\rho^{K_A}[\Phi]
\coloneqq
-\nabla_x\cdot\left(
\rho(x)\int_{\R^d}K_A(x,y)\nabla\Phi(y)\rho(y) \d y
\right).
\end{equation}
As in the Otto--Wasserstein case, $L_\rho^{K_A}$ formally gives rise to a
\emph{Riemannian metric} on the directions generated by the mobility operator:
\[
\begin{aligned}
G_\rho^{K_A}\bigl(L_\rho^{K_A}[\Phi],L_\rho^{K_A}[\Psi]\bigr)
&\coloneqq
\int_{\R^d}\Phi\,L_\rho^{K_A}[\Psi] \d x\\
&=
\int_{\R^d}\int_{\R^d}
K_A(x,y)\,\nabla\Phi(x)^\top\nabla\Psi(y)
\,\rho(x)\rho(y) \d y \d x.
\end{aligned}
\]

For a functional $\F$, the associated \emph{Hamiltonian} is
\begin{equation}\label{eq:nesterov-stein-hamiltonian}
\mathcal H^{K_A}(\rho,\Phi)
\coloneqq
\frac12 G_\rho^{K_A}
\bigl(L_\rho^{K_A}[\Phi],L_\rho^{K_A}[\Phi]\bigr)
+\F(\rho).
\end{equation}
The \emph{damped Hamiltonian} equations associated with
\eqref{eq:nesterov-stein-hamiltonian} formally read
\begin{equation}\label{eq:nesterov-stein-system}
\left\{
\begin{aligned}
&\partial_t\rho_t(x) = -
\nabla_x\cdot\left(
\rho_t(x)\int_{\R^d}
K_A(x,y)\nabla\Phi_t(y)\rho_t(y) \d y
\right),
\\[1ex]
&\partial_t\Phi_t(x)+\alpha_t\Phi_t(x) = -\int_{\R^d}K_A(x,y)\,
\nabla\Phi_t(x)^\top\nabla\Phi_t(y)\,
\rho_t(y) \d y-\frac{\delta\F}{\delta\rho}(\rho_t)(x),
\end{aligned}
\right.
\end{equation}
where the second equation is understood modulo spatial constants.

To make this construction rigorous, we must identify suitable spaces
of probability densities and corresponding tangent and cotangent spaces
on which the mobility operator is well-defined. Several obstacles arise: in particular, the bilinear kernel is not integrally strictly positive
definite, so the usual Stein framework based on this assumption does
not apply directly. The mobility operator has finite-dimensional range
and a nontrivial kernel beyond constants. Recovering an isomorphism
therefore requires restricting the tangent space to its range and
quotienting the cotangent space by its kernel. We develop this
construction in Section~\ref{sec:stein}.
\\

In summary, making the geometric interpretation of accelerated
linear and softmax attention rigorous involves several fundamental
obstacles, including:
\begin{itemize}
    \item \textbf{Function spaces and mobility operators.}
    Choosing appropriate function spaces, including the necessary
    restrictions and quotients, so that the mobility operators
    $L_\rho^{K_A}$ and $L_\rho^{\SM}$ become isomorphisms between
    suitable cotangent and tangent spaces and induce well-defined metrics.

    \item \textbf{Differentiability and variations.}
    Specifying a notion of differentiability that makes variations
    with respect to $\rho$ and $\Phi$ well-defined, since the chosen
    classes of probability densities on $\R^d$ have not been equipped
    with a smooth manifold structure.
\end{itemize}
\part{Geometric foundations}

\section{A Non-Compact Functional Framework for Otto Calculus} \label{sec3}
In this section, we develop a rigorous local-mobility
calculus for smooth probability densities on $\R^d$, encompassing
both Otto--Wasserstein geometry
from optimal transport \cite{V2003,V2008},
and softmax--Wasserstein geometry.
Even for the classical Otto calculus, smooth positivity and
finite second moments do not by themselves justify the formal
geometric construction on a non-compact domain. In particular,
the operator
\begin{equation}\label{eq:Otto_operator_again}
\Phi\longmapsto-\nabla\cdot(\rho\nabla\Phi)
\end{equation}
requires suitable domains before it can be inverted modulo
constants and used to define a metric or a gradient.

On a closed connected Riemannian manifold, smooth strictly
positive densities are bounded above and away from zero.
Weighted Poincar\'e inequalities \cite{BGL2014} and elliptic regularity
therefore provide the analytic tools needed for this
construction. On $\R^d$, these properties require additional
assumptions. Although the Wasserstein distance is well defined
on $\P_2(\R^d)$, membership in this space does not supply the
regular tangent and cotangent spaces \cite{AGS2008}, operator invertibility,
or notion of differentiation needed for the density calculus
used here.

We address these issues by specifying the function spaces
and analytic assumptions explicitly. After recalling the
relevant distinctions from infinite-dimensional Riemannian
geometry in Subsection~\ref{ss:background_inf_dim_geometry},
we introduce the local-mobility framework in
Subsection~\ref{subsec:Rigorous_Local_Mobility_Calculus}.
Its starting point is a weighted Sobolev space of potentials
and a weighted Poincar\'e inequality. This inequality makes
the centered potential space complete, as established in
Lemma~\ref{lem:LM_centered_energy_space}, and provides the
coercivity needed to solve the potential equation.

The main result in Theorem~\ref{prop:LM_operator_invertible}
shows that the local-mobility operator is a linear isomorphism
between the chosen regular cotangent and tangent spaces.
Its inverse defines the metric in
Corollary~\ref{prop:LM_metric_well_defined}.
Then Corollary~\ref{cor:LM_gradient} identifies the gradient of
a functional whenever its first variation belongs to the
regular cotangent space. The required notion of differentiation
is defined through regular first variations in
Appendix~\ref{app:diif_in_in_local_mobility_calculus}.

Finally, we specialize this construction to the
Otto--Wasserstein and softmax--Wasserstein geometry, which
correspond to different choices of mobility.
This gives explicit admissible density classes and regular
tangent and cotangent spaces for both geometries.
The resulting calculus provides the analytic framework
for the gradient and Hamiltonian formulations used later
in the paper.
All proofs are given in Appendix \ref{app:proofs}.

\subsection{Background on Infinite Dimensional Geometry}\label{ss:background_inf_dim_geometry}
We start by recalling the basic notation and definitions from infinite-dimensional Riemannian geometry used throughout the paper. For a more comprehensive introduction, see \cite{La99,maor2022riemannian,schmeding2022introduction}. 

Let \(\mathcal M\) be a smooth infinite-dimensional manifold. A \emph{Riemannian metric} on \(\mathcal M\) is a smooth map 
\[
    G:T\mathcal M\times_{\mathcal M}T\mathcal M\to\R
\]
such that, for every \(x\in\mathcal M\), the map
\[
    G_x:T_x\mathcal M\times T_x\mathcal M\to\R
\]
is an inner product on \(T_x\mathcal M\). The metric \(G\) is called \emph{strong} if, for every \(x\in\mathcal M\), the topology induced by \(G_x\) agrees with the given manifold topology on \(T_x\mathcal M\). It is called \emph{weak} if, for at least one \(x\in\mathcal M\), the topology induced by \(G_x\) is strictly weaker than the given manifold topology. 
We recall that the existence of a strong Riemannian metric is a very restrictive
condition in infinite dimensions. By \cite[Thm. 7.2.]{Bauer_2025}, a smooth infinite-dimensional manifold
\(\mathcal M\) admits a strong Riemannian metric if and only if
\(\mathcal M\) is a Hilbert manifold and the \emph{induced musical map}
\begin{equation}\label{eq:G_flat}
    G^\flat:T\mathcal M\to T^*\mathcal M,
    \qquad
    G^\flat_x(v)(w)=G_x(v,w),
\end{equation}
is a bundle isomorphism. Consequently, Fréchet manifolds, such as the space of
smooth probability densities \(\P^{\infty}(\mathbb \R^d)\), do not admit strong
Riemannian metrics in this sense.

The metrics considered in this paper therefore are weak Riemannian metrics.
Prominent examples include, besides the Otto--Wasserstein metric on spaces of probability
densities, the \(L^2\)-metric appearing in Arnold's geometric formulation of
hydrodynamics~\cite{A1966,O2001}. Weak Riemannian metrics that are not strong are a genuinely
infinite-dimensional phenomenon and lead to several difficulties that have no
finite-dimensional analog.
We highlight the distinction that is most important in our context.

\begin{remark}[Possible nonexistence of Riemannian gradients]\label{rmk:Possible nonexistence of Riemannian gradients}
    For a smooth
functional \(F:\mathcal M\to\R\), the differential
\(dF_x\in T_x^*\mathcal M\) is always defined. However, a Riemannian gradient  
\[\operatorname{grad}_G F(x) \in T_x\mathcal{M},\qquad x\in \mathcal{M}\] exists only if \(dF_x\) lies in the
range of the musical map \(G_x^\flat\) defined in~\eqref{eq:G_flat}. For strong
Riemannian metrics, this condition is automatically fulfilled, since \(G^\flat\) is a bundle
isomorphism. For weak Riemannian metrics, the map \(G_x^\flat\) is injective,
but not surjective, and hence not every smooth functional admits a Riemannian
gradient. In particular, formal expressions such as
\[
    (G_x^\flat)^{-1}dF_x\in T_x\mathcal{M}, \qquad x\in \mathcal{M}
\]
are meaningful only after verifying that \(dF_x\in\operatorname{Range}(G_x^\flat).\)
This range condition is one of the main analytic issues that should be addressed
when applying Otto-type calculus on spaces of smooth probability densities on non-compact domains.
\end{remark} 

We conclude this subsection by noting striking infinite-dimensional pathologies that have no finite-dimensional counterparts. One classical manifestation is the failure of the Hopf--Rinow theorem, first observed by Atkin~\cite{HopfrinowfalseAktkin} and later extended by Ekeland~\cite{HopfrinowfalseEkeland}. A second striking phenomenon is that the geodesic distance induced by a weak Riemannian or weak Finsler metric may vanish identically, even though the infinitesimal metric is pointwise non-degenerate. In symplectic geometry, Eliashberg and
Polterovich~\cite{EliashbergPolterovich1993} exhibited this phenomenon for \(L^p\)-analogs of Hofer's metric on the Hamiltonian diffeomorphism group. Michor and Mumford~\cite {MichorMumford2005} studied related vanishing-distance phenomena for weak Riemannian metrics.
Finally, while a Levi--Civita connection exists for strong Riemannian metrics, this may fail in the weak Riemannian setting~\cite{bauer2014homogeneous}.

\subsection{Weighted Local-Mobility Geometry on \texorpdfstring{$\mathcal{P}^\infty(\R^d)$}{Pinfty(Rd)}}
\label{subsec:Rigorous_Local_Mobility_Calculus}

This subsection develops a unified, rigorous framework that encompasses both Otto--Wasserstein and softmax--Wasserstein geometries for smooth probability densities on $\R^d$. On closed manifolds, these two geometries are norm-equivalent, whereas on non-compact domains such as $\R^d$, they differ significantly. Nevertheless, both belong to a more general class of geometries \cite{CLSS2010, DNS2009}, which we call local mobility geometries.
In the following, we introduce our general framework and then show that both settings fit within it.

\paragraph{General setting.}
For the rest of the paper, we set
$$
\P^{\infty}(\R^d) \coloneqq \{\rho \in C_{>0}^\infty(\R^d): \int_{\R^d}\rho(x) \d x = 1 \}.
$$
Further, let   
$
\mathfrak m_\rho  =
m_\rho(x) \d x,
$
and $B \in \R^{d \times d}$
a symmetric, positive definite matrix.
Here are our main model cases.
\begin{align} \label{set1}
\textbf{Otto--Wasserstein: }&  
 m_\rho=\rho \quad \text{and} \quad B=I, 
\\
\textbf{Softmax--Wasserstein: }&
m_\rho=\rho/Z_\rho 
, \;
Z_\rho(x) \coloneqq \int_{\R^d} \exp \left( z^{\top} A x\right) \rho(z) \d z
\quad \text{and} \quad
B\coloneqq VA^{-1}. \label{set2} 
\end{align}
Before moving on, we note that sigmoid--Wasserstein
and native $L^2$ geometry also fit within this framework. Now let
\begin{align*} 
    L^2( m_\rho)
    &\coloneqq
    \left\{
        \Phi:\R^d\to\R\text{ measurable}:
        \Vert \Phi\Vert_{L^2(m_\rho)}^2 \coloneqq \int_{\R^d}|\Phi|^2\,\d\mathfrak m_\rho<\infty
    \right\},
\\
    L^2( m_\rho,B;\R^d)
    &\coloneqq
    \left\{
        v:\R^d\to\R^d\text{ measurable}:
        \Vert v\Vert_{L^2(m_\rho,B;\R^d)}^2\coloneqq \int_{\R^d}v^\top Bv\,\d\mathfrak m_\rho<\infty
    \right\}.
\end{align*}
We consider the Hilbert space
\begin{equation}\label{eq:LM_H_1}
     H^1(m_\rho,B)
    \coloneqq
    \left\{
        \Phi\in L^2( m_\rho):
        \nabla\Phi
        \in L^2(m_\rho,B;\R^d) \right\}
\end{equation}
with distributional gradient and equipped with the norm
\begin{equation}\label{eq:norm_on_H_1_B_m_rho}
      \|\Phi\|_{H^1(m_\rho,B)}^2
    \coloneqq \Vert \Phi\Vert_{L^2( m_\rho)}^2+ \Vert \nabla \Phi\Vert_{L^2( m_\rho,B;\R^d)}^2=
    \int_{\R^d}|\Phi|^2\,\d\mathfrak m_\rho
    +
    \int_{\R^d}
    \nabla\Phi^\top B\nabla\Phi\,
    \d\mathfrak m_\rho.
\end{equation}
We further define the centered potential space by
\begin{equation} \label{eq:LM_centered_space}
    H^1_0(m_\rho,B)
    \coloneqq
    \left\{
        \Phi\in H^1(m_\rho,B):
        \int_{\R^d}\Phi\d\mathfrak m_\rho=0
    \right\}.
\end{equation}
Then
\begin{equation} \label{norm_h01}
    \|\Phi\|_{m_\rho,B}
    \coloneqq
    \left(
        \int_{\R^d}
        \nabla\Phi^\top B\nabla\Phi\,
        \d\mathfrak m_\rho
    \right)^{1/2}
\end{equation}
is a norm on $H^1_0(m_\rho, B)$, but the space is in general not complete with this norm. 
The missing analytic condition is a weighted Poincar\'e inequality.
We say that \(m_\rho\) satisfies the
\emph{local-mobility Poincar\'e inequality} (LM-PI), if there exists
\(C_\rho>0\)  such that
\begin{equation}\label{eq:LM_PI}
    \int_{\R^d}
    |\Phi|^2\d\mathfrak m_\rho
    \le
    C_\rho
    \int_{\R^d}
    \nabla\Phi^\top B\nabla\Phi\,
    \d\mathfrak m_\rho
    \quad
    \text{for all} \quad
    \Phi\in H^1_0(m_\rho,B). 
\end{equation}
With this notation at hand, we define the space of \emph{smooth local mobility probability densities} by
\begin{equation}\label{lmc}
    \P_{m_\rho,B}
    \coloneqq
        \left\{
        \rho\in\P^\infty(\R^d):
        \mathfrak m_\rho(\R^d) \in (0, \infty), m_\rho\in C^\infty_{>0}(\R^d) 
         \text{ and satisfies LM-PI }
        \eqref{eq:LM_PI}
        \right\}.
\end{equation}

Before proceeding, we note that the admissible densities
in both \eqref{set1} and \eqref{set2} have finite second
moments. Thus, each defines a probability measure in
$\P_2(\R^d)$.

In the Otto--Wasserstein case, this follows by applying
the Poincar\'e inequality \eqref{eq:LM_PI} to centered
truncations of the coordinate functions and passing to
the limit.
In the softmax--Wasserstein case, invertibility of $A$
and finiteness of $Z_\rho(x)$ for every $x\in\R^d$
imply
\[
\int_{\R^d}\exp(w^\top y)\rho(y)\,\d y<\infty
\qquad\text{for every }w\in\R^d.
\]
Consequently, all polynomial moments are finite.

We now establish the properties of the potential spaces
needed to invert the mobility operator.
The following lemma shows that the Poincar\'e inequality
\eqref{eq:LM_PI} makes the centered space
\eqref{eq:LM_centered_space} complete under the gradient
norm \eqref{norm_h01}, and that its elements can be
approximated by centered smooth test functions.

\begin{lemma}[Centered local-mobility energy space]\label{lem:LM_centered_energy_space}
For \(\rho\in \P_{m_\rho,B}\), it holds:
\begin{enumerate}
    \item \label{it:1_lem:LM_centered_energy_space}\(H^1_0(m_\rho,B)\) with norm 
    \eqref{norm_h01} is a Hilbert space; in particular, this norm is equivalent to
 the induced norm  \eqref{eq:norm_on_H_1_B_m_rho}.
\item \label{it:2_lem:LM_centered_energy_space}\(C_c^\infty(\R^d)\) is dense in \(H^1(m_\rho,B)\), and
the centered functions
\begin{equation} \label{centered}
    \psi
    -
    \frac{1}{\mathfrak m_\rho(\R^d)}
    \int_{\R^d}\psi\d\mathfrak m_\rho,
    \qquad
    \psi\in C_c^\infty(\R^d),
\end{equation}
are dense in \(H^1_0(m_\rho,B)\).
\end{enumerate}
\end{lemma}
\paragraph{Regular tangent and cotangent spaces.}

For \(\rho\in\P_{m_\rho,B}\), define the \emph{regular formal
tangent space} by
\begin{equation} \label{eq:LM_regular_tangent}
    T_\rho\P_{m_\rho,B}
    \coloneqq
    \left\{
        \dot\rho\in C^\infty(\R^d):
        \frac{\dot\rho}{m_\rho}
        \in L^2( m_\rho),
        \quad
        \int_{\R^d}\dot\rho\d x=0
    \right\}.
\end{equation}
In the Otto–Wasserstein setting, the condition \(\frac{\dot\rho}{m_\rho} \in L^2(m_\rho)\) simplifies to \(\frac{\dot\rho}{\rho} \in L^2(\rho)\), which can be thought of as a natural decay condition. It also arises in Fisher–Rao geometry on the space of probability densities on \(\mathbb{R}^d\), as recently investigated in \cite{maier2026informationgeometrysimplex,maier2026informationgeometryqroottransform}.
Next, for a smooth potential \(\Phi\), we define the \emph{local-mobility
operator} by
\begin{equation}
\label{eq:LM_operator_expression}
    L_{m_\rho,B}\Phi
    \coloneqq
    -\nabla\cdot
    \left(
        m_\rho B\nabla\Phi
    \right).
\end{equation}
In Otto's calculus, we need 
this operator to be well-defined and invertible. 
Next, we introduce the \emph{regular formal cotangent space} as a domain of this operator, so that it will be invertible, namely
\begin{equation}
\label{eq:LM_regular_cotangent}
    T_\rho^*\P_{m_\rho,B}
    \coloneqq
    \left\{
        [\Phi]\in
        \frac{
            C^\infty(\R^d)\cap H^1(m_\rho,B)
        }{\R}:
        \frac{ L_{m_\rho,B}\Phi}{m_\rho}
        \in L^2(m_\rho)
    \right\}.
\end{equation}
The usual tangent--cotangent pairing is
\begin{equation}
\label{eq:LM_tangent_cotangent_pairing}
    \left\langle[\Phi],\dot\rho\right\rangle
    \coloneqq
    \int_{\R^d}\Phi \dot\rho\d x.
\end{equation}
Indeed, by the Cauchy--Schwarz inequality, the pairing is well-defined because
\begin{equation}\label{eq:fin_of_cotangent_pairing}
    \left|
        \int_{\R^d}\Phi \dot\rho\d x
    \right|=
    \left|
        \int_{\R^d}
        \Phi\frac{\dot\rho}{m_\rho}\,
        \d\mathfrak m_\rho
    \right|
    \le
    \|\Phi\|_{L^2(\mathfrak m_\rho)}
    \left\|
        \frac{\dot\rho}{m_\rho}
    \right\|_{L^2(\mathfrak m_\rho)}.
\end{equation}
Moreover, the pairing is independent of the representative of \([\Phi]\), since  \(\int_{\R^d}\dot\rho\d x=0\).

\paragraph{Invertibility of the local mobility operator and mobility metric.}
Next, we check that the local mobility operator \eqref{eq:LM_operator_expression} is a bijection from the regular cotangent space~\eqref{eq:LM_regular_cotangent} to the regular tangent space~\eqref{eq:LM_regular_tangent}.

\begin{theorem}[Invertibility of the local-mobility operator]
\label{prop:LM_operator_invertible}
Let \(\rho\in\P_{m_\rho,B}\). Then it holds:
\begin{enumerate}
\item
\label{it:1_prop:LM_operator_invertible}
For every
\([\Phi]\in T_\rho^*\P_{m_\rho,B}\) and every
\(\Psi\in H^1(m_\rho,B)\), we have
\begin{equation}
\label{eq:LM_weak_integration_by_parts}
    \int_{\R^d}
    (L_{m_\rho,B}\Phi)\Psi\d x
    =
    \int_{\R^d}
    \nabla\Phi^\top B\nabla\Psi\,
    \d \mathfrak m_\rho.
\end{equation}
\item
\label{it:2_prop:LM_operator_invertible}
The local-mobility operator
\begin{equation}
\label{eq:LM_operator_map}
    L_{m_\rho,B}:
    T_\rho^*\P_{m_\rho,B}
    \to
    T_\rho\P_{m_\rho,B},
    \qquad
    [\Phi]\mapsto
    -\nabla\cdot(m_\rho B\nabla\Phi),
\end{equation}
is a linear isomorphism. \end{enumerate}
\end{theorem}

As a direct consequence of Theorem \ref{prop:LM_operator_invertible} we can define 
the local-mobility metric. 

\begin{cor}[Well-definedness of the local-mobility metric]
\label{prop:LM_metric_well_defined}
For every \(\rho\in\P_{m_\rho,B}\), the following \emph{local-mobility metric} is well-defined:
\begin{equation}\label{eq:LM_metric}
   G_{m_\rho,B}:T_\rho \P_{m_\rho,B} \times T_\rho\P_{m_\rho,B} \to \R,\qquad 
    G_{m_\rho,B}(\dot\rho,\dot\rho)
    \coloneqq
    \int_{\R^d}
    \Phi_{\dot\rho} \dot\rho\d x,
\end{equation}
 where \(
    [\Phi_{\dot\rho}]
    \in T_\rho^*\P_{m_\rho,B}
\)
is determined by
\begin{equation}
\label{eq:LM_div_dot_rho_metric}
    \dot\rho
    +
    \nabla\cdot
    \left(
        m_\rho B\nabla\Phi_{\dot\rho}
    \right)
    =
    0.
\end{equation}
\end{cor}

\begin{remark}[Equivalent energy and bilinear forms]
\label{rem:LM_equivalent_metric_forms}
For all \(\dot\rho,\dot\eta
    \in T_\rho\P_{m_\rho,B},\)
the bilinear local-mobility metric can be written by \eqref{eq:LM_div_dot_rho_metric} and the polarization formula as 
\begin{equation}
\label{eq:LM_metric_bilinear}
\begin{aligned}
    G_{m_\rho,B}(\dot\rho,\dot\eta)=
    \int_{\R^d}
    \Phi_{\dot\rho} \dot\eta\d x
    =
    \int_{\R^d}
    \nabla\Phi_{\dot\rho}^{\top} 
    B\nabla\Phi_{\dot\eta}\,
    \d \mathfrak m_\rho.
\end{aligned}
\end{equation}
\end{remark}

\paragraph{First variations and the local-mobility gradient.}
\label{subsec:LM_first_variation_gradient}
Before stating the following corollary, we want to stress that
\(\P_{m_\rho,B}\) has not been equipped with a differentiable
structure, so the notion of differentiability requires a separate
definition. In Appendix~\ref{app:diif_in_in_local_mobility_calculus},
we provide such a definition in terms of local regular first
variations. It agrees with ordinary Fr\'echet differentiation
in the Otto--Wasserstein geometry on smooth probability densities
on a closed manifold, whenever the differential has a smooth
representative. The reader may therefore think of ordinary
differentiation with a smooth first-variation representative in
what follows, while \(D\F(\rho)\) denotes the associated regular
extension defined in Appendix~\ref{app:diif_in_in_local_mobility_calculus}.
Finally, we provide a criterion for the corresponding
local-mobility gradient to be a regular tangent vector.

\begin{cor}[Local-mobility gradient]\label{cor:LM_gradient}
Let \( \F \colon 
    \P_{m_\rho,B}
    \to
    \R\) have a first variation satisfying
\begin{equation}
\label{eq:LM_first_variation}
    D\F(\rho)[\dot\eta]
    =
    \int_{\R^d}
    \frac{\delta\F}{\delta\rho}\,
    \dot\eta\d x
    \quad
    \text{for all} \quad
    \dot\eta\in
    T_\rho\P_{m_\rho,B}, 
\end{equation}
and so that \(\left[
        \frac{\delta\F}{\delta\rho}
    \right]
    \in
    T_\rho^*\P_{m_\rho,B}.\) Then the gradient of \(\F\) with respect to the
local-mobility metric \(G_{m_\rho,B}\) is
\begin{equation}
\label{eq:LM_gradient}
    \operatorname{grad}_{G_{m_\rho,B}}
    \F(\rho)
    =
    L_{m_\rho,B}
    \left[
        \frac{\delta\F}{\delta\rho}
    \right]
    =
    -\nabla\cdot
    \left(
        m_\rho B
        \nabla\frac{\delta\F}{\delta\rho}
    \right)
    \in
    T_\rho\P_{m_\rho,B}.
\end{equation}
\end{cor}

Since every tangent vector satisfies \(   \int_{\R^d}\dot\rho\d x=0,\) the first variation \( \frac{\delta\F}{\delta\rho}\) is defined only modulo constants.

Finally, we consider the model cases from \eqref{set1} and \eqref{set2}.

\paragraph{Otto--Wasserstein.} \label{para:Otto--Wasserstein-geometry}
For \(\d\mathfrak m_\rho=\rho\d x\) and \( B=I_d\), we obtain
\[
         H^1(\rho) \coloneqq H^1(m_\rho,I_d)
               =
        \left\{
        \Phi\in L^2(\rho):
        \nabla\Phi\in L^2(\rho;\R^d)
        \right\}.
\]
Then LM-PI \eqref{eq:LM_PI} becomes the usual \emph{\(\rho\)-weighted
Poincar\'e inequality}: there exists \(C_\rho>0\) such that
\begin{equation}
\label{eq:PI_1}
    \int_{\R^d}|\Phi|^2\rho\d x
    \le
    C_\rho
    \int_{\R^d}|\nabla\Phi|^2\rho\d x
    \qquad
    \text{for all }\Phi\in H_0^1(\rho),
\end{equation}
and the space \eqref{lmc} reads as 
\[
    \P_{W}
    \coloneqq 
    \left\{
    \rho\in\P^\infty(\R^d):
        \rho \text{ satisfies }
        \eqref{eq:PI_1}
    \right\}.
\]
Typical measures in this space are Gaussians and more general strongly log-concave measures. 
Moreover, \(\P_W\) is convex and hence path-connected in the topology induced by \(C^\infty(\R^d)\).
Its regular tangent space is given by
\[
    T_\rho \P_{W}
      =
    \left\{
        \dot\rho\in C^\infty(\R^d):
        \frac{\dot\rho}{\rho}\in L^2(\rho),
        \quad
        \int_{\R^d}\dot\rho\d x=0
    \right\}.
\]
The local mobility operator~\eqref{eq:LM_operator_expression} 
becomes 
\begin{equation} \label{eq:Otto_operator}
    L_\rho^W \Phi
    =
    -\nabla\cdot(\rho\nabla\Phi),
\end{equation}
and the cotangent space is
\[
    T_\rho^* \P_{W}
      =
    \left\{
        [\Phi]\in
        \frac{
            C^\infty(\R^d)\cap H^1(\rho)
        }{\R}:
        \frac{L_\rho^W \Phi}{\rho}
        \in L^2(\rho)
    \right\}.
    \]
Finally, the local mobility metric reads as
\[
G_\rho^W\colon T_\rho \P_{W} \times T_\rho \P_{W} \to \R, \quad
    G_\rho^W(\dot\rho,\dot\eta)
    \coloneqq
    \int_{\R^d}
    \rho\,
    \nabla\Phi_{\dot\rho}\cdot
    \nabla\Phi_{\dot\eta}\d x,
\]
where $\Phi_{\dot\rho},\Phi_{\dot\eta}$ are defined with the potential equation 
\begin{equation} 
    \dot\rho
    +
    \nabla\cdot
    \left(
        \rho\nabla\Phi_{\dot\rho}
    \right)
    =
    0.
\end{equation}

    These regular tangent and cotangent spaces make explicit the regularity assumptions tacitly made in \cite[Subsec.~4.2]{MPR2014} and \cite[Subsec.~4.1]{ADPZ2013}.

\paragraph{Softmax--Wasserstein.} \label{paragraph:Softmax-Wasserstein-geometry}
Softmax gradient flows have recently gained interest because they have applications in the mathematical foundations of transformers.
Let
\begin{equation}\label{eq:softmax_B}
     A=A^\top \in \R^{d\times d} \text{ be invertible and }
                B\coloneqq VA^{-1}=(VA^{-1})^\top\succ0 .
\end{equation}
For \(\rho\in\P_2(\R^d) \cap \P^\infty(\R^d)\), we define the  softmax mobility measure by
\begin{equation}\label{eq:mobility_measure}
         \d\mathfrak m_\rho(x) \coloneqq  \frac{\rho(x)}{Z_\rho(x)}\d x ,
         \qquad Z_\rho(x)
        \coloneqq 
        \int_{\R^d}\exp(z^\top A x) \rho(z)\d z.
\end{equation}
This gives rise to the weighted Sobolev space
\[
       H^1(m_\rho,B) = 
        \left\{
        \Phi\in L^2(m_\rho):
        \nabla\Phi\in L^2(m_\rho,VA^{-1};\R^d)
        \right\}.
\]
We say that \(m_\rho \coloneqq \frac{\rho(x)}{Z_\rho(x)}\) satisfies the
\emph{local-mobility Poincar\'e inequality}, if there exists
\(C_\rho>0\) such that
\begin{equation}\label{eq:SM-PI}
    \Vert\Phi\Vert_{{L^2\left(\frac{\rho(x)}{Z_\rho(x)}\right)}}^{2} \leq C_\rho \Vert \nabla \Phi\Vert_{L^2\left(
    \frac{\rho(x)}{Z_\rho(x)},
    VA^{-1};
    \mathbb{R}^d
\right)}^2 \quad \text{for all} \quad
    \Phi\in H_0^1(m_\rho,B).
\end{equation}
With this notation, we can define, as a special case of \eqref{lmc}, the corresponding smooth softmax--Wasserstein class by
\begin{equation}
\label{eq:LM_softmax_class}
    \P_{\SM} \coloneqq 
    \left\{
        \rho\in \P^\infty(\R^d):
        Z_\rho\in C^\infty_{>0}(\R^d),
        \int_{\R^d}\frac{\rho(x)}{Z_\rho(x)}\d x\in (0,\infty),
        m_\rho \text{ satisfies }\eqref{eq:SM-PI}
    \right\}.
\end{equation}
Before we define the associated regular tangent space, we observe that \[
    \left\|
        \frac{\dot\rho}{m_\rho}
    \right\|_{L^2(m_\rho)}^2
    =
    \int_{\R^d}
    \left|
        \frac{Z_\rho\dot\rho}{\rho}
    \right|^2
    \frac{\rho}{Z_\rho}\d x
    =
    \int_{\R^d}
    Z_\rho
    \frac{|\dot\rho|^2}{\rho}\d x.
\]
So the regular local mobility tangent space simplifies in the softmax setting to
\begin{equation}
\label{eq:LM_softmax_tangent}
    T_\rho\P_{\SM}
    =
    \left\{
        \dot\rho\in C^\infty(\R^d):
        \int_{\R^d}\dot\rho\d x=0,
        \quad
        \int_{\R^d}
        Z_\rho
        \frac{|\dot\rho|^2}{\rho}\d x
        <\infty
    \right\},
\end{equation}
Next, we define the softmax operator as
\begin{equation}
\label{eq:LM_softmax_operator}
    L_\rho^{\SM}\Phi
    =
    -\nabla\cdot
    \left(
        \frac{\rho}{Z_\rho}
        VA^{-1}\nabla\Phi
    \right), 
\end{equation}
and the regular softmax cotangent space is
\begin{equation}
\label{eq:LM_softmax_cotangent}
\begin{aligned}
    T_\rho^*\P_{\SM}
    =
    \left\{
        [\Phi]\in
        \frac{
            C^\infty(\R^d)\cap H^1(m_\rho,B)
        }{\R}:
        \int_{\R^d}
        \frac{Z_\rho}{\rho}
        |L_\rho^{\SM}\Phi|^2\d x
        <\infty
    \right\}.
\end{aligned}
\end{equation}
Finally, the \emph{local-mobility metric} at $\rho \in \P_{\SM}$ is given by
\begin{equation}\label{eq:SM_metric}
   G_\rho^{\SM} \colon T_\rho \P_{\SM}\times T_\rho\P_{\SM}\to \R,\quad 
    G_\rho^{\SM}(\dot\rho,\dot\rho)
    \coloneqq
    \int_{\R^d}
    \Phi_{\dot\rho} \dot\rho\d x,
\end{equation}
 where \(
    [\Phi_{\dot\rho}]
    \in T_\rho^*\P_{\SM}
\)
is determined by
\begin{equation}
\label{eq:SM_div_dot_rho_metric}
    \dot\rho
    +
        \nabla\cdot
    \left(
        \frac{\rho}{Z_\rho}
        VA^{-1}\nabla\Phi_{\dot\rho}
    \right)
    =
    0.
\end{equation}

We close this subsection by highlighting a fundamental difference
between softmax--Wasserstein and Otto--Wasserstein geometry on
probability densities on \(\R^d\), in contrast to the more commonly
considered setting of compact domains.\begin{remark}[Comparison of Otto and softmax geometry on non-compact domains]
On compact domains, \(Z_\rho\) is bounded above and away from zero,
so the Otto--Wasserstein and softmax--Wasserstein metrics induce
equivalent tangent norms at each fixed density.
On \(\R^d\), however, these norms need not be equivalent:
\(Z_\rho\) can grow without bound, even for Gaussian densities,
so the mobilities \(\rho\) and \(\rho/Z_\rho\) are no longer
uniformly comparable.

Moreover, there are densities in \(\P_W\) that do not belong to
\(\P_{\SM}\).
For example, in dimension one with \(A=B=1\), consider the smooth
probability density
\[
    \rho(x)\propto \exp\!\bigl(-\sqrt{1+x^2}\bigr).
\]
Since this density is uniformly comparable to the Laplace density,
it satisfies a Poincar\'e inequality and therefore belongs to \(\P_W\).
However, its exponential tails imply
\[
    Z_\rho(x)=\int_{\R}\exp(zx)\rho(z) \d z=\infty
    \qquad\text{whenever }|x|\geq1,
\]
so \(\rho\notin\P_{\SM}\).
\end{remark} 

\section{Stein Calculus for Bilinear Kernels} \label{sec:stein}
To interpret linear-attention dynamics geometrically, we
would like to associate a Stein-type geometry with the
bilinear kernel $K_A(x,y)=x^\top Ay$, where $A=A^\top\succ0$.
However, this kernel is not integrally strictly positive
definite, so the usual Stein framework based on this
assumption does not apply directly. Its associated
Stein--Onsager operator has finite-dimensional range and
an infinite-dimensional kernel, even after identifying
potentials that differ by a constant, see
Lemma~\ref{lem:kernel}.
Consequently, it cannot be inverted on the full regular
tangent and cotangent spaces.

We overcome this obstruction by introducing a
\emph{quotient Stein geometry}. We restrict the tangent
space to the density variations generated by the mobility
operator and identify cotangent elements that generate
the same variation. On these effective spaces,
Theorem~\ref{prop:isomorphism} establishes an induced
Stein--Onsager isomorphism. Its inverse defines a
positive definite metric on the effective tangent space,
as shown in
Corollary~\ref{cor:stein_metric_well_defined}.

We begin by specifying an admissible density class and
the regular tangent--cotangent pairing. Then we establish
the well-definedness of the Stein--Onsager operator and
its Green identity in Lemma~\ref{lem:well-defined},
identify its kernel, and carry out the quotient
construction. Finally, Corollary~\ref{lem:stein-gradient}
identifies the gradient of a functional with an admissible
regular first variation. The corresponding notion of
differentiation is introduced in
Appendix~\ref{app:diff_stein_geom}.
This calculus provides the geometric framework for the
linear-attention dynamics studied in
Section~\ref{sec:lin-accelerated-linear-attention}.
All proofs are given in Appendix \ref{app:5}.

\paragraph{General setting.}
Throughout this section, let $A=A^\top\succ0$ and
$K_A(x,y)\coloneqq x^\top Ay$.  We define the space of \emph{smooth bilinear Stein probability densities} by
\begin{equation}\label{eq:density-class}
\P_{K_A}
\coloneqq
\left\{\rho\in
\P^{\infty}(\R^d):
\quad
\int_{\R^d}|x|^2\,\rho\d x<\infty,
\quad
\int_{\R^d}|x|^2|\nabla\log\rho|^2\,\rho\d x<\infty
\right\}.
\end{equation}
First we observe that, as $\P^{\infty}(\R^d)$, the set $\P_{K_A}$ is convex.
As we shall see, a rigorous extension of the local-mobility calculus developed in Section~\ref{sec3} to the Stein setting requires two additional assumptions on the probability density: a finite second moment and the logarithmic condition. 
\paragraph{Regular tangent and cotangent spaces.}
For $\rho\in\P_{K_A}$, we define the regular formal tangent space by
\begin{equation}\label{eq:tangent}
  T_\rho\P_{K_A}
  \coloneqq
  \left\{
  \dot\rho\in C^\infty(\R^d):
  \frac{\dot\rho}{\rho}\in L^2(\rho),
  \quad
  \int_{\R^d}\dot\rho \d x=0
\right\}.
\end{equation}
Next, we define the corresponding regular formal cotangent space by
\begin{equation}\label{eq:cotangent}
  T_\rho^*\P_{K_A} \coloneqq
  \frac{C^\infty(\R^d)\cap H^1(\rho)}{\R}.
\end{equation}
For $[\Phi]\in T_\rho^*\P_{K_A}$ and
$\dot\rho\in T_\rho\P_{K_A}$, the usual tangent--cotangent
pairing is
\begin{equation}\label{eq:pairing}
  \big\langle[\Phi],\dot\rho\big\rangle_\rho
  \coloneqq \int_{\R^d}\Phi\dot\rho \d x
  =\int_{\R^d}\Phi\frac{\dot\rho}{\rho}\,\rho \d x.
\end{equation}
This pairing is well-defined by Cauchy--Schwarz in combination with the defining conditions of the tangent and cotangent spaces in \eqref{eq:tangent} and \eqref{eq:cotangent}:
\begin{equation}\label{eq:pairing-bound}
\left|
  \big\langle[\Phi],\dot\rho\big\rangle_\rho
\right|
\leq
\lVert\Phi\rVert_{L^2(\rho)}
\left\lVert\frac{\dot\rho}{\rho}\right\rVert_{L^2(\rho)}.
\end{equation}
Moreover, replacing $\Phi$ by $\Phi+c$ changes the pairing by
$c\int_{\R^d}\dot\rho \d x=0$.  

\paragraph{Stein--Onsager operator.}
To define the analog of the local-mobility operator in this framework, we first introduce the matrix-moment operator on representatives by
\begin{equation}\label{eq:moment-map}
\begin{aligned}
  M_\rho:C^\infty(\R^d)\cap H^1(\rho)
  &\longrightarrow\R^{d\times d},\\
  \Phi&\longmapsto
  M_\rho(\Phi)\coloneqq \int_{\R^d}x\,\nabla\Phi(x)^\top \d\rho(x).
\end{aligned}
\end{equation}
By the Cauchy--Schwarz inequality, we can check that $M_\rho(\Phi)$ is well-defined for every element of its domain, and  using $M_\rho(\Phi+c)=M_\rho(\Phi)$ for every $c\in\R$, we see that the induced operator \begin{equation}\label{eq:quotient-moment-map}
  \overline M_\rho:T_\rho^*\P_{K_A}
  \longrightarrow\R^{d\times d},
  \qquad
  \overline M_\rho[\Phi]\coloneqq M_\rho(\Phi).
\end{equation} 
is independent of the representative.
For the bilinear kernel $K_A$, the kernel velocity factors through exactly this
matrix moment, that is,
\begin{equation}\label{eq:kernel-velocity}
  \mathcal T_{K_A,\rho}(\nabla\Phi)(x)
  \coloneqq \int_{\R^d}K_A(x,y)\nabla\Phi(y)\,\rho \d y
  =M_\rho(\Phi)^\top Ax.
\end{equation}
Now we define, for $\rho\in\P_{K_A}$,  the \emph{Stein--Onsager operator} by
\begin{equation} \label{eq:onsager}
      L_\rho^{K_A} \colon T_\rho^*\P_{K_A}
\longrightarrow T_\rho\P_{K_A},\quad
  L_\rho^{K_A}[\Phi]\coloneqq -\nabla\!\cdot\!\left(
  \rho\,M_\rho(\Phi)^\top Ax
  \right).
\end{equation}
Next, we check that this operator is indeed well-defined and provide a useful correspondence between the matrix-moment operator and the Stein--Onsager operator.

\begin{lemma}[Well-definedness and Green identity]\label{lem:well-defined} For each $\rho\in \P_{K_A}$, the following holds true:
\begin{enumerate}
    \item\label{it:1_lem_well_def}\emph{Well-definedness:} The Stein--Onsager operator $L_\rho^{K_A}$ in \eqref{eq:onsager} is a
well-defined linear operator.
\item \label{it:2_lem_well_def} \emph{Green identity:} For all
$[\Phi],[\Psi]\in T_\rho^*\P_{K_A}$:
\begin{equation}\label{eq:green-identity}
  \big\langle[\Phi],L_\rho^{K_A}[\Psi]\big\rangle_\rho
  =\tr\!\left( M_\rho(\Phi)^\top A M_\rho(\Psi) \right).
\end{equation}
\end{enumerate}
\end{lemma}

Next, we observe a feature absent in the local-mobility geometry framework in Section \ref{sec3}: a nonzero kernel of the operator $L^{K_A}$.

\begin{lemma}[Kernel of the Stein--Onsager operator]\label{lem:kernel}
For every $\rho\in\P_{K_A}$, the kernel of the Stein--Onsager
operator $L^{K_A}$ in \eqref{eq:onsager} is precisely the space of zero-moment cotangent classes, that is,
\begin{equation}\label{eq:operator-kernel}
  \ker L_\rho^{K_A}
  =\left\{[\Phi]\in T_\rho^*\P_{K_A}:M_\rho(\Phi)=0\right\}.
\end{equation}
In particular, $\ker L_\rho^{K_A}$ is infinite-dimensional.
\end{lemma}


We have seen in Section \ref{sec3} that an essential ingredient for making the local-mobility calculus rigorous is ensuring that the local-mobility operator is an isomorphism from the cotangent space to the tangent space. By Lemma~\ref{lem:kernel}, this cannot hold for the Stein geometry corresponding to the bilinear kernel. To work around this issue, we restrict ourselves to smaller subspaces of the tangent and cotangent spaces:
\begin{equation}\label{eq:effective-spaces}
  T_\rho^{\mathrm{eff}}\P_{K_A}\coloneqq\ran L_\rho^{K_A}
  \subset T_\rho\P_{K_A},
  \qquad
  T_\rho^{*,\mathrm{eff}}\P_{K_A}
  \coloneqq\frac{T_\rho^*\P_{K_A}}{\ker L_\rho^{K_A}}.
\end{equation}

Before moving on, we first want to point out that the dimension of $T_\rho^{\mathrm{eff}}$ varies with $\rho$. Furthermore, we comment on \eqref{eq:effective-spaces} through the lens of geometry: $T_\rho^{\eff}$ contains exactly the directions generated by the
bilinear Stein kernel, while $T_\rho^{*,\eff}$ identifies precisely those
regular cotangent classes that generate the same direction. Next, we see that on these spaces the Stein--Onsager operator induces an isomorphism.

\begin{theorem}[Effective Stein--Onsager isomorphism]
\label{prop:isomorphism}
For all $\rho\in \P_{K_A}$, the operator $L_\rho^{K_A}$  in~\eqref{eq:onsager} induces a linear isomorphism
\begin{equation}\label{eq:induced-isomorphism}
  \overline L_\rho^{K_A}\colon T_\rho^{*,\mathrm{eff}}\P_{K_A}
  \longrightarrow T_\rho^{\mathrm{eff}}\P_{K_A},
  \qquad
  \overline L_\rho^{K_A}\big([[\Phi]]\big)
  \coloneqq L_\rho^{K_A}[\Phi].
\end{equation}
Moreover, $T_\rho^{*,\eff}\cong\ran M_\rho$ via
$[[\Phi]]\mapsto M_\rho(\Phi)$.
\end{theorem}

\paragraph{Stein metric for the bilinear kernel and Stein gradients.}
Next, as in the local-mobility framework, we use Theorem~\ref{prop:isomorphism} to define the Stein metric associated with the bilinear kernel and show it is well-defined. First, for each $\rho\in\P_{K_A}$, we define the \emph{effective Stein metric} by
\begin{equation}\label{eq:metric-potential}
  G_\rho^{K_A}:
  T_\rho^{\mathrm{eff}}\P_{K_A}\times T_\rho^{\mathrm{eff}}\P_{K_A}\longrightarrow\R,
  \qquad
  G_\rho^{K_A}(\dot\rho,\dot\rho)
  \coloneqq \int_{\R^d}\Phi_{\dot\rho} \dot\rho \d x,
\end{equation}
where the unique effective cotangent class
$[[\Phi_{\dot\rho}]]\in T_\rho^{*,\mathrm{eff}}\P_{K_A}$ is determined by
\begin{equation}\label{eq:metric-potential-equation}
  \dot\rho
  +\nabla\!\cdot\!\left(
    \rho M_\rho(\Phi_{\dot\rho})^\top Ax
  \right)=0.
\end{equation}
Next, we obtain the following corollary
as a direct consequence of Lemma ~\ref{lem:well-defined}, Lemma~\ref{lem:kernel} and Theorem~\ref{prop:isomorphism}. 

\begin{cor}[Expression for the Stein metric]\label{cor:stein_metric_well_defined}
    For all $\rho\in \P_{K_A}$, the following holds true:
    \begin{enumerate}
        \item\label{it:1_stein_metric_well_defined} The metric $G_\rho^{K_A}$ in \eqref{eq:metric-potential} is a well-defined, positive-definite quadratic form on $T_\rho^{\mathrm{eff}}\P_{K_A}$.
        \item \label{it:2_stein_metric_well_defined}
        For all $[\Phi]\in T_\rho^*\P_{K_A}$, we have
        \begin{equation}\label{eq:metric}
            G_\rho^{K_A}\bigl(L_\rho^{K_A}[\Phi],L_\rho^{K_A}[\Phi]\bigr)
            =\tr\!\left(M_\rho(\Phi)^\top A M_\rho(\Phi)\right).
        \end{equation}
    \end{enumerate}
\end{cor}
We want to stress that
\(\P_{K_A}\) has not been equipped with a differentiable
structure, so the notion of differentiability requires a separate
definition. In Appendix~\ref{app:diff_stein_geom}, we provide such
a definition in terms of local regular first variations and their
associated regular extensions. The reader may think of ordinary
differentiation with a regular first-variation representative in
what follows, while \(D\F(\rho)\) denotes the associated regular
extension defined in the appendix.

Finally, following the argument of
Corollary~\ref{cor:LM_gradient}, we provide a criterion for the
gradient of a functional with respect to the metric
in~\eqref{eq:metric} to be a regular tangent vector.

\begin{cor}[The Stein gradient]\label{lem:stein-gradient}
Let \( \F \colon \P_{K_A}\longrightarrow\R\)
have a first variation satisfying
\begin{equation}\label{eq:stein-first-variation}
  D\F(\rho)[\dot\eta]
  =
  \int_{\R^d}
  \frac{\delta\F}{\delta\rho}\,
  \dot\eta \d x
  \qquad
  \text{for all }
  \dot\eta\in T_\rho^{\mathrm{eff}}\P_{K_A},
\end{equation}
and so that \(\left[\left[
    \frac{\delta\F}{\delta\rho}
  \right]\right]
  \in T_\rho^{*,\mathrm{eff}}\P_{K_A}\). Then, the gradient with respect to the Stein metric $G^{K_A}$ in \eqref{eq:metric} is given by
\begin{equation}\label{eq:stein-gradient}
  \operatorname{grad}_{G^{K_A}}\F(\rho)
  =
  L_\rho^{K_A}
  \left[
    \frac{\delta\F}{\delta\rho}
  \right]
  =  x \longrightarrow
  -\nabla\!\cdot\!\left(
    \rho\,
    M_\rho\!\left(
      \frac{\delta\F}{\delta\rho}
    \right)^\top
    Ax
  \right)
  \in T_\rho^{\mathrm{eff}}\P_{K_A}.
\end{equation}
\end{cor}

\part{Geometry and Dynamics of Attention}

\section{Softmax Attention: Geometry and Particle Dynamics}
\label{sec:softmax-regular-dynamics}
This section aims to provide a rigorous approach to softmax attention.

In Subsection \ref{subsec:sm-ordinary-gradient}, we give a gradient-flow interpretation of the transformer PDE.
We characterize the densities at which the interaction
energy admits a regular first variation in
Proposition~\ref{prop:sm-regular-variation}.
Corollary~\ref{thm:sm-gradient} then identifies the PDE
as the corresponding regular gradient equation.
Example~\ref{ex:sm-infinite-energy} illustrates why
pointwise finiteness of the attention velocity alone
does not suffice for this interpretation.
The resulting gradient identity holds at each admissible
density. 

Then, in Subsection \ref{subsec:sm-regular-hamiltonian} 
we identify conditions
under which the accelerated softmax system is pointwise
well-defined. We then restrict the domain further to
ensure that its right-hand sides arise from regular
derivatives of the Hamiltonian,
giving the PDE a rigorous interpretation as a damped
Hamiltonian system.
We also provide explicit examples of states satisfying
the conditions for the pointwise formulation and the
stronger conditions for the regular Hamiltonian formulation.

Finally, Subsection \ref{subsec:softmax-particle-dynamics} uses the softmax geometry
to construct the attention updates considered in
Section~\ref{sec:numerics}.
Proofs are given in Appendix \ref{app:softmax-proofs}.

Figure~\ref{fig:softmax-hierarchy} summarizes the geometric
interpretations studied in this section and their relation
to the density and particle dynamics. 


\begin{figure}[ht]
\centering
\begin{tikzpicture}[
    box/.style={
        draw=blue!55!black,
        fill=blue!10,
        rounded corners=3pt,
        line width=0.7pt,
        minimum height=1.5cm,
        text width=4.8cm,
        align=center,
        inner sep=8pt,
        font=\small
    },
    arrow/.style={
        -{Stealth[length=2.2mm]},
        line width=0.7pt
    },
    explanation/.style={
        font=\scriptsize,
        align=center,
        fill=white,
        inner sep=3pt,
        text width=4.6cm
    }
]
\node[font=\small\bfseries] at (0,1.3)
    {First-order softmax dynamics};
\node[font=\small\bfseries] at (6.1,1.3)
    {Accelerated softmax dynamics};

\node[box] (ordinary) at (0,0)
    {\textbf{Pointwise density equation}\\[5pt]
     \eqref{eq:sm-ordinary}};

\node[box] (accelerated) at (6.1,0)
    {\textbf{Pointwise density equations}\\[5pt]
     Lemma~\ref{prop:sm-pointwise}};

\node[box] (gradient) at (0,-3.3)
    {\textbf{Regular gradient equation}\\[5pt]
     Corollary~\ref{thm:sm-gradient}};

\node[box] (hamiltonian) at (6.1,-3.3)
    {\textbf{Regular damped Hamiltonian system}\\[5pt]
     Theorem~\ref{thm:sm-hamiltonian}};

\draw[arrow] (ordinary) --
    node[explanation, midway]
    {Finite energy and regularity\\
     Proposition~\ref{prop:sm-regular-variation}}
    (gradient);

\draw[arrow] (accelerated) --
    node[explanation, midway]
    {Finite energy, regularity,\\
     and compatibility conditions}
    (hamiltonian);

\node[box, text width=10.9cm] at (3.05,-6.1)
    {\textbf{Separate finite-token construction}\\[5pt]
     An explicit geometric correction in
     Lemma~\ref{lem:sm-particle-connection}
     supplies the position and velocity updates
     used to construct attention layers
     in Section~\ref{sec:numerics}.};
\end{tikzpicture}
\caption{From pointwise softmax equations to regular
geometric formulations. The arrows indicate additional
conditions on the density and potential under which
the equations admit a gradient or Hamiltonian
interpretation. These are statements at admissible
states; existence of trajectories and preservation of
admissibility require a separate analysis.
The finite-token geometry is constructed directly
and provides explicit corrections for attention layers.}
\label{fig:softmax-hierarchy}
\end{figure}
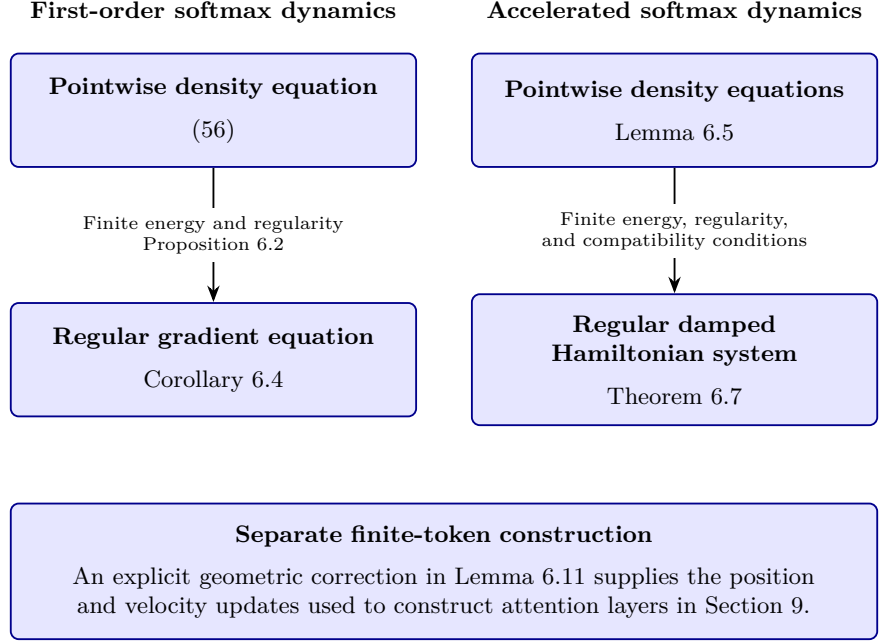

\subsection{Geometry and First-Order Softmax Dynamics}
\label{subsec:sm-ordinary-gradient}

Let $A=A^\top$ be invertible and
$B=VA^{-1}=B^\top\succ0$ as in \eqref{eq:softmax_B}.
We use the density class $\P_{\SM}$ in
\eqref{eq:LM_softmax_class}, the regular tangent and cotangent
spaces in \eqref{eq:LM_softmax_tangent} and
\eqref{eq:LM_softmax_cotangent}, and the metric
$G_\rho^{\SM}$ in \eqref{eq:SM_metric}.
The normalizer $Z_\rho$ and mobility $m_\rho=\rho/Z_\rho$
are defined in \eqref{eq:mobility_measure}.

We begin by recalling the transformer PDE
\eqref{eq:transformer_PDE} for softmax attention
\begin{equation}\label{eq:sm-ordinary}
\partial_t\rho_t+\nabla\cdot(\rho_t\Gamma_{\rho_t})=0,
\qquad
\Gamma_\rho(x)
\coloneqq
\frac{1}
     {Z_\rho(x)} \int_{\R^d}Vy\,\exp(y^\top Ax)\rho(y) \d y
\end{equation}
with normalizing factor 
\begin{equation}\label{eq:def_Z_rho}
  Z_\rho(x)\coloneqq
\int_{\R^d}\exp(y^\top Ax)\rho(y) \d y.
\end{equation}
In \cite[Sec.~5.4]{CACP2025}, the authors give a formal
gradient-flow interpretation of this PDE and leave its
rigorous justification on non-compact domains open.
This subsection aims to address this issue
within the local-mobility framework of Section~\ref{sec3}.

\paragraph{Gradient-flow structure of the softmax transformer PDE.}
Recall that a Riemannian gradient flow is specified by
a metric $G$ and a functional $\F$ whose gradient with
respect to $G$ exists.
Here, we use the softmax--Wasserstein metric
$G_\rho^{\SM}$  in \eqref{eq:SM_metric}, together
with the interaction energy
\begin{equation}\label{eq:sm-energy}
\F^{\SM}(\rho)\coloneqq
-\frac12\int_{\R^d}\rho(x)Z_\rho(x) \d x.
\end{equation}
A formal computation as in~\cite[Sec.~5.4]{CACP2025} identifies \eqref{eq:sm-ordinary}
as the gradient-flow equation of $\F$ with respect to
$G_\rho^{\SM}$.
However, $\F^{\SM}$ is not finite on all of $\P_{\SM}$,
defined in \eqref{eq:LM_softmax_class}, as the following
example shows.

\begin{example}[Infinite energy at an admissible density]
\label{ex:sm-infinite-energy}
Let $d=1$, $A=B=1$, and $\rho=\mathcal N(0,2)$.
Then
\[
Z_\rho(x)=\exp(x^2),
\qquad
m_\rho(x)=\frac{\rho(x)}{Z_\rho(x)}
=\frac1{\sqrt{4\pi}}\exp(-5x^2/4).
\]
Since $m_\rho$ is a finite Gaussian weight, it satisfies
the weighted Poincar\'e inequality \eqref{eq:SM-PI}.
Thus $\rho\in\P_{\SM}$.
Nevertheless, we have
\[
\int_{\R}\rho(x)Z_\rho(x) \d x
=
\frac1{\sqrt{4\pi}}\int_{\R}\exp(3x^2/4) \d x
=\infty,
\]
so that the interaction energy is not finite at $\rho$.
\end{example}
To resolve this obstruction, we restrict $\F^{\SM}$ to its
finite-energy domain. We must then identify its domain
of differentiability, that is, the densities at which
$\F^{\SM}$ admits a local regular first variation in the sense
of \eqref{eq:def_local_regular_variation}.
The following proposition characterizes this domain.

\begin{proposition}[Regular variation criterion]
\label{prop:sm-regular-variation}
Let $\rho\in\P_{\SM}$ satisfy $\F^{\SM}(\rho)\in\R$, where
$\F^{\SM}$ is defined in \eqref{eq:sm-energy}.
Then $\F^{\SM}$ admits a local regular first variation at $\rho$
in the sense of \eqref{eq:def_local_regular_variation}
if and only if
\begin{equation}\label{eq:sm-gradient-tests}
\int_{\R^d}\frac{Z_\rho}{\rho}
|L_\rho^{\SM}Z_\rho|^2 \d x<\infty.
\end{equation}
\end{proposition}
Before moving on, we comment briefly on
\eqref{eq:sm-gradient-tests} from three perspectives.

\begin{remark}[Meaning of the regularity condition]
Let $\rho\in\P_{\SM}$ satisfy $\F^{\SM}(\rho)\in\R$.
\\[1ex]
\textbf{Admissibility of the density variation.}
Condition~\eqref{eq:sm-gradient-tests} ensures that
$L_\rho^{\SM}Z_\rho$ belongs to the regular tangent space
$T_\rho\P_{\SM}$ in \eqref{eq:LM_softmax_tangent}.
Thus, the formal softmax descent direction is admissible.
\\[1ex]
\textbf{Weighted Sobolev regularity.}
Together with the finite energy,
\eqref{eq:sm-gradient-tests} implies
\[
\int_{\R^d}\frac{\rho}{Z_\rho}
\nabla Z_\rho^\top B\nabla Z_\rho \d x<\infty.
\]
This verifies the weighted gradient condition in the
definition of the regular cotangent space
$T_\rho^*\P_{\SM}$ in \eqref{eq:LM_softmax_cotangent}.
\\[1ex]
\textbf{Sufficient uniform bound.}
Condition~\eqref{eq:sm-gradient-tests} holds whenever
$L_\rho^{\SM}Z_\rho/\rho\in L^\infty(\R^d)$, since
\[
\int_{\R^d}\frac{Z_\rho}{\rho}
|L_\rho^{\SM}Z_\rho|^2 \d x
\leq
\left\|\frac{L_\rho^{\SM}Z_\rho}{\rho}\right\|_\infty^2
\int_{\R^d}\rho Z_\rho \d x<\infty.
\]
\end{remark}

Proposition~\ref{prop:sm-regular-variation} allows us to
justify the formal gradient computation in
the following corollary.

\begin{cor}[Regular softmax gradient]
\label{thm:sm-gradient}
Let $\rho$ be as in
Proposition~\ref{prop:sm-regular-variation} and satisfy
\eqref{eq:sm-gradient-tests}.
Then the following hold true:
\begin{enumerate}
\item (Gradient identity)
The negative gradient of $\F^{\SM}$ at $\rho$ is
\begin{equation}\label{eq:sm-gradient-identity}
-\grad_{G_\rho^{\SM}}\F^{\SM}(\rho)
=L_\rho^{\SM}[Z_\rho]
=-\nabla\cdot(\rho\Gamma_\rho)
\qquad\text{in }T_\rho\P_{\SM}.
\end{equation}
\item (Gradient-flow formulation)
On the domain of densities satisfying the above assumptions,
the transformer PDE \eqref{eq:sm-ordinary} is the regular
gradient-flow equation associated with $(G^{\SM},\F^{\SM})$:
\begin{equation} \label{eq:G_SM,F_SM_gradient_flow}
    \partial_t\rho_t
    =
    -\grad_{G_{\rho_t}^{\SM}}\F^{\SM}(\rho_t).    
\end{equation}
\end{enumerate}
\end{cor}

Corollary~\ref{thm:sm-gradient} identifies the gradient
at each admissible density. In this sense, it provides
a rigorous justification of the formal gradient-flow
interpretation in \cite[Sec.~5.4]{CACP2025} on the
domain specified above.

We close this subsection by emphasizing that constructing
gradient-flow trajectories additionally requires local
existence and short-time preservation of the domain on
which \eqref{eq:sm-gradient-identity} holds.
A solution theory for general admissible initial densities
is left for future research.
For more restricted classes of initial densities, such as
perturbations of Gaussian densities, partial results can
be obtained concerning local well-posedness, preservation
of the gradient domain, and loss of finite energy.
We leave these developments to separate work.

\subsection{Accelerated Softmax Dynamics and Hamiltonian Structure}
\label{subsec:sm-regular-hamiltonian}
We begin with the accelerated softmax system
introduced in \eqref{eq:nesterov-softmax-system},
{\small
\begin{equation}\label{eq:sm-accelerated-pde}
\left\{
\begin{aligned}
&\partial_t\rho_t =
-\nabla\cdot\left(
\frac{\rho_t}{Z_{\rho_t}}B\nabla\Phi_t
\right),\\[1ex]
&\partial_t\Phi_t(x)+\alpha_t\Phi_t(x) =
Z_{\rho_t}(x)
-\frac{\nabla\Phi_t(x)^\top B\nabla\Phi_t(x)}
      {2Z_{\rho_t}(x)} +
\frac12\int_{\R^d}\exp(y^\top Ax)
\frac{\nabla\Phi_t(y)^\top B\nabla\Phi_t(y)}
     {Z_{\rho_t}(y)^2}
\,\rho_t(y) \d y
\end{aligned}
\right.
\end{equation}
}
where we used, if the differential of $\F$ is defined, that 
$\frac{\delta \mathcal F}{\delta \rho} = -Z_\rho$ for the interaction energy $\mathcal F$ in \eqref{eq:sm-energy}, see Remark \ref{rem:ass}.
Before moving on, we point out that the second equation
in \eqref{eq:sm-accelerated-pde} is understood modulo
spatial constants. This means that its two sides may
differ by a term depending only on time.
This term can be removed by choosing a suitable additive
function of time in $\Phi_t$, which leaves $\nabla\Phi_t$
and hence the density evolution unchanged.

The aim of this subsection is first to identify conditions
under which \eqref{eq:sm-accelerated-pde} is pointwise
well-defined. We then restrict the domain further to
ensure that its right-hand sides arise from regular
derivatives of the Hamiltonian \eqref{eq:sm-hamiltonian},
giving the PDE a rigorous interpretation as a damped
Hamiltonian system.
We also provide explicit examples of states satisfying
the conditions for the pointwise formulation and the
stronger conditions for the regular Hamiltonian formulation.
\paragraph{Pointwise well-definedness.}
We first point out that the defining conditions for
$\P_{\SM}$ in \eqref{eq:LM_softmax_class} and the regular
cotangent space $T_\rho^*\P_{\SM}$ in
\eqref{eq:LM_softmax_cotangent} ensure that
\eqref{eq:sm-accelerated-pde} is pointwise well-defined
at every pair $(\rho,[\Phi])$ with
$\rho\in\P_{\SM}$ and $[\Phi]\in T_\rho^*\P_{\SM}$.

\begin{lemma}[Pointwise formulation]
\label{prop:sm-pointwise}
Let $\rho\in\P_{\SM}$ and $[\Phi]\in T_\rho^*\P_{\SM}$
with the cotangent space defined in
\eqref{eq:LM_softmax_cotangent}.
Then, all spatial quantities in
\eqref{eq:sm-accelerated-pde} are pointwise well-defined
at $(\rho,\Phi)$, whenever $\alpha(t)$ is finite.
\end{lemma}

Note that finiteness of the interaction energy $\F^{\SM}(\rho)$ is
not required for this pointwise formulation.
Further, it is  straightforward  to verify that
\eqref{eq:sm-accelerated-pde} is also pointwise well-defined
for the following family of densities and potentials.

\begin{example}[Perturbations of Gaussian densities]
\label{prop:sm-pointwise-data}
Let $\gamma=\NN(\mu,\Sigma)$ with $\Sigma\succ0$ and \begin{equation}\label{eq:sm-relative-pointwise-data}
\rho_0=\gamma(1+u),
\qquad
u\in C^\infty(\R^d),
\qquad
\int_{\R^d}\gamma u \d x=0,
\qquad
\|u\|_\infty<1.
\end{equation}
Let $\Phi_0\in C^\infty(\R^d)$ have a gradient of at
most polynomial growth, that is,
\[
|\nabla\Phi_0(x)|\le C(1+|x|)^k
\qquad(x\in\R^d)
\]
for some finite $C,k\ge0$.
Then all spatial quantities in
\eqref{eq:sm-accelerated-pde} are pointwise well-defined
at $(\rho_0,\Phi_0)$ whenever $\alpha(0)$ is finite.
\end{example}

\paragraph{Hamiltonian and formal equations.}
As for gradient flows on a manifold $\mathcal M$, the
construction of a damped Hamiltonian system starts from
a Riemannian metric and a functional $\F^{\SM}$, together with
a damping coefficient.
Here, we use the softmax--Wasserstein metric
$G_\rho^{\SM}$ in \eqref{eq:SM_metric} and the interaction
energy $\F^{\SM}$ in \eqref{eq:sm-energy}.
For $\rho\in\P_{\SM}$ and $[\Phi]\in T_\rho^*\P_{\SM}$,
we define
\begin{equation}\label{eq:sm-hamiltonian}
\mathcal H(\rho,[\Phi])\coloneqq
\frac12 G_\rho^{\SM}
\bigl(L_\rho^{\SM}[\Phi],L_\rho^{\SM}[\Phi]\bigr)
+\F^{\SM}(\rho)=
\frac12\int_{\R^d}
\frac{\rho}{Z_\rho}\nabla\Phi^\top B\nabla\Phi \d x
+\F^{\SM}(\rho).
\end{equation}
The kinetic term in~\eqref{eq:sm-hamiltonian} is finite by the definition of
$T_\rho^*\P_{\SM}$ in \eqref{eq:LM_softmax_cotangent}.
Consequently, $\mathcal H(\rho,[\Phi])$ is finite if and
only if $\F^{\SM}(\rho)$ is finite.
As Example~\ref{ex:sm-infinite-energy} shows, this is
not guaranteed by $\rho\in\P_{\SM}$ alone.

For finite-energy densities,
Proposition~\ref{prop:sm-regular-variation} already
characterizes when $\F^{\SM}$ admits a local regular first
variation.
Next, we give sufficient conditions for the regular
variations of $\mathcal H$ with respect to $\Phi$ and
$\rho$, understood in the sense of
\eqref{eq:def_reg_del_phi_H} and
\eqref{eq:def_reg_del_rho_H}, respectively.

\begin{theorem}[Regular Hamiltonian formulation]
\label{thm:sm-hamiltonian}
Let $\rho$ satisfy the assumptions of
Corollary~\ref{thm:sm-gradient}, and let
$[\Phi]\in T_\rho^*\P_{\SM}$. Further, assume
the following:
\begin{enumerate}
\item
The following cotangent condition holds:
\begin{equation}\label{eq:sm-graph-tests}
\left[
\frac{\nabla\Phi(x)^\top B\nabla\Phi(x)}
     {2Z_\rho(x)}
-\frac12\int_{\R^d}\exp(y^\top Ax)
\frac{\nabla\Phi(y)^\top B\nabla\Phi(y)}
     {Z_\rho(y)^2}\rho(y) \d y
\right]
\in T_\rho^*\P_{\SM}.
\end{equation}
\item
For every $\eta\in C_c^\infty(\R^d)$ with
$\int_{\R^d}\eta \d x=0$, we have
\begin{equation}\label{eq:sm-compatibility}
\left(\nabla\frac{Z_\eta}{Z_\rho}\right)^\top
B\nabla\Phi\in L^2(m_\rho).
\end{equation}
\end{enumerate}
Then~\eqref{eq:sm-accelerated-pde} can be written for the Hamiltonian $\mathcal H$  in~\eqref{eq:sm-hamiltonian} as a
regular damped Hamiltonian system
\begin{equation}\label{eq:sm-regular-hamiltonian-system}
\begin{cases}
\displaystyle
\partial_t\rho_t
=
\frac{\delta^{\reg}\mathcal H}{\delta\Phi}
(\rho_t,[\Phi_t]),\\[1.5ex]
\displaystyle
[\partial_t\Phi_t]+\alpha_t[\Phi_t]
=
-\frac{\delta^{\reg}\mathcal H}{\delta\rho}
(\rho_t,[\Phi_t]).
\end{cases}
\end{equation}
\end{theorem}

Let us examine the above assumptions.

\begin{remark}[Meaning of the additional conditions]
\label{rem:ass}
The assumptions of Corollary~\ref{thm:sm-gradient}
ensure that the interaction energy $\F^{\SM}$ has the regular
first variation $[-Z_\rho]\in T_\rho^*\P_{\SM}$.
The additional conditions concern the kinetic energy:
\\[1ex]
\textbf{Regularity at the given density.}
Condition~\eqref{eq:sm-graph-tests} requires the density
variation of the kinetic energy to be represented by
a regular cotangent vector in $T_\rho^*\P_{\SM}$.
Thus, both contributions to the density variation of
$\mathcal H$ belong to this space.
\\[1ex] 
\textbf{Admissibility under density variations.}
The cotangent space itself depends on $\rho$.
By Lemma~\ref{lem:compatible_variations_otto_softmax},
condition~\eqref{eq:sm-compatibility} ensures that
$[\Phi]\in T_{\rho+\varepsilon\eta}^*\P_{\SM}$
for sufficiently small $|\varepsilon|$.
This allows us to vary the density while keeping
the potential fixed.
\\[1ex]
At $\Phi=0$, both additional conditions hold
automatically, leaving only the assumptions of
Corollary~\ref{thm:sm-gradient}.
\end{remark}
Next, we further illustrate the assumptions of
Theorem~\ref{thm:sm-hamiltonian} and how they go
beyond those of Corollary~\ref{thm:sm-gradient}.
The first example provides density--potential pairs
satisfying the regular Hamiltonian conditions.
The second shows that a density can satisfy the
first-order gradient criterion while an associated
regular cotangent potential fails the additional
Hamiltonian conditions.
We leave the computations verifying these examples
to the reader.
\begin{example}[Gaussian densities and their perturbations]
\label{prop:sm-initial-data}
Let $\gamma=\NN(\mu,\Sigma)$ with $\Sigma\succ0$ satisfy
\begin{equation}\label{eq:sm-gaussian-gap}
\Sigma^{-1}-A\Sigma A\succ0,
\end{equation}
and let $\Phi_0\in C^\infty(\R^d)$ have compactly
supported gradient.
Then the following pairs satisfy the assumptions of
Theorem~\ref{thm:sm-hamiltonian}:
\begin{enumerate}
\item (Gaussian densities)
\label{it:sm-gaussian-data}
The pair $(\gamma,[\Phi_0])$.

\item (Non-Gaussian perturbations)
\label{it:sm-relative-data}
The pair $(\rho_0,[\Phi_0])$, where
$\rho_0=\gamma(1+u)$ satisfies
\eqref{eq:sm-relative-pointwise-data} and
\begin{equation}\label{eq:sm-relative-data}
\|\nabla u\|_\infty<\infty.
\end{equation}
\end{enumerate}
The choice $\Phi_0=0$ is included, and no amplitude
restriction on $\Phi_0$ is required.
\end{example}

Finally, we give an example where
Corollary~\ref{thm:sm-gradient} applies to the density,
but the density--potential pair fails the additional
conditions of Theorem~\ref{thm:sm-hamiltonian}.

\begin{example}[Failure of regular Hamiltonian admissibility]
\label{ex:sm-kinetic-variation-failure}
Let $d=1$, $A=B=1$, and choose
\[
\rho=\NN(0,\tfrac12),
\qquad
\Phi(x)=\exp(x^2/2).
\]
Then
\[
Z_\rho(x)=\exp(x^2/4),
\qquad
m_\rho(x)=\frac1{\sqrt\pi}\exp(-5x^2/4).
\]
One can check that $\rho$ satisfies the assumptions of
Corollary~\ref{thm:sm-gradient}, that
$[\Phi]\in T_\rho^*\P_{\SM}$, and that
\eqref{eq:sm-compatibility} holds.

However, a direct computation shows that the expression
inside the brackets in \eqref{eq:sm-graph-tests} equals
\[
\frac{x^2}{2}\exp(3x^2/4)
-\frac{1+x^2}{\sqrt2}\exp(x^2/2).
\]
Its square is not integrable against $m_\rho$.
Thus, \eqref{eq:sm-graph-tests} fails, although
Corollary~\ref{thm:sm-gradient} applies to $\rho$
and the accelerated PDE is pointwise well-defined
at $(\rho,\Phi)$.
\end{example}

\subsection{From Softmax Geometry to Attention Layers}
\label{subsec:softmax-particle-dynamics}
Now we use the softmax geometry
to construct the attention updates considered in
Section~\ref{sec:numerics}.
We first write the inertial particle dynamics in matrix coordinates. Then we compute the geometric correction
and explain how it enters the position and velocity updates.

\paragraph{Inertial particle dynamics.}
We keep $A=A^\top$ invertible and
$B=VA^{-1}=B^\top\succ0$ fixed.
Let $\mathbf X\in\R^{N\times d}$ contain the token
positions, with rows $X_i^\top$, and define
\begin{equation}\label{eq:sm-particle-weights}
\mathbf M_{ij}(\mathbf X)\coloneqq \exp(X_i^\top AX_j),
\qquad
z_i(\mathbf X)\coloneqq\sum_{j=1}^N\mathbf M_{ij}(\mathbf X).
\end{equation}
We write $D_z=\diag(z_1,\ldots,z_N)$ and denote by
$\mathbf1\in\R^N$ the vector whose entries are all one.

On the token configuration space $\R^{N\times d}$, we
introduce the finite-token analog of the
softmax--Wasserstein metric $G_\rho^{\SM}$ in
\eqref{eq:SM_metric}:
\begin{equation}\label{eq:softmax_metric_particles}
g_{\mathbf X}(\mathbf U,\mathbf H)
\coloneqq
\frac1N\tr\bigl(D_{z(\mathbf X)}
\mathbf U B^{-1}\mathbf H^\top\bigr),
\qquad
\mathbf U,\mathbf H\in\R^{N\times d},
\end{equation}
where the weights $z_i(\mathbf X)$ are defined in
\eqref{eq:sm-particle-weights}. Since $z_i(\mathbf X)>0$ and $B\succ0$, expression
\eqref{eq:softmax_metric_particles} defines a smooth
Riemannian metric.
The interaction energy $\F$ in \eqref{eq:sm-energy}
has the corresponding finite-token form
\begin{equation}\label{eq:softmax_particles}
F^{\SM}(\mathbf X)
\coloneqq
-\frac1{2N}\mathbf1^\top\mathbf M(\mathbf X)\mathbf1.
\end{equation}
Differentiating the energy \eqref{eq:softmax_particles}
and using the metric \eqref{eq:softmax_metric_particles},
the symmetry of $A$ and $\mathbf M(\mathbf X)$, and
the identity $AB=V^\top$ gives
\begin{equation}\label{eq:sm-particle-gradient}
-\grad_gF^{\SM}(\mathbf X)
=
D_{z(\mathbf X)}^{-1}\mathbf M(\mathbf X)\mathbf XV^\top
\eqqcolon\Attn(\mathbf X).
\end{equation}
The right-hand side is precisely the softmax attention
map introduced in Section~\ref{sec:trans}, which we
denote by $\Attn(\mathbf X)$.
The gradient identity \eqref{eq:sm-particle-gradient}
suggests an inertial extension of softmax attention:
we replace the first-order equation
$\dot{\mathbf X}_t=\Attn(\mathbf X_t)$ by the damped
Newton equation
\[
\nabla^g_{\dot{\mathbf X}_t}\dot{\mathbf X}_t
+\alpha_t\dot{\mathbf X}_t
+\grad_gF^{\SM}(\mathbf X_t)=0,
\]
where $\nabla^g$ is the Levi--Civita connection of
\eqref{eq:softmax_metric_particles} and
$\alpha\in C([0,\infty);[0,\infty))$.
In matrix coordinates, this equation reads
\begin{equation}\label{eq:acc_softmax_damped_Newton}
\ddot{\mathbf X}_t+\alpha_t\dot{\mathbf X}_t
=
\Attn(\mathbf X_t)
-
\Gamma_{\mathbf X_t}^g
(\dot{\mathbf X}_t,\dot{\mathbf X}_t),
\end{equation}
where $\Gamma_{\mathbf X_t}^g$ denotes the Christoffel
symbol of the metric \eqref{eq:softmax_metric_particles}.
Momentum-based attention updates already appear in
YuriiFormer \cite[Sec.~3]{ZPR2026}.
The geometric formulation
\eqref{eq:acc_softmax_damped_Newton} additionally includes
the connection term
$-\Gamma_{\mathbf X_t}^g
(\dot{\mathbf X}_t,\dot{\mathbf X}_t)$,
which is determined by the position-dependent metric
\eqref{eq:softmax_metric_particles}.
Thus, this geometry provides a specific correction
to momentum attention.
We next compute this term explicitly in
Lemma~\ref{lem:sm-particle-connection}, providing the
basis for the corrected updates introduced in
subsection~\ref{subsec:attention-discretizations}.

\paragraph{Computing the geometric correction.}
To use \eqref{eq:acc_softmax_damped_Newton} in a layer
update, we need an explicit expression for
$\Gamma_{\mathbf X}^g(\mathbf U,\mathbf U)$.
We first record the directional derivative of the
normalizing vector $z$ given by
\begin{equation}\label{eq:softmax_dz_directional}
\delta z_{\mathbf X}[\mathbf U]
\coloneqq Dz(\mathbf X)[\mathbf U]
=
\bigl[
\mathbf M(\mathbf X)\odot
(\mathbf UA\mathbf X^\top+\mathbf XA\mathbf U^\top)
\bigr]\mathbf1,
\end{equation}
where $\odot$ denotes entrywise multiplication.

The following lemma shows that, since we consider densities on $\R^d$ rather than the sphere, the softmax metric can be expressed in ordinary coordinates, yielding an explicit form of the Christoffel symbols.
\begin{lemma}[Explicit connection correction]
\label{lem:sm-particle-connection}
Let $\mathbf U\in\R^{N\times d}$ have rows $U_i^\top$,
and set $s_i\coloneqq U_i^\top B^{-1}U_i$ and
$D_s\coloneqq\diag(s_1,\ldots,s_N)$.
Then
\begin{equation}\label{eq:softmax_connection_contracted}
\Gamma_{\mathbf X}^g(\mathbf U,\mathbf U)
=
D_{z(\mathbf X)}^{-1}
\left(
D_{\delta z_{\mathbf X}[\mathbf U]}\mathbf U
-\frac12
\bigl(D_s\mathbf M(\mathbf X)+\mathbf M(\mathbf X)D_s\bigr)
\mathbf XV^\top
\right).
\end{equation}
\end{lemma}
The proof of Lemma~\ref{lem:sm-particle-connection}
is given in Appendix~\ref{app:softmax-proofs}.
Formula~\eqref{eq:softmax_connection_contracted}
makes the connection correction directly computable
using matrix products and row rescalings, so that it
can be incorporated into attention layers.
For fixed dense $A,B$ and a precomputed $B^{-1}$,
evaluating the correction requires
$\mathcal O(N^2d+Nd^2)$ arithmetic operations.
The practical computational overhead is assessed in
Section~\ref{sec:numerics}.

\paragraph{Position and velocity updates.}
Next, we use the connection correction~\eqref{eq:softmax_connection_contracted} to construct
discrete position and velocity updates for
\eqref{eq:acc_softmax_damped_Newton}.
Let $\mathbf U$ denote a proposed displacement from
$\mathbf X$, obtained by combining the current velocity
with the attention output $\Attn(\mathbf X)$ defined
in \eqref{eq:sm-particle-gradient}.
To update the token positions, we consider the geodesic
on $\R^{N\times d}$ with respect to the Levi--Civita
connection $\nabla^g$ of the metric
\eqref{eq:softmax_metric_particles}, starting at
$\mathbf X$ with initial velocity $\mathbf U$.
Its second-order Taylor expansion gives the position update
\begin{equation}\label{eq:softmax_second_order_retraction}
R_{\mathbf X}^{(2)}(\mathbf U)
\coloneqq
\mathbf X+\mathbf U
-\frac12\Gamma_{\mathbf X}^g(\mathbf U,\mathbf U).
\end{equation}
To carry the displacement into the next layer, we use
the differential of \eqref{eq:softmax_second_order_retraction}
with respect to $\mathbf U$, keeping $\mathbf X$ fixed.
By bilinearity of $\Gamma_{\mathbf X}^g$, this gives
\begin{equation}\label{eq:softmax_second_order_transport}
\mathcal T_{\mathbf U}^{(2)}[\mathbf U]
\coloneqq
DR_{\mathbf X}^{(2)}(\mathbf U)[\mathbf U]
=
\mathbf U-\Gamma_{\mathbf X}^g(\mathbf U,\mathbf U).
\end{equation}
Both updates reuse the same connection correction,
computed explicitly by
\eqref{eq:softmax_connection_contracted}.
They can therefore be implemented directly as layer
updates, without numerically solving geodesic or
parallel-transport equations.

Section~\ref{sec:numerics} adapts this construction to
causal multi-head attention with learned weights.
There, the connection correction is replaced by the
proxy $\widehat\Gamma$ in \eqref{eq:connection_proxy},
which uses $B=I_d$ and the native attention weights.
Combining \eqref{eq:softmax_second_order_retraction}
and \eqref{eq:softmax_second_order_transport} with the
LayerNorm modifications gives the layer update
\eqref{eq:R2}.
The coordinate variant
\eqref{eq:H1} instead inserts
the correction directly into the velocity update,
following \eqref{eq:acc_softmax_damped_Newton}.
These adaptations provide the architectures whose
predictive performance and computational cost we
compare in the numerical experiments.

\section{Linear Attention: Geometry, Inertial Dynamics and Effective Depth} \label{sec:lin-accelerated-linear-attention}
In  this section, we give a rigorous interpretation of
the accelerated linear-attention system
\eqref{eq:nesterov-stein-system} as damped Hamiltonian dynamics
within the quotient Stein geometry developed in
Section~\ref{sec:stein}, specialized to the linear-attention
energy. Then we use this formulation to study the corresponding
PDE and determine how its finite-time blow-up appears in
discrete attention layers.

In Subsection~\ref{sec:lin-gradient-flow}, we  study
first-order linear attention. We introduce covariance
coordinates on families of elliptically contoured densities
and identify the induced Stein metric in
Lemma~\ref{lem:lin-covariance-metric}. These coordinates
give a finite-dimensional gradient-flow formulation and
specify the metric and energy used to introduce acceleration.
Theorem~\ref{thm:lin-ordinary-correspondence} also establishes
an exact correspondence between the first-order PDE and
its second-moment equation for general $C^1$ initial
probability densities with finite second moment.
Thus, the solution theory extends beyond the elliptical
families used to describe the geometry.

Subsection~\ref{subsec:lin-hamiltonian-dynamics} uses this
metric and energy to construct the accelerated dynamics.
We derive their matrix Hamiltonian formulation and identify
the corresponding damped Newton equation in
Proposition~\ref{thm:lin-newton}.
The main analytical result,
Theorem~\ref{thm:lin-correspondence}, establishes an exact
ODE--PDE correspondence for admissible initial densities
and quadratic initial potentials. This yields local strong
well-posedness and continuous dependence on the initial data.
We then use the matrix equations to characterize continuation:
Proposition~\ref{prop:lin-continuation} shows that finite-time
breakdown requires divergence of the second moment and that
$V\preceq0$ gives global existence.
For commuting data and zero initial momentum,
Theorem~\ref{thm:lin-bounded-damping-blowup} proves that
a positive eigenvalue of $V$ forces finite-time moment
blow-up under bounded nonnegative damping.

Finally, Subsection~\ref{subsec:lin-discrete-attention}
connects these continuous dynamics to finite-token attention
layers. Proposition~\ref{prop:lin-transformer-lift} shows that the
same matrix system determines the continuous token dynamics,
so moment blow-up also forces unbounded token norms at a
fixed number of tokens. After discretization,
Theorem~\ref{thm:lin-transformer-depth} predicts the number
of layers needed to reach a prescribed second-moment
threshold from the corresponding continuous crossing time.
Corollaries~\ref{thm:lin-transformer-subcritical-depth}
and~\ref{cor:lin-transformer-critical-depth} then identify
the continuous blow-up time as a critical effective depth:
as the discretization is refined at fixed effective depth,
token norms remain bounded below this time and become
unbounded at or above it.
All proofs of this section 
are given in Appendix \ref{app:lin-attention-proofs}.

Figure~\ref{fig:linear-attention-dynamics} summarizes
the structure of this section and shows how the continuous
analysis leads to predictions for discrete attention layers.

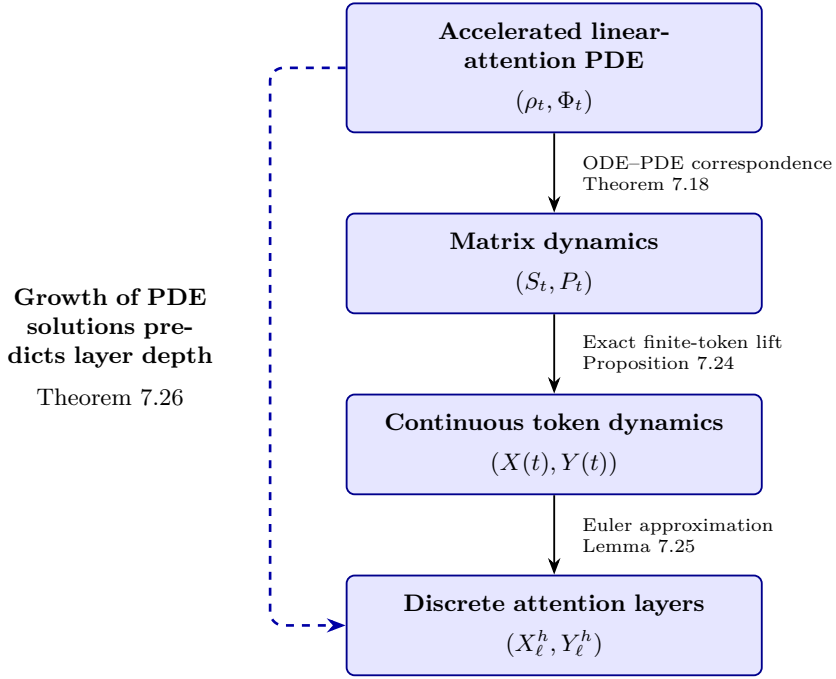
\begin{figure}[ht]
\centering
\begin{tikzpicture}[
    box/.style={
        draw=blue!55!black, fill=blue!10, text=black,
        rounded corners=3pt, line width=0.7pt,
        minimum height=1.05cm, text width=5cm,
        align=center, inner sep=7pt, font=\small
    },
    arrow/.style={
        -{Stealth[length=2.2mm]}, line width=0.7pt
    },
    resultarrow/.style={
        -{Stealth[length=2.5mm]},
        draw=blue!65!black, dashed,
        line width=1pt, rounded corners=5pt
    },
    explanation/.style={
        font=\scriptsize, text=black, align=left
    },
    node distance=1.05cm
]
\node[box] (pde)
    {\textbf{Accelerated linear-attention PDE}\\[4pt]
     $(\rho_t,\Phi_t)$};

\node[box, below=of pde] (matrix)
    {\textbf{Matrix dynamics}\\[4pt]
     $(S_t,P_t)$};

\node[box, below=of matrix] (tokens)
    {\textbf{Continuous token dynamics}\\[4pt]
     $(X(t),Y(t))$};

\node[box, below=of tokens] (layers)
    {\textbf{Discrete attention layers}\\[4pt]
     $(X_\ell^h,Y_\ell^h)$};

\draw[arrow] (pde) --
    node[explanation, midway, right=7pt]
    {ODE--PDE correspondence\\
     Theorem~\ref{thm:lin-correspondence}}
    (matrix);

\draw[arrow] (matrix) --
    node[explanation, midway, right=7pt]
    {Exact finite-token lift\\
     Proposition~\ref{prop:lin-transformer-lift}}
    (tokens);

\draw[arrow] (tokens) --
    node[explanation, midway, right=7pt]
    {Euler approximation\\
     Lemma~\ref{prop:lin-transformer-transfer}}
    (layers);

\coordinate (upper) at ([xshift=-1cm]pde.west);
\coordinate (lower) at (upper |- layers.west);

\draw[resultarrow]
    (pde.west) -- (upper) -- (lower) -- (layers.west);

\node[
    anchor=east, xshift=-9pt,
    text width=3.35cm, align=center,
    font=\small, text=black
] at ($(upper)!0.5!(lower)$)
    {\textbf{Growth of PDE solutions predicts layer depth}\\[5pt]
Theorem~\ref{thm:lin-transformer-depth}};

\end{tikzpicture}
\caption{Understanding transformer dynamics through
the continuum limit. Analytical information about the
growth of PDE solutions yields quantitative predictions
for the depth at which large token norms occur in discrete
attention layers, as established in
Theorem~\ref{thm:lin-transformer-depth}.}
\label{fig:linear-attention-dynamics}
\end{figure}
\subsection{First-Order Linear Attention: Geometry and Analysis}\label{sec:lin-gradient-flow}
Let \(V=V^\top\) and consider \(K_A(x,y)=x^\top Ay\). 
To a given probability density \(\rho\) with finite second moment, we associate the second-moment matrix
\begin{equation}\label{eq:S_rho}
     S(\rho)\coloneqq\int_{\R^d}yy^\top \d\rho(y).
\end{equation}
If \(\rho\)
is centered, the expression in~\eqref{eq:S_rho} is its covariance. We consider the quadratic energy
\begin{equation}\label{eq:lin-density-energy}
 \F^{\lin}(\rho)\coloneqq
 -\frac12\int_{\R^d}x^\top Vx \d\rho(x)
 =-\frac12\tr(VS(\rho)).
\end{equation}
Averaging the values $Vy$ with weights $K_A(x,y)$ gives the
attention field from Section~\ref{sec:trans},
\[
\Gamma_\rho(x)
=
\int_{\R^d}K_A(x,y)\,Vy \d\rho(y)
=
VS(\rho)Ax.
\]
Thus, the associated transformer PDE reads
\begin{equation}\label{eq:lin-ordinary-pde}
\partial_t\rho_t+
\nabla\cdot\left(\rho_tVS(\rho_t)Ax\right)=0,
\qquad
\rho_{t=0}=\rho_{0}.
\end{equation}
Next, we comment on how this relates to Gaussian softmax dynamics studied in \cite{CACP2025}.

\paragraph{Comparison with Gaussian softmax dynamics.}
\label{rem:lin-gaussian-comparison}
As observed by Castin, Ablin, Carrillo, and Peyr\'e
\cite[Sec.~2.2 and Lemma~4.1]{CACP2025},
first-order linear and softmax attention coincide for
centered Gaussian densities.
Here, we study linear attention for non-Gaussian initial
densities as well, where the two dynamics need not coincide.

The quotient Stein geometry also admits matrix choices
excluded by the softmax--Wasserstein framework.
Indeed, for $A=A^\top\succ0$, the construction in
Section~\ref{sec:stein} allows any symmetric $V$,
whereas the softmax geometry in Section~\ref{sec3}
requires $B=VA^{-1}$ to be symmetric positive definite.
For example, $V=-A$ is admissible in the quotient
Stein framework but gives $B=-I_d$, which is excluded
by the softmax metric.

Next, we introduce covariance coordinates to describe
the dynamics on families generated by radial reference
densities.
\paragraph{Covariance coordinates.}
Recall the notation $\P_{K_A}$, $M_\rho$, and
$L_\rho^{K_A}$ from Section~\ref{sec:stein}, the regular effective
tangent and cotangent spaces \eqref{eq:effective-spaces}, and
their metric $G_\rho^{K_A}$ from
Corollary~\ref{cor:stein_metric_well_defined}.
Next, we choose a radial, standardized reference density
$\eta\in\P_{K_A}$, so that
\begin{equation}\label{eq:lin-reference-density}
\eta(Ox)=\eta(x)\quad(O\in\mathrm O(d)),
\qquad
\int_{\R^d}xx^\top\eta(x) \d x=I_d.
\end{equation}
Then, we introduce a family of probability densities parametrized
by their covariance matrices
\begin{equation}\label{eq:lin-elliptical-family}
\begin{aligned}
\iota:\Sym_{++}(d)
&\longrightarrow
\mathcal E_\eta\coloneqq\{\rho_S:S\in\Sym_{++}(d)\},\\
S&\longmapsto\rho_S,
\qquad
\rho_S(x)\coloneqq
\frac{\eta(S^{-1/2}x)}{\sqrt{\det S}}.
\end{aligned}
\end{equation}
Since $\rho_S$ has covariance $S$, the map $\iota$ is injective.
This parametrization equips $\mathcal E_\eta$ with a
finite-dimensional smooth structure. It is straightforward to check that
for every $S\in\Sym_{++}(d)$, the density $\rho_S$ belongs to
$\P_{K_A}$ and is centered.
For each $S\in\Sym_{++}(d)$, we introduce the
\emph{Lyapunov operator} associated with the matrix $SAS$ by
\begin{equation}\label{eq:lin-lyapunov}
\Lambda_S:\Sym(d)\longrightarrow\Sym(d),
\qquad
\Lambda_S(P)=PSAS+SASP.
\end{equation}
With this notation, we can express the quotient Stein metric
in covariance coordinates and reduce the associated gradient
flows from infinite to finite dimensions. The first step is
the following observation.
\begin{lemma}[Properties of the Lyapunov operator]\label{lem:lin-covariance-isomorphism}
For each $S\in\Sym_{++}(d)$, the map $\Lambda_S$ in
\eqref{eq:lin-lyapunov} is an isomorphism, and $D\iota(S)$,
with $\iota$ defined in \eqref{eq:lin-elliptical-family},
is injective.
\end{lemma}

Now we can  use $\Lambda_S^{-1}$ to define a smooth
Riemannian metric on $\Sym_{++}(d)$. For $H,K\in\Sym(d)$, set
\begin{equation}\label{eq:lin-induced-metric}
\begin{aligned}
g_S^{\lin}(H,K)
&\coloneqq
\tfrac12\tr\bigl(H\Lambda_S^{-1}(K)\bigr)
=\tr(PSASQ),\\
P&=\Lambda_S^{-1}(H),
\qquad Q=\Lambda_S^{-1}(K).
\end{aligned}
\end{equation}
This metric can be interpreted as the pullback of $G^{K_A}$
under the parametrization $\iota$, as the following lemma shows.
\begin{lemma}[Relating $g^{\lin}$ and $G^{K_A}$]\label{lem:lin-covariance-metric}
For each $S\in\Sym_{++}(d)$ and all $H,K\in\Sym(d)$,
\begin{equation}\label{eq:lin-covariance-metric-agreement}
g_S^{\lin}(H,K)
=
G_{\rho_S}^{K_A}
\bigl(D\iota(S)[H],D\iota(S)[K]\bigr).
\end{equation}
\end{lemma}

There is the following relation to the Bures--Wasserstein metric.

\begin{remark}[Relation to the Bures--Wasserstein metric]
Lemma~\ref{lem:lin-covariance-metric} identifies $g^{\lin}$
as the pullback of the Stein metric under $\iota$.
It also admits an algebraic relation to the
Bures--Wasserstein metric \cite{BJL2019,MMP2018,T2011}.
For $C\in\Sym_{++}(d)$, let
$\mathcal L_C(P):=PC+CP$. Since
$\Lambda_S=\mathcal L_{SAS}$, we have
\[
g_S^{\lin}(H,K)
=
\tfrac12\tr\bigl(H\mathcal L_{SAS}^{-1}(K)\bigr)
=
g_{SAS}^{\mathrm{BW}}(H,K).
\]
This identifies the bilinear forms at $S$ and $SAS$. However,
it does not identify $g^{\lin}$ as the pullback of
$g^{\mathrm{BW}}$ under $S\mapsto SAS$.
\end{remark}
We close this paragraph by reducing
\eqref{eq:lin-ordinary-pde} on $\mathcal E_\eta$ to a
finite-dimensional gradient flow with respect to this metric.
For the functional
\[
F^{\lin}(S)\coloneqq
\F^{\lin}(\rho_S)=-\tfrac12\tr(VS),
\]
we first observe that
\[
\grad_{g^{\lin}}F^{\lin}(S)=-\Lambda_S(V).
\]
Moreover, the following proposition establishes the equivalence
between the density evolution and this covariance gradient flow.
\begin{proposition}[Covariance metric and gradient-flow formulation]
\label{prop:lin-covariance-geometry}
A $C^1$ curve $S_t\in\Sym_{++}(d)$ gives a solution
$\rho_t=\rho_{S_t}$ of \eqref{eq:lin-ordinary-pde} if and only if
\begin{equation}\label{eq:lin-ordinary-ode}
\dot S_t
=
VS_tAS_t+S_tAS_tV
=
\Lambda_{S_t}(V)
=
-\grad_{g^{\lin}}F^{\lin}(S_t),
\end{equation}
with the corresponding initial data.
\end{proposition}

\begin{remark}[Independence of the radial profile]
\label{rem:lin-gaussian-stein-metric}
The metric \eqref{eq:lin-induced-metric} depends only on $A$
and the covariance $S$, not on the reference density $\eta$.
Thus, all admissible standardized radial profiles induce the
same covariance geometry with Gaussian densities as a special case.

For Gaussian $\eta$, the formula agrees algebraically with
\cite[arXiv v2, Lemma~5.1 and Remark~5.1]{SL2025}
and, when $A=I_d$, with
\cite[arXiv v4, Corollary~A.5]{LGBP2023}.
Such agreement alone does not justify the underlying geometric
construction. Here, well-definedness and the identification
\eqref{eq:lin-covariance-metric-agreement} follow from the
effective tangent and cotangent spaces and the quotient
Stein--Onsager isomorphism developed in Section~\ref{sec:stein}.
\end{remark}

Next, we use this finite-dimensional parametrization to establish
local well-posedness of the associated PDE and give precise
criteria for finite-time blow-up.
\paragraph{Strong solutions and ODE--PDE correspondence.}
We begin by defining a notion of a (local) strong solution to the
transformer PDE \eqref{eq:lin-ordinary-pde} for $C^1$ initial
probability densities with finite second moment.
We omit the proofs of this paragraph, since they follow by minor adaptations of the arguments developed for the more involved damped Hamiltonian case in Section~\ref{subsec:lin-hamiltonian-dynamics}.

Let $I=[0,T)$, and let $\rho_{\mathrm{in}}\in C^1(\R^d)$
be a probability density with finite second moment.

A function $\rho$ is called a \emph{strong solution of
\eqref{eq:lin-ordinary-pde}}  with  initial datum $\rho_{\mathrm{in}}$, if the
following conditions hold:
\begin{enumerate}
\item
The map $t\mapsto\rho_t \in  C\bigl(I;C^1_{\loc}(\R^d)\bigr)
\cap W^{1,1}_{\loc}\bigl(I;C^0_{\loc}(\R^d)\bigr).$ 
\item
The initial condition holds pointwise \(\rho_0(x)=\rho_{\mathrm{in}}(x)\) for \(x\in\R^d\).
\item
For almost every $t\in I$, the matrix integral \(S_t:=S(\rho_t)\) with \(S(\rho_t)\) as in \eqref{eq:S_rho} is absolutely convergent, and $S\in L^1_{\loc}(I;\Sym(d))$.
\item
For every compact set $K\subset\R^d$ and almost every $t\in I$,
\[
\partial_t\rho_t+\nabla\cdot(\rho_tVS_tAx)=0
\qquad\text{in }C(K).
\]
\end{enumerate}
The third condition ensures that the moment-dependent velocity is well-defined and locally integrable in time, while condition four
requires the density to satisfy the resulting continuity equation
on every compact set, hence throughout $\R^d$, without imposing
global uniform bounds on the density or its derivatives.

With this definition, we can establish local well-posedness of
\eqref{eq:lin-ordinary-pde} by reducing it to the matrix equation
\eqref{eq:lin-ordinary-ode}. Moreover, this correspondence shows
that the evolution preserves the class $\P_{K_A}$, i.e., if
$\rho_{\mathrm{in}}\in\P_{K_A}$, then $\rho_t\in\P_{K_A}$
throughout its interval of existence.

\begin{theorem}[Gradient-flow ODE--PDE correspondence]
\label{thm:lin-ordinary-correspondence}
Let $\rho_{0}\in C^1(\R^d)$ be a probability density
with finite second moment, and set \(S_0:= S(\rho_0)\). Then the following statements hold:
\begin{enumerate}
\item\label{it:lin-ordinary-matrix}
The matrix equation \eqref{eq:lin-ordinary-ode} has a unique
maximal solution \(S_t\in C^1([0,T_{\max});\Sym_{++}(d))\) with initial value $S_0$.

\item\label{it:lin-correspondence-ordinary}
Let $G$ solve
\begin{equation}\label{eq:lin-ordinary-propagator}
\dot G_t=VS_tAG_t,
\qquad G_0=I_d.
\end{equation}
Then
\begin{equation}\label{eq:rho_t_orbit_form}
    \rho_t=(G_t)_\#\rho_{0}
\end{equation}
defines a strong solution of \eqref{eq:lin-ordinary-pde}
with second-moment matrix \(S_t=G_tS_0G_t^\top.\)\\
Conversely, every strong solution with initial datum
$\rho_{0}$ has the form~\eqref{eq:rho_t_orbit_form}, and its second moment
solves \eqref{eq:lin-ordinary-ode}.

\item\label{it:lin-ordinary-geometry}
If additionally $\rho_{0}\in\P_{K_A}$, then
we have throughout the evolution that
\[
\rho_t\in\P_{K_A},
\qquad
\partial_t\rho_t
=
-\nabla\cdot(\rho_tVS_tAx)
\in T_{\rho_t}^{\eff}\P_{K_A}.
\]
\item 
If $\rho_0=\rho_{S_0}\in\mathcal E_\eta$, then
$\rho_t=\rho_{S_t}$ remains in $\mathcal E_\eta$ and is
the gradient-flow curve represented in covariance coordinates
by \eqref{eq:lin-ordinary-ode}.
\end{enumerate}
\end{theorem}

We omit the proof of the theorem as well as of its corollaries, since we will prove the corresponding result
in the more involved setting of damped Hamiltonian dynamics
in Theorem~\ref{thm:lin-correspondence}.

\begin{remark}[Relation to Gaussian softmax dynamics]
For $\rho_0=\mathcal N(0,S_0)$, the solution is
$\rho_t=\mathcal N(0,S_t)$.
Since linear and softmax attention agree on centered Gaussians,
this recovers the centered Gaussian case of
\cite[Proposition~4.2]{CACP2025} for the constant symmetric
matrices considered here. The theorem extends the
linear-attention correspondence to non-Gaussian initial densities.
\end{remark}

\paragraph{Local and global well-posedness.}
We first derive local well-posedness of
\eqref{eq:lin-ordinary-pde} by applying the Picard--Lindel\"of
theorem to \eqref{eq:lin-ordinary-propagator} and using the
ODE--PDE correspondence in
Theorem~\ref{thm:lin-ordinary-correspondence},
item~\ref{it:lin-correspondence-ordinary}.

\begin{cor}[Local strong well-posedness]
\label{cor:lin-ordinary-wellposedness}
For every $C^1$ probability density $\rho_0$ on $\R^d$ with finite
second moment, \eqref{eq:lin-ordinary-pde} has a
unique maximal strong solution with initial value $\rho_0$.
Its maximal interval equals that of
\eqref{eq:lin-ordinary-ode} with initial value $S_0=S(\rho_0)$.
\end{cor}
It is natural to ask how the solution of
\eqref{eq:lin-ordinary-pde} depends on the initial density.
The ODE--PDE correspondence in
Theorem~\ref{thm:lin-ordinary-correspondence} transfers
continuous dependence from the matrix equation
\eqref{eq:lin-ordinary-ode} to the PDE, yielding the
following corollary.
We omit its proof, which follows the same argument as
the corresponding result for the damped system in
Corollary~\ref{cor:lin-stability}.

\begin{cor}[Wasserstein convergence and continuous dependence]
\label{cor:lin-ordinary-stability}
Let
$\rho_{0}^{(n)},\rho_{0}$ be $C^1$
probability densities satisfying
\[
\rho_{0}^{(n)}
\longrightarrow\rho_{0}
\qquad\text{in }W_2\text{ and in }C^1_{\loc}(\R^d).
\]
Denote by  $\rho^{(n)}$ and $\rho$  the corresponding maximal
strong solutions of \eqref{eq:lin-ordinary-pde}. Let $[0,T_{\max})$ be the maximal interval of existence of $\rho$. Then, for every $0<T<T_{\max}$, the solutions $\rho^{(n)}$
exist on $[0,T]$ for all sufficiently large $n$, and
\begin{equation}\label{eq:lin-ordinary-stability}
\rho^{(n)}\longrightarrow\rho
\quad\text{in }
C([0,T];W_2)
\cap C([0,T];C^1_{\loc})
\cap C^1([0,T];C^0_{\loc}).
\end{equation}
\end{cor}

\paragraph{Continuation and moment blow-up.} Like many transport equations, \eqref{eq:lin-ordinary-pde}
can exhibit finite-time breakdown. The following proposition characterizes this
behavior through the second moment.

\begin{proposition}[Continuation criterion]
\label{prop:lin-ordinary-continuation}
Let $\rho_0\in C^1(\R^d)$ be a probability density with
finite second moment, and let $\rho$ be the maximal strong
solution of \eqref{eq:lin-ordinary-pde}
with initial datum $\rho_0$.
Denote its maximal existence time by $T_{\max}$ and set
$S_t\coloneqq S(\rho_t)$. If $T_{\max}<\infty$, then the second moments of $\rho_t$ become unbounded as
$t\uparrow T_{\max}$, meaning that
\[
\limsup_{t\uparrow T_{\max}}
\int_{\R^d}|x|^2 \d\rho_t(x)=\infty.
\]
\end{proposition}
We again omit the proof, as it follows by minor
adaptations of the proof of the more involved
continuation criterion for the damped system in
Proposition~\ref{prop:lin-continuation}.
\begin{remark}
The second-moment blow-up is equivalent to
$\limsup_{t\uparrow T_{\max}}\|S_t\|_{\mathrm F}=\infty$.
Indeed, since $S_t$ is positive semidefinite,
\[
\|S_t\|_{\mathrm F}
\leq \tr(S_t)
=\int_{\R^d}|x|^2 \d\rho_t(x)
\leq \sqrt d\,\|S_t\|_{\mathrm F}.
\]
\end{remark}
Next, we use Proposition~\ref{prop:lin-ordinary-continuation}
to obtain global well-posedness under a sign condition on $V$.

\begin{cor}[Well-posedness under sign condition]\label{cor:maxtime}
Under the assumptions of
Proposition~\ref{prop:lin-ordinary-continuation},
if $V\preceq0$, then the maximal existence time of $\rho$
is $T_{\max}=\infty$.
\end{cor}

More interestingly, when $A$, $S_0$, and $V$ commute, we can
solve the matrix equation explicitly. Together with
Proposition~\ref{prop:lin-ordinary-continuation}, this allows
us to determine directly from the initial data whether
finite-time blow-up occurs and, if so, when.

\begin{proposition}[Explicit solution in the commuting case]
\label{prop:lin-ordinary-commuting}
Under the assumptions of
Proposition~\ref{prop:lin-ordinary-continuation},
suppose that $A$, $S_0=S(\rho_0)$, and $V$ commute.
In a common orthonormal eigenbasis, write
\[
A=\diag(a_i),\qquad
S_0=\diag(s_{i,0}),\qquad
V=\diag(v_i),
\qquad a_i,s_{i,0}>0.
\]
Then the following statements hold:
\begin{enumerate}
\item (Representation in the eigenbasis)
For $0\leq t<T_{\max}$,
\[
S_t=\diag(s_i(t)),
\qquad
s_i(t)=\frac{s_{i,0}}{1-2a_iv_is_{i,0}t},
\]
and the maximal existence time is
\[
T_{\max}
=
\min_{i:v_i>0}\frac{1}{2a_iv_is_{i,0}},
\]
where the minimum over the empty set is understood as infinity.

\item (Moment blow-up)
If some $v_i>0$, then $T_{\max}<\infty$ and
\[
\lim_{t\uparrow T_{\max}}
\int_{\R^d}|x|^2 \d\rho_t(x)=\infty.
\]
\end{enumerate}
\end{proposition}

We omit the proof, which follows by diagonalizing
\eqref{eq:lin-ordinary-ode} and solving the resulting
scalar equations. The more involved commuting case
of the damped system is treated in
Theorem~\ref{thm:lin-bounded-damping-blowup}.

For centered Gaussian initial data, the explicit formulas
and blow-up criterion recover the corresponding results
of \cite[Sec.~4.1]{CACP2025}.
For linear attention, these formulas also hold for every
$C^1$ initial probability density with finite second
moment, provided the commutativity assumption of
Proposition~\ref{prop:lin-ordinary-commuting} holds.
Indeed, Theorem~\ref{thm:lin-ordinary-correspondence}
shows that, for fixed $A,V$, initial densities with the
same second-moment matrix $S_0$ have the same moment
evolution and maximal existence time, regardless of
whether they are Gaussian.

\subsection{Accelerated Linear Attention: Geometry and Hamiltonian Dynamics}
\label{subsec:lin-hamiltonian-dynamics}
The aim of this subsection is to formulate accelerated
linear attention as a damped Hamiltonian system and
use this formulation to study the associated PDE.
We first identify admissible densities and quadratic
potentials, then derive the matrix equations in
covariance coordinates.

The exact ODE--PDE correspondence in
Theorem~\ref{thm:lin-correspondence} allows us to establish
strong well-posedness beyond the elliptical family
$\mathcal E_\eta$ and to study global existence and
finite-time moment blow-up.
These results provide the basis for the analysis of
finite-token layers in
Subsection~\ref{subsec:lin-discrete-attention}.

We first fix the notation and assumptions used throughout
this subsection. Let $A=A^\top\succ0$ and $V=V^\top$
be fixed, and let
$\alpha\in C([0,\infty);[0,\infty))$.
For damping schedules singular at the origin, we translate
time and use $\alpha(t_0+t)$ with $t_0>0$.
Here, acceleration refers to the introduction of inertia;
a convergence-rate improvement would require a separate
analysis.
\paragraph{Quadratic potentials.}
We begin by identifying a class of probability densities for which
quadratic potentials belong to the regular cotangent space.
For each $P\in\Sym(d)$, define
\begin{equation}\label{eq:lin-quadratic-potential}
\Phi_P(x)\coloneqq\tfrac12x^\top Px.
\end{equation}
We observe that
\[
|\Phi_P(x)|^2
\leq\tfrac14\|P\|_{\mathrm{op}}^2|x|^4,
\qquad
|\nabla\Phi_P(x)|^2
\leq\|P\|_{\mathrm{op}}^2|x|^2.
\]
Thus, $\Phi_P\in H^1(\rho)$ whenever $\rho\in\P_{K_A}$ has
finite fourth moment. The choice $P=I_d$ also shows that
$\Phi_{I_d}\in L^2(\rho)$ requires a finite fourth moment.
We therefore introduce
\begin{equation}\label{eq:lin-fourth-moment-class}
\P_{K_A}^{4}
\coloneqq
\left\{
\rho\in\P_{K_A}:
\int_{\R^d}|x|^4 \d\rho(x)<\infty
\right\}.
\end{equation}

For the family $\mathcal E_\eta$, we henceforth choose the radial,
standardized reference density $\eta\in\P_{K_A}^{4}$.
The covariance metric $g^{\lin}$ remains the one defined in
\eqref{eq:lin-induced-metric}.
Under this additional moment assumption, a short computation,
which we leave to the reader, shows that, for every
$\rho_S\in\mathcal E_\eta$,
\begin{equation}\label{eq:lin-full-leaf-tangent}
\ran M_{\rho_S}=S\Sym(d),
\qquad
T_{\rho_S}^{\eff}\P_{K_A}=T_{\rho_S}\mathcal E_\eta.
\end{equation}
Moreover, every effective cotangent class has a unique
homogeneous quadratic representative
\[
\Phi_P(x)=\tfrac12x^\top Px,
\qquad
P=S^{-1}M_{\rho_S}(\Phi)\in\Sym(d),
\]
where $\Phi$ is any representative of the class.

\paragraph{Hamiltonian formulation.}
For $\rho\in\P_{K_A}^{4}$ and
$[\Phi]\in T_\rho^*\P_{K_A}$, we consider the Hamiltonian
associated with the energy $\F^{\lin}$:
\begin{equation}\label{eq:lin-density-hamiltonian}
\mathcal H(\rho,[\Phi])
\coloneqq
\frac12\left\langle
[\Phi],L_\rho^{K_A}[\Phi]
\right\rangle_\rho+\F^{\lin}(\rho)=
\frac12\tr\bigl(M_\rho(\Phi)^\top A M_\rho(\Phi)\bigr)
-\frac12\tr(VS(\rho)).
\end{equation}
This is the Hamiltonian in \eqref{eq:def_stein_ham_app}
with $\F=\F^{\lin}$. The notion of differentiability used
to compute its variations is introduced in
Appendix~\ref{app:diff_stein_geom}. To describe its variation with respect to the density, we set
\begin{equation}\label{eq:lin-kinetic-variation}
\mathcal Q_\rho(\Phi)(x)
\coloneqq
\nabla\Phi(x)^\top M_\rho(\Phi)^\top Ax.
\end{equation}
A short computation confirms that the first variation of $\F^{\lin}$ is $-\Phi_V$. So for a quadratic potential $\Phi_P$, we have
\[
M_\rho(\Phi_P)=S(\rho)P,
\qquad
\mathcal Q_\rho(\Phi_P)
=
\Phi_{P^2S(\rho)A+AS(\rho)P^2}.
\]
Here the assumption $\rho\in \P_{K_A}^4$, that is $\rho$ has a finite fourth moment, ensures that both
$\mathcal Q_\rho(\Phi_P)$ and $\Phi_V$ belong to $H^1(\rho)$.
Moreover, compactly supported admissible density variations
preserve finiteness of the fourth moment.
For $\Phi=\Phi_P$, we can therefore apply
Lemma~\ref{lem:regular-stein-hamiltonian-representatives}
to obtain
\begin{equation}\label{eq:lin-regular-variations}
\frac{\delta^{\reg}\mathcal H}{\delta\Phi}
=
L_\rho^{K_A}[\Phi],
\qquad
\frac{\delta^{\reg}\mathcal H}{\delta\rho}
=
[\mathcal Q_\rho(\Phi)-\Phi_V].
\end{equation}
 It would be interesting to extend
this Hamiltonian formulation to non-quadratic potentials in future
work. This would require, in particular, identifying conditions
that ensure $\mathcal Q_\rho(\Phi)\in H^1(\rho)$.

With these variations at hand, the damped Hamiltonian equations read
\begin{equation}\label{eq:lin-accelerated-pde}
\begin{cases}
\partial_t\rho_t+
\nabla\cdot\bigl(\rho_tM_{\rho_t}(\Phi_t)^\top Ax\bigr)=0,
&\rho_{t=0}=\rho_{0},\\[0.4ex]
\partial_t\Phi_t+\alpha_t\Phi_t+
\nabla\Phi_t^\top M_{\rho_t}(\Phi_t)^\top Ax
-\tfrac12x^\top Vx=0,
&\Phi_{t=0}=\Phi_{0}.
\end{cases}
\end{equation}
Before moving on, we emphasize the following.
\begin{remark}
In \eqref{eq:stein_newt_app_pde}, the second equation is
understood modulo spatial constants. The first equation
depends only on $\nabla\Phi_t$ and is therefore unchanged
if a function of time is added to $\Phi_t$.
We use this freedom to choose $\Phi_t$ so that the second
equation in \eqref{eq:lin-accelerated-pde} holds exactly,
without an additional spatially constant term.
The precise interpretation in terms of cotangent representatives
is given in Appendix~\ref{app:diff_stein_geom}.
\end{remark}

\paragraph{Covariance Hamiltonian and Newton equation.}
We now  reduce \eqref{eq:lin-accelerated-pde} on $\mathcal E_\eta$
to a finite-dimensional system for the covariance matrix $S_t$
in \eqref{eq:lin-elliptical-family} and the quadratic coefficient
$P_t$ in \eqref{eq:lin-quadratic-potential}.
\begin{lemma}[Quadratic reduction]
\label{lem:lin-quadratic-reduction}
Let $I=[0,T)$, $S\in C^1(I;\Sym_{++}(d))$, and
$P\in C^1(I;\Sym(d))$.
Set $\rho_t=\rho_{S_t}$ and $\Phi_t=\Phi_{P_t}$, as defined in
\eqref{eq:lin-elliptical-family} and
\eqref{eq:lin-quadratic-potential}, respectively.\\
Then $(\rho,\Phi)$ solves \eqref{eq:lin-accelerated-pde}
with initial data $\rho_0=\rho_{S_0}$ and $\Phi_0=\Phi_{P_0}$
if and only if
\begin{equation}\label{eq:lin-accelerated-ode}
\begin{cases}
\dot S_t=P_tS_tAS_t+S_tAS_tP_t,\\[0.4ex]
\dot P_t=-\alpha_tP_t-AS_tP_t^2-P_t^2S_tA+V.
\end{cases}
\end{equation}
\end{lemma}
The proof of Lemma~\ref{lem:lin-quadratic-reduction}
is given in Appendix~\ref{app:lin-quadratic-correspondence}. Next, we formulate the matrix dynamics on the cotangent bundle
$T^*\Sym_{++}(d)$ of the covariance manifold.
Since $\Sym_{++}(d)$ is an open subset of $\Sym(d)$, its
tangent space at each $S$ is $\Sym(d)$.
Every cotangent vector can therefore be represented uniquely
by a symmetric matrix $\Pi$ through the Frobenius pairing
\[
H\longmapsto\langle\Pi,H\rangle_{\mathrm F}
\coloneqq\tr(\Pi H),
\qquad H\in\Sym(d).
\]
This identifies
$T^*\Sym_{++}(d)$ with $\Sym_{++}(d)\times\Sym(d)$.

The quadratic potential $\Phi_P$ pairs with a covariance
variation $H$ as $\tfrac12\tr(PH)$.
Its corresponding momentum coordinate is therefore $\Pi=P/2$.
We define the Hamiltonian
\begin{equation}\label{eq:lin-matrix-hamiltonian}
\mathcal H^{\lin}:T^*\Sym_{++}(d)\longrightarrow\R,
\qquad
\mathcal H^{\lin}(S,\Pi)
\coloneqq
2\tr(\Pi SAS\Pi)-\tfrac12\tr(VS).
\end{equation}
With this notation, we can confirm that \eqref{eq:lin-accelerated-ode} can be written as a damped Hamiltonian system.
\begin{lemma}[Covariance Hamiltonian]
\label{lem:lin-canonical-hamiltonian}
The change of variables $\Pi=P/2$ transforms
\eqref{eq:lin-accelerated-ode} into the damped Hamiltonian system
\[
\dot S_t
=
\partial_\Pi\mathcal H^{\lin}(S_t,\Pi_t),
\qquad
\dot\Pi_t
=
-\partial_S\mathcal H^{\lin}(S_t,\Pi_t)-\alpha_t\Pi_t.
\]
\end{lemma}
The proof of Lemma~\ref{lem:lin-canonical-hamiltonian}
is given in Appendix~\ref{app:lin-quadratic-correspondence}. We close this paragraph by expressing
\eqref{eq:lin-accelerated-ode} as a damped Newton equation
for the metric $g^{\lin}$ in \eqref{eq:lin-induced-metric}.
We denote its Levi--Civita connection by
$\nabla^{g^{\lin}}$.

\begin{proposition}[Newton formulation]
\label{thm:lin-newton}
Let $(S,P)$ be a solution of \eqref{eq:lin-accelerated-ode}.
Then $S\in C^2$ and
\begin{equation}\label{eq:lin-newton}
\nabla^{g^{\lin}}_{\dot S_t}\dot S_t
+\alpha_t\dot S_t+\grad_{g^{\lin}}F^{\lin}(S_t)=0.
\end{equation}
Conversely, every solution $S$ of \eqref{eq:lin-newton}
in $\Sym_{++}(d)$ induces a solution of
\eqref{eq:lin-accelerated-ode} upon setting
$P_t=\Lambda_{S_t}^{-1}(\dot S_t)$.
\end{proposition}
The proof of Proposition~\ref{thm:lin-newton}
is given in Appendix~\ref{app:lin-quadratic-correspondence}.
Next, we use the viewpoint developed here to gain further
insight into the dynamics of \eqref{eq:lin-accelerated-pde}.
\paragraph{Strong solutions and exact correspondence.}
First, we  define a suitable notion of a strong solution of
\eqref{eq:lin-accelerated-pde}. Then, we  use the covariance
formulation to establish local well-posedness for initial
densities in $\P_{K_A}^{4}$ with quadratic initial potentials.

Let $I=[0,T)$, and let $\rho_0,\Phi_0\in C^1(\R^d)$,
where $\rho_0$ is a probability density with finite second moment.
A pair $(\rho,\Phi)$ is called a \emph{strong solution of}
\eqref{eq:lin-accelerated-pde} with initial data $(\rho_0,\Phi_0)$,
if the following conditions are fulfilled:
\begin{enumerate}
\item\label{it:lin-strong-regularity}
The maps $t\mapsto\rho_t$ and $t\mapsto\Phi_t$ belong to
$C\bigl(I;C^1_{\loc}(\R^d)\bigr)
\cap W^{1,1}_{\loc}\bigl(I;C^0_{\loc}(\R^d)\bigr)$.
\item\label{it:lin-strong-initial}
The initial conditions hold pointwise:
$\rho(0,x)=\rho_0(x)$ and $\Phi(0,x)=\Phi_0(x)$
for all $x\in\R^d$.
\item\label{it:lin-strong-moment}
For almost every $t\in I$, the integral defining
$M_{\rho_t}(\Phi_t)$ in \eqref{eq:moment-map}
is absolutely convergent, and
\[
t\longmapsto M_{\rho_t}(\Phi_t)
\in L^1_{\loc}(I;\R^{d\times d}).
\]
\item\label{it:lin-strong-pde}
For every compact set $K\subset\R^d$ and almost every $t\in I$,
\[
\begin{pmatrix}
\partial_t\rho_t+
\nabla\cdot\bigl(\rho_tM_{\rho_t}(\Phi_t)^\top Ax\bigr)\\
\partial_t\Phi_t+\alpha_t\Phi_t+
\nabla\Phi_t^\top M_{\rho_t}(\Phi_t)^\top Ax
\end{pmatrix}
=
\begin{pmatrix}
0\\
\frac12x^\top Vx
\end{pmatrix}
\qquad\text{in }C(K)\times C(K).
\]
\end{enumerate}

With this definition, we can now establish the equivalence between an ODE and PDE system, which will be essential to establish well-posedness and show that the flow preserves \(\P_{K_A}^{4}\).

\begin{theorem}[Strong correspondence and maximal well-posedness]
\label{thm:lin-correspondence}
Let $\rho_0\in\P_{K_A}^{4}$, with $\P_{K_A}^{4}$ defined in
\eqref{eq:lin-fourth-moment-class}.
Set $S_0=S(\rho_0)$ as in \eqref{eq:S_rho},
and choose $P_0\in\Sym(d)$ with
$\Phi_0=\Phi_{P_0}$ as in \eqref{eq:lin-quadratic-potential}.
Assume that $\alpha\in C([0,\infty);[0,\infty))$.
Then the following statements hold:
\begin{enumerate}
\item (Matrix well-posedness)\label{it:lin-correspondence-matrix}
The matrix system \eqref{eq:lin-accelerated-ode} has a unique
maximal solution 
\[
(S,P)\in C^1\big([0,T_{\max});\Sym_{++}(d)\times \Sym(d)\big)
\] with initial values $(S_0,P_0)$.
    
\item (Construction of strong solutions)\label{it:lin-correspondence-lift}
Let $G$ solve
\begin{equation}\label{eq:lin-propagator}
\dot G_t=P_tS_tAG_t,
\qquad G_0=I_d.
\end{equation}
Then
\begin{equation}\label{eq:lin-accelerated-lift}
\rho_t=(G_t)_\#\rho_0,
\qquad
\Phi_t(x)=\tfrac12x^\top P_tx,
\qquad
S_t=G_tS_0G_t^\top
\end{equation}
defines a strong solution of \eqref{eq:lin-accelerated-pde}
with initial data $(\rho_0,\Phi_0)$.
\item (Converse correspondence)\label{it:lin-correspondence-converse}
Every strong solution $(\rho,\Phi)$ of
\eqref{eq:lin-accelerated-pde}  with initial values
$(\rho_0,\Phi_0)$ has the form \eqref{eq:lin-accelerated-lift},
and its coefficients solve \eqref{eq:lin-accelerated-ode}.
\item (Preservation of geometric admissibility)
\label{it:lin-correspondence-geometry}
For every $t\in[0,T_{\max})$, we have
\[
\rho_t\in\P_{K_A}^{4},
\qquad
\partial_t\rho_t\in T_{\rho_t}^{\eff}\P_{K_A}.
\]
\end{enumerate}
\end{theorem}
The proof of Theorem~\ref{thm:lin-correspondence}
is given in Appendix~\ref{app:lin-quadratic-correspondence}.

\begin{remark}[Relation to accelerated Stein flows]
\label{rem:lin-accelerated-precedents}
1.
For $K_A$ and $\F^{\lin}$, \eqref{eq:lin-accelerated-pde}
agrees algebraically with \cite[Example~4.1]{SL2025}.
Our quotient Stein construction makes this specialization
rigorous, while Theorem~\ref{thm:lin-correspondence}
establishes strong well-posedness beyond Gaussian densities
and allows arbitrary quadratic initial potentials.
\\
2.
The Gaussian kinetic terms in
\cite[Theorem~3.5]{LGBP2023} and
\cite[Theorem~5.1(4)]{SL2025} are recovered for $A=I_d$
and general $A$, respectively.
Their full dynamics differ because they use KL energy
rather than $\F^{\lin}$.
\end{remark}

\paragraph{Local well-posedness and continuous dependence on initial values}
We start by recording the following direct
consequence of Theorem~\ref{thm:lin-correspondence}.

\begin{cor}[Local strong well-posedness]
\label{cor:lin-wellposedness}
Under the assumptions of Theorem~\ref{thm:lin-correspondence},
equation \eqref{eq:lin-accelerated-pde} has a unique maximal
strong solution with initial data $(\rho_0,\Phi_0)$.
Its maximal interval equals that of the matrix system
\eqref{eq:lin-accelerated-ode} with initial values $(S_0, P_0)$.
\end{cor}
Next, we study how solutions of
\eqref{eq:lin-accelerated-pde} depend on their initial
values in the Wasserstein metric and the local strong
topologies. The following corollary is proved in
Appendix~\ref{app:lin-quadratic-correspondence}.

\begin{cor}[Wasserstein convergence and continuous dependence]
\label{cor:lin-stability} Suppose
\(\rho_{0}^{(n)},\rho_{0}\in\P_{K_A}^{4}\) satisfy
\[
\rho_0^{(n)}\longrightarrow\rho_0
\quad\text{in }W_2\text{ and in }C^1_{\loc}(\R^d),
\qquad
P_0^{(n)}\longrightarrow P_0
\quad\text{in }(\Sym(d),\|\cdot\|_{\mathrm F}).
\]
For fixed $\alpha$, let $(\rho^{(n)},\Phi^{(n)})$ and
$(\rho,\Phi)$ be the unique maximal strong solutions of
\eqref{eq:lin-accelerated-pde} with initial data
$(\rho_0^{(n)},\Phi_{P_0^{(n)}})$ and
$(\rho_0,\Phi_{P_0})$, respectively.
Let $T_{\max}$ denote the maximal existence time of
$(\rho,\Phi)$. Then, for every $0<T<T_{\max}$, the solutions
$(\rho^{(n)},\Phi^{(n)})$ exist on $[0,T]$ for 
sufficiently large $n$, and
\begin{equation}\label{eq:lin-strong-stability}
\begin{gathered}
\rho^{(n)}\longrightarrow\rho
\quad\text{in }
C([0,T];W_2)
\cap C([0,T];C^1_{\loc})
\cap C^1([0,T];C^0_{\loc}),\\
\Phi^{(n)}\longrightarrow\Phi
\quad\text{in }
C([0,T];C^1_{\loc})
\cap C^1([0,T];C^0_{\loc}).
\end{gathered}
\end{equation}
\end{cor}

Next, we investigate when strong solutions of
\eqref{eq:lin-accelerated-pde} can be continued.
Proposition~\ref{prop:lin-continuation} shows that,
if the maximal existence time is finite, the second
moment of $\rho_t$ must become unbounded.
\paragraph{Continuation, global existence, and blow-up.}
Next, we give a continuation criterion for
\eqref{eq:lin-accelerated-pde} and identify conditions
for global existence and finite-time blow-up.

\begin{proposition}[Continuation and global existence]
\label{prop:lin-continuation}
Under the assumptions of Theorem~\ref{thm:lin-correspondence},
let $(S,P)$ be the maximal solution of
\eqref{eq:lin-accelerated-ode} and $(\rho,\Phi)$ the
corresponding maximal strong solution of
\eqref{eq:lin-accelerated-pde}.
Denote their common maximal existence interval by
$[0,T_{\max})$.
Then the following statements hold:
\begin{enumerate}
\item\label{it:lin-continuation-criterion}
If $T_{\max}<\infty$, then
\begin{equation}\label{eq:lin-continuation-criterion}
\limsup_{t\uparrow T_{\max}}\|S_t\|_{\mathrm F}=\infty,
\qquad
\limsup_{t\uparrow T_{\max}}
\int_{\R^d}|x|^2 \d\rho_t(x)=\infty.
\end{equation}

\item\label{it:lin-continuation-global}
If $V\preceq0$, then $T_{\max}=\infty$.
\end{enumerate}
\end{proposition}
The proof of Proposition~\ref{prop:lin-continuation}
is given in Appendix~\ref{app:lin-quadratic-correspondence}.

To identify conditions for finite-time blow-up, we now
assume that $P_0=0$ and that $A, S_0=S(\rho_0), V$ commute.
In a common orthonormal eigenbasis, write
\[
A=\diag(a_i),\qquad
S_0=\diag(s_{i,0}),\qquad
V=\diag(v_i),
\qquad a_i,s_{i,0}>0.
\]
Uniqueness for \eqref{eq:lin-accelerated-ode} implies
that $S_t$ and $P_t$ remain diagonal in this basis.
Writing $S_t=\diag(s_i(t))$ and $P_t=\diag(p_i(t))$,
the matrix system reduces to
\begin{equation}\label{eq:lin-diagonal-ode}
\begin{cases}
\dot s_i=2a_ip_is_i^2,
&s_i(0)=s_{i,0},\\
\dot p_i=-\alpha_t p_i-2a_is_ip_i^2+v_i,
&p_i(0)=0.
\end{cases}
\end{equation}
This allows us to state the following theorem.

\begin{theorem}[Finite-time moment blow-up under bounded damping]
\label{thm:lin-bounded-damping-blowup}
Under the assumptions of
Proposition~\ref{prop:lin-continuation},
suppose additionally that $P_0=0$, that $A,S_0,V$ commute,
and that $\alpha_t\leq\bar\alpha<\infty$ for all $t\geq0$.
Let $s_i,p_i$ be the solutions of
\eqref{eq:lin-diagonal-ode}.
Then we have the following:
\begin{enumerate}
\item\label{it:lin-bounded-damping-blowup}
If some $v_i>0$, then $T_{\max}<\infty$.
At least one index $i$ with $v_i>0$ satisfies
\begin{equation}\label{eq:lin-commuting-blowup}
\lim_{t\uparrow T_{\max}}s_i(t)=\infty,
\qquad
\lim_{t\uparrow T_{\max}}
\int_{\R^d}|x|^2 \d\rho_t(x)=\infty.
\end{equation}

\item\label{it:lin-positive-moment-monotonicity}
For every index $i$ with $v_i>0$, the function $s_i$
is strictly increasing on $(0,T_{\max})$.

\item\label{it:lin-commuting-global}
If all $v_i\leq0$, then $T_{\max}=\infty$.
\end{enumerate}
\end{theorem}
The proof of Theorem~\ref{thm:lin-bounded-damping-blowup}
is given in Appendix~\ref{app:lin-quadratic-correspondence}.
Its key step is to establish that, for $v_i>0$,
$\dot s_i\geq k s_i^{3/2}$ for some $k>0$ once $s_i$
is sufficiently large. This superlinear growth forces
finite-time blow-up despite bounded damping.

For illustration, in the undamped case $\alpha\equiv0$,
the solutions of \eqref{eq:lin-diagonal-ode} can be
computed explicitly. For each $v_i>0$, one obtains
\begin{equation}\label{eq:lin-explicit-undamped}
s_i(t)
=
s_{i,0}\sec^2\!\bigl(\sqrt{a_iv_is_{i,0}}\,t\bigr),
\end{equation}
and hence
\begin{equation}\label{eq:lin-explicit-escape-time}
T_{\max}
=
\min_{i:v_i>0}
\frac{\pi}{2\sqrt{a_iv_is_{i,0}}}.
\end{equation}
Thus, the initial second moments and the eigenvalues
of $A$ and $V$ determine the blow-up time explicitly.
\subsection{Finite-Token Layers and Effective Depth}
\label{subsec:lin-discrete-attention}
The aim of this subsection is to use continuous moment
growth to predict when large token norms occur in
discrete attention layers.
We first relate the continuous token dynamics to the
matrix system \eqref{eq:lin-accelerated-ode}, then
discretize them using the explicit Euler method.

Theorem~\ref{thm:lin-transformer-depth} uses this
correspondence to predict the number of layers needed
to reach a prescribed second-moment threshold.
We then examine how the continuous blow-up time
separates bounded and unbounded token growth as the
discretization is refined, keeping the number of
tokens fixed.

\paragraph{Finite-token dynamics and attention layers.}
We begin by fixing $N\geq d$, $A=A^\top\succ0$, $V=V^\top$, and
$X_0\in\R^{N\times d}$ with full column rank.
The rows of $X_0$ are the initial token vectors, and 
\[
S_0\coloneqq\frac1N X_0^\top X_0\succ0
\]
is their second-moment matrix. Note that, as is well known, for centered tokens the
second-moment matrix $X^\top X/N$ is also their empirical
covariance matrix, analogous to the covariance parameter
$S$ of the centered densities in
\eqref{eq:lin-elliptical-family}. Next, we introduce a coupled system of ODEs for the token
matrix $X(t)\in\R^{N\times d}$ and its momentum matrix
$Y(t)\in\R^{N\times d}$:
\begin{equation}\label{eq:lin-transformer-flow}
\begin{cases}
\dot X=\dfrac1N XAX^\top Y,
&X(0)=X_0,\\[1ex]
\dot Y=-\alpha(t)Y-\dfrac1N YY^\top XA+XV,
&Y(0)=0.
\end{cases}
\end{equation}
Next, we observe that the matrix system
\eqref{eq:lin-accelerated-ode} also determines the token
positions and momenta in \eqref{eq:lin-transformer-flow}.
This is part of the following proposition, which follows directly from a computation,
uniqueness for \eqref{eq:lin-transformer-flow}, and the
fact that an unbounded second-moment matrix forces
unbounded token positions.

\begin{proposition}[Exact finite-token representation]
\label{prop:lin-transformer-lift}
Let $\alpha\in C([0,\infty);[0,\infty))$, and let $(S,P)$
be the maximal solution of \eqref{eq:lin-accelerated-ode}
with initial values $(X_0^\top X_0/N,0)$ on
$[0,T_{\max})$.
Let $G$ solve \eqref{eq:lin-propagator}.
Then the following statements hold:
\begin{enumerate}
\item (Exact finite-token representation)
\label{it:lin-transformer-lift}
The unique solution of \eqref{eq:lin-transformer-flow}
on $[0,T_{\max})$ is given by
\begin{equation}\label{eq:lin-transformer-lift}
X(t)=X_0G_t^\top,
\qquad
Y(t)=X(t)P_t,
\qquad
\frac1N X(t)^\top X(t)=S_t.
\end{equation}
\item (Equality of maximal existence times)
\label{it:lin-transformer-maximal-time}
If $T_{\max}<\infty$ and $S_t$ is unbounded as
$t\uparrow T_{\max}$, then $T_{\max}$ is also the maximal
existence time of \eqref{eq:lin-transformer-flow}.
\end{enumerate}
\end{proposition}
The proof of Proposition~\ref{prop:lin-transformer-lift}
is given in Appendix~\ref{app:lin-continuation-proofs}. By \eqref{eq:lin-transformer-lift}, blow-up of the
second-moment matrix $S_t$ forces unbounded token norms
in \eqref{eq:lin-transformer-flow}, even for a fixed
number of tokens $N$.

Having described the continuous token dynamics, we next
construct the corresponding discrete attention layers.
Applying the explicit Euler method to
\eqref{eq:lin-transformer-flow} with step size $h>0$ gives
\begin{equation}\label{eq:lin-transformer-layers}
\begin{cases}
X_{\ell+1}^h
=X_\ell^h+\dfrac hN X_\ell^h A(X_\ell^h)^\top Y_\ell^h,\\[1ex]
Y_{\ell+1}^h
=(1-h\alpha(\ell h))Y_\ell^h
-\dfrac hN Y_\ell^h(Y_\ell^h)^\top X_\ell^h A
+hX_\ell^h V
\end{cases}
\end{equation}
with initial values $(X_0^h,Y_0^h)=(X_0,0)$.
Following the residual attention formulation in
Section~\ref{sec:trans}, \eqref{eq:lin-transformer-layers}
defines layers with shared matrices $A,V$ and an additional
momentum variable $Y$. Our analysis concerns these updates
without normalization or MLP blocks. 

Next, to measure how well the layers \eqref{eq:lin-transformer-layers}
approximate the solution of \eqref{eq:lin-transformer-flow},
we introduce the notation
\[
Z(t)\coloneqq(X(t),Y(t)),
\qquad
Z_\ell^h\coloneqq(X_\ell^h,Y_\ell^h),
\qquad
\|(X,Y)\|^2\coloneqq
\|X\|_{\mathrm F}^2+\|Y\|_{\mathrm F}^2.
\]
To compare their second moments, we also set
\begin{equation}\label{eq:S_ell_h}
    S_\ell^h\coloneqq\frac1N(X_\ell^h)^\top X_\ell^h.
\end{equation}
This matrix will also allow us to identify the layers
at which a prescribed moment threshold is reached.
With this notation at hand, we can state the following lemma.

\begin{lemma}[Euler approximation before blow-up]
\label{prop:lin-transformer-transfer}
Under the assumptions of Proposition~\ref{prop:lin-transformer-lift},
suppose additionally that $\alpha\in C^1([0,\infty))$.
Let $(X,Y)$ be the solution given by
\eqref{eq:lin-transformer-lift}, and let $T_{\max}$ be
the maximal existence time of $(S,P)$.
Let $(X_\ell^h,Y_\ell^h)$ be the iterates of
\eqref{eq:lin-transformer-layers} with initial values
$(X_0^h,Y_0^h)=(X_0,0)$. Then, for every $0<T<T_{\max}$, there are constants
$C_T<\infty$ and $h_T>0$ such that
\begin{equation}\label{eq:lin-transformer-euler-error}
\max_{0\leq\ell h\leq T}
\left(
\|Z_\ell^h-Z(\ell h)\|
+
\|S_\ell^h-S(\ell h)\|_{\mathrm F}
\right)
\leq C_T h
\qquad(0<h<h_T).
\end{equation}
\end{lemma}
The proof of Lemma~\ref{prop:lin-transformer-transfer}
is given in Appendix~\ref{app:lin-continuation-proofs}. Next, we use the Euler error estimate
\eqref{eq:lin-transformer-euler-error} to relate the continuous
blow-up time $T_{\max}$ to the depth at which large token
norms occur in the discrete attention layers.

\paragraph{Activation thresholds and effective depth.}
The aim of this paragraph is to use the continuous
dynamics \eqref{eq:lin-transformer-flow} to predict when
large token norms occur in the discrete attention layers
\eqref{eq:lin-transformer-layers}.
Using the Euler estimate in
Lemma~\ref{prop:lin-transformer-transfer}, we first
estimate the number of layers needed to reach a prescribed
second-moment threshold in
Theorem~\ref{thm:lin-transformer-depth}.
We then derive the consequences for token growth and
compare networks whose effective depth lies below
or above the continuous blow-up time.
To this end, we first fix the assumptions and notation.

Let $(S,P)$, $(X,Y)$, and $(X_\ell^h,Y_\ell^h)$ be
the solutions and iterates from
Lemma~\ref{prop:lin-transformer-transfer}.
Denote the maximal existence time of $(S,P)$ by
$T_{\max}$.
Throughout this paragraph, assume additionally that
$A,V,S_0$ commute, that $V$ has a positive eigenvalue,
and that $0\leq\alpha_t\leq\bar\alpha<\infty$.

Since $V$ has a positive eigenvalue,
Theorem~\ref{thm:lin-bounded-damping-blowup} gives
$T_{\max}<\infty$.
We choose the common unit eigenvector $e_{i_*}$ provided
by this theorem and, using
Proposition~\ref{prop:lin-transformer-lift}, write its directional
second moment as
\begin{equation}\label{eq:lin-directional-moment}
s(t)\coloneqq e_{i_*}^\top S_t e_{i_*}
=\frac1N\|X(t)e_{i_*}\|_2^2.
\end{equation}

The function $s$ is continuous because $S$ is a $C^1$
solution of \eqref{eq:lin-accelerated-ode}.
Its strict increase and divergence as $t\uparrow T_{\max}$
follow from
Theorem~\ref{thm:lin-bounded-damping-blowup}.
Setting $s_*\coloneqq s(0)$, we therefore obtain, for every
$R>s_*$, a unique continuous crossing time $t_R$ satisfying
\begin{equation}\label{eq:lin-continuous-crossing-time}
s(t_R)=R,
\qquad 0<t_R<T_{\max}.
\end{equation}
For the discrete layers \eqref{eq:lin-transformer-layers},
the directional second moment corresponding to
\eqref{eq:lin-directional-moment} is
\[
\frac1N\|X_\ell^h e_{i_*}\|_2^2.
\]
To determine how many layers are needed to reach the
same threshold $R$, we define the first threshold layer by
\begin{equation}\label{eq:lin-layer-threshold}
L_R(h)\coloneqq
\min\left\{
\ell\in\mathbb N_0:
\frac1N\|X_\ell^h e_{i_*}\|_2^2\geq R
\right\},
\end{equation}
with $L_R(h)=\infty$ if no layer reaches $R$.
Since each layer represents a time step $h$, the
corresponding accumulated time is $hL_R(h)$.
Comparing it with $t_R$ tells us how accurately the
continuous dynamics predict the first threshold layer.

\begin{theorem}[Activation thresholds]
\label{thm:lin-transformer-depth}
Under the preceding assumptions, let $t_R$ and $L_R(h)$
be defined by \eqref{eq:lin-continuous-crossing-time}
and \eqref{eq:lin-layer-threshold}.\\

\noindent Then for every fixed $R>s_*$, the layer index $L_R(h)$ is
finite for all sufficiently small $h>0$, and
\begin{equation}\label{eq:lin-transformer-threshold-depth}
hL_R(h)=t_R+O_R(h).
\end{equation}
Thus, in particular, \[
\lim_{R\to\infty}\lim_{h\downarrow0}hL_R(h)=T_{\max}.
\]
\end{theorem}
The proof of Theorem~\ref{thm:lin-transformer-depth}
is given in Appendix~\ref{app:lin-continuation-proofs}.
Before moving on, we explain how this result connects
continuous moment growth to the depth of the discrete
network.
\begin{remark}[Predicting depth from continuous dynamics]
\label{rem:fin}
Theorem~\ref{thm:lin-transformer-depth} gives a quantitative
prediction for the discrete network: for each fixed
threshold $R$, the continuous crossing time $t_R$
predicts the first threshold layer as $t_R/h$, up to
a bounded number of layers.

As the discretization is refined and then the threshold
is increased, the effective depths of these first
threshold layers approach $T_{\max}$.
Thus, the continuous blow-up time predicts where
increasingly large token norms occur in the discrete
layers, even for a fixed number of tokens.
\end{remark}

Next, we express this conclusion in terms of individual
tokens.
Write $(x_{j,\ell}^h)^\top$ for row $j$ of $X_\ell^h$.
Reaching a sufficiently large second-moment threshold
forces at least one of these tokens to have a large norm.
\begin{cor}[Arbitrarily large token norms]
\label{cor:lin-transformer-amplification}
Under the assumptions of
Theorem~\ref{thm:lin-transformer-depth},
for every $M>0$, all sufficiently small $h>0$ admit
a layer $\ell$ with $\ell h<T_{\max}$ such that
\begin{equation}\label{eq:lin-transformer-amplification}
\max_{1\leq j\leq N}\|x_{j,\ell}^h\|_2>M.
\end{equation}
\end{cor}
The proof of Corollary~\ref{cor:lin-transformer-amplification}
is given in Appendix~\ref{app:lin-continuation-proofs}.
The following remark explains its implication for token
growth in the discrete network.

\begin{remark}[Continuous blow-up appears in discrete layers]
\label{rem:2}
Corollary~\ref{cor:lin-transformer-amplification} shows that
any prescribed token-norm bound is exceeded at a layer
with accumulated time $\ell h<T_{\max}$, provided the
step size is sufficiently small.
Thus, the growth predicted by the continuous dynamics
also occurs in the discrete layers, even for a fixed
number of tokens.
\end{remark}
We now compare networks with different numbers of layers.
For $L\in\mathbb N$ layers and step size $h>0$, we call
$T=Lh$ the \emph{effective depth}.
Keeping $T$ fixed while increasing $L$ therefore means
choosing $h=T/L$.
To record the largest token norm throughout such a
network, set
\begin{equation}\label{eq:lin-maximal-activation}
B_L(T)\coloneqq
\max_{0\leq\ell\leq L}\max_{1\leq j\leq N}
\|x_{j,\ell}^{T/L}\|_2.
\end{equation}
Corollary~\ref{cor:lin-transformer-amplification} implies
that, for fixed $T\geq T_{\max}$, these maxima diverge
as the number of layers tends to infinity.

\begin{cor}[Critical and supercritical effective depth]
\label{cor:lin-transformer-critical-depth}
Under the assumptions of
Theorem~\ref{thm:lin-transformer-depth},
for every fixed $T\geq T_{\max}$,
\begin{equation}\label{eq:lin-transformer-critical-depth}
\lim_{L\to\infty}B_L(T)=\infty.
\end{equation}
\end{cor}
The proof of Corollary~\ref{cor:lin-transformer-critical-depth}
is given in Appendix~\ref{app:lin-continuation-proofs}.
We now turn to networks whose fixed effective depth
lies below $T_{\max}$. In this regime, the Euler estimate
in Lemma~\ref{prop:lin-transformer-transfer} shows that
increasing the number of layers improves the approximation
while token norms remain bounded, as the following
corollary states.

\begin{cor}[Subcritical effective depth]
\label{thm:lin-transformer-subcritical-depth}
Under the assumptions of
Theorem~\ref{thm:lin-transformer-depth},
let $Z(t)=(X(t),Y(t))$ and
$Z_\ell^h=(X_\ell^h,Y_\ell^h)$, equipped with the norm
\[
\|(X,Y)\|^2=\|X\|_{\mathrm F}^2+\|Y\|_{\mathrm F}^2.
\]
For every fixed $0<T<T_{\max}$, choosing $h=T/L$ gives
\begin{equation}\label{eq:lin-transformer-subcritical-error}
\max_{0\leq\ell\leq L}
\|Z_\ell^{T/L}-Z(\ell T/L)\|
=O(L^{-1}),
\qquad
\limsup_{L\to\infty}B_L(T)<\infty.
\end{equation}
\end{cor}
The proof of Corollary~\ref{thm:lin-transformer-subcritical-depth}
is given in Appendix~\ref{app:lin-continuation-proofs}.
Before moving on, we explain what this result means
for refining the discrete network.

\begin{remark}[Refining the network below the critical depth]
\label{nextrem}
At a fixed effective depth $T<T_{\max}$, increasing the
number of layers while choosing $h=T/L$ improves the
approximation of the continuous dynamics.
The token norms remain uniformly bounded across all
layers as $L$ increases.
Thus, below $T_{\max}$, finer discretization gives a
more accurate approximation of a bounded trajectory.
\end{remark}

These results lead us to ask how normalization and residual
scaling affect the relation between effective depth and
token growth.
\paragraph{Relation to normalization and depth control.}
Under the assumptions of
Corollary~\ref{cor:lin-transformer-critical-depth},
bounded nonnegative damping does not remove the
critical effective-depth threshold for the unnormalized
updates \eqref{eq:lin-transformer-layers}.

Other approaches modify the architecture to control
growth. DeepNorm combines modified residual connections
with depth-dependent initialization to control model
updates during training \cite{WangEtAl2022DeepNet}.
The spherical attention model in
\cite[Sec.~2.2]{GLPR2025} instead constrains token
positions to the unit sphere, excluding token-norm
blow-up by construction.
Extending the present analysis to such modifications
would clarify how normalization and residual scaling
affect accelerated attention dynamics.

We close this section by summarizing how the continuous
analysis yields information about discrete attention layers.
Figure~\ref{fig:linear-attention-dynamics} shows the successive
reductions. Theorem~\ref{thm:lin-transformer-depth} connects
the first and last stages: continuous moment growth predicts
the layer at which a prescribed threshold is reached.


\part{Discretizations and Numerical Results}
\section{Discrete SympFormer architectures}\label{sec:Discrete SympFormer architectures}

The aim of this section is to turn the accelerated evolutions
derived in Subsection~\ref{subsec:softmax-particle-dynamics}
into explicit attention layers and to evaluate the performance of the resulting
architectures.
First, we specify the discrete updates.
Then, we explain their implementation in a decoder-only
transformer and compare predictive performance,
computational cost, and the effect of the correction.
We observe that the geometrically corrected discretizations
outperform the nanoGPT2-style baseline and a different accelerated transformer architecture
without geometric correction in terms of mean NLL at a fixed training-token budget in most settings,
most prominently in looped architectures.
However, the additional geometric correction increases computational time
compared with the other methods.

\subsection{Time Discretizations of Riemannian Accelerated Gradient Flows}
\label{subsec:attention-discretizations}
We begin by discretizing
\eqref{eq:acc_softmax_damped_Newton}.
The purpose is to obtain explicit updates for the token
positions and an auxiliary variable carrying information
between consecutive layers.

Let $h>0$, set $t_k=kh$, and let
$\mathbf U^{(k)}$ approximate the physical velocity
$\dot{\mathbf X}(t_k)$.
We use the displacement variable
\begin{equation}\label{eq:displacement}
\mathbf V^{(k)}\coloneqq h\mathbf U^{(k)}.
\end{equation}
For a fixed-step discretization, the damping and force
coefficients are $\beta_k\coloneqq1-h\alpha(t_k)$ and $\gamma_k\coloneqq h^2$.

Following the sequential velocity-discretization
approach in \cite{AOBL2020}, we first combine the
previous displacement with the attention force.
We then update the position and carry the displacement
to the new position.
Using \eqref{eq:softmax_second_order_retraction}
and \eqref{eq:softmax_second_order_transport} gives
\begin{equation}\label{eq:Riemannian_HB_R2}
\begin{cases}
\widetilde{\mathbf V}^{(k+1)}
=
\beta_k\mathbf V^{(k)}
+\gamma_k\Attn(\mathbf X^{(k)}),\\[0.5ex]
\mathbf C^{(k)}
=
\Gamma_{\mathbf X^{(k)}}^g
\bigl(\widetilde{\mathbf V}^{(k+1)},
      \widetilde{\mathbf V}^{(k+1)}\bigr),\\[0.5ex]
\mathbf X^{(k+1)}
=
\mathbf X^{(k)}
+\widetilde{\mathbf V}^{(k+1)}
-\frac12\mathbf C^{(k)},\\[0.5ex]
\mathbf V^{(k+1)}
=
\widetilde{\mathbf V}^{(k+1)}-\mathbf C^{(k)}.
\end{cases}
\end{equation}

Alternatively, we can discretize the coordinate equation
\eqref{eq:acc_softmax_damped_Newton} directly.
Explicit Euler for the physical velocity, followed
by a position update using the new velocity, gives
\begin{equation}\label{eq:Riemannian_HB_coordinate}
\begin{cases}
\mathbf V^{(k+1)}
=
\beta_k\mathbf V^{(k)}
+\gamma_k\Attn(\mathbf X^{(k)})
-\Gamma_{\mathbf X^{(k)}}^g
\bigl(\mathbf V^{(k)},\mathbf V^{(k)}\bigr),\\[0.5ex]
\mathbf X^{(k+1)}
=
\mathbf X^{(k)}+\mathbf V^{(k+1)}.
\end{cases}
\end{equation}
The connection is evaluated at the previous displacement
in \eqref{eq:Riemannian_HB_coordinate}, whereas
\eqref{eq:Riemannian_HB_R2} uses the proposed displacement
$\widetilde{\mathbf V}^{(k+1)}$.

Omitting the connection correction in either scheme gives
\begin{equation}\label{eq:attention-discretization-control}
\begin{cases}
\mathbf V^{(k+1)}
=
\beta_k\mathbf V^{(k)}
+\gamma_k\Attn(\mathbf X^{(k)}),\\[0.5ex]
\mathbf X^{(k+1)}
=
\mathbf X^{(k)}+\mathbf V^{(k+1)}.
\end{cases}
\end{equation}

\subsection{SympFormer Architecture}

Next, we adapt
\eqref{eq:Riemannian_HB_R2} and
\eqref{eq:Riemannian_HB_coordinate}
to a decoder-only transformer for next-token prediction.
The implementation follows a causal, multi-head
nanoGPT-style architecture.\footnote{The implementation
is available at
\url{https://github.com/ViktorAJStein/SympFormer}.}
Our construction modifies the attention update and
retains the MLP update of the momentum baseline.

Let $\mathbf X^{(k)}\in\R^{N\times d}$ contain the
token states entering layer $k$, and let
$\mathbf V^{(k)}\in\R^{N\times d}$ denote the auxiliary
velocity stream.
Additional state variables also appear in augmented neural ODEs and second-order latent models, where they serve expressivity and dynamical-modeling purposes, respectively \cite{DDT2019,YHL2019}.
The intermediate index $k+\frac12$ denotes the output
of the attention update, before applying the MLP.
We write $\Attn_k$ and $\MLP_k$ for the attention and
MLP maps in layer $k$.
The maps $\LN_k^A$ and $\LN_k^M$ are the corresponding
position LayerNorms \cite{brody2023expressivity}, while $\LN_k^V$ is the velocity
LayerNorm shared by the two updates within that layer.

We compare our approaches with the  following architectures from the literature:

\paragraph{1. Ordinary transformer baseline B.}
It has the layer update
\begin{equation}
\label{eq:baseline}
\begin{cases}
\mathbf X^{(k+\frac12)}
=
\mathbf X^{(k)}
+\Attn_k\bigl(\LN_k^A(\mathbf X^{(k)})\bigr),\\[0.5ex]
\mathbf X^{(k+1)}
=
\mathbf X^{(k+\frac12)}
+\MLP_k\bigl(\LN_k^M(\mathbf X^{(k+\frac12)})\bigr).
\end{cases}
\end{equation}

\paragraph{2. Momentum baseline with YuriiFormer Y
\cite{ZPR2026}.}
Its coefficients $\beta_k^A,\beta_k^M$ control the
previous velocity contribution, while
$\gamma_k^A,\gamma_k^M$ scale the attention and MLP
outputs.
The coefficients $\mu_k^A,\mu_k^M$ specify the
look-ahead points
\[
\mathbf Q_k^A
\coloneqq
\mathbf X^{(k)}+\mu_k^A\mathbf V^{(k)},
\qquad
\mathbf Q_k^M
\coloneqq
\mathbf X^{(k+\frac12)}
+\mu_k^M\mathbf V^{(k+\frac12)}.
\]
The layer update is
\begin{equation}
\label{eq:YuriiFormer}
\begin{cases}
\mathbf V^{(k+\frac12)}
=
\LN_k^V\!\left(
\beta_k^A\mathbf V^{(k)}
+\gamma_k^A\Attn_k\bigl(\LN_k^A(\mathbf Q_k^A)\bigr)
\right),\\[0.5ex]
\mathbf X^{(k+\frac12)}
=
\mathbf X^{(k)}+\mathbf V^{(k+\frac12)},\\[0.5ex]
\mathbf V^{(k+1)}
=
\LN_k^V\!\left(
\beta_k^M\mathbf V^{(k+\frac12)}
+\gamma_k^M\MLP_k\bigl(\LN_k^M(\mathbf Q_k^M)\bigr)
\right),\\[0.5ex]
\mathbf X^{(k+1)}
=
\mathbf X^{(k+\frac12)}+\mathbf V^{(k+1)}.
\end{cases}
\end{equation}
The attention and MLP updates have independently learned
coefficients.

\paragraph{3. YuriiFormer variant $\mathbf{H_0}$.} Setting $\mu_k^A=0$ yields
\begin{equation}
\label{eq:H0}
\begin{cases}
\mathbf V^{(k+\frac12)}
=
\LN_k^V\!\left(
\beta_k^A\mathbf V^{(k)}
+\gamma_k^A\Attn_k\bigl(\LN_k^A(\mathbf X^{(k)})\bigr)
\right),\\[0.5ex]
\mathbf X^{(k+\frac12)}
=
\mathbf X^{(k)}+\mathbf V^{(k+\frac12)},\\[0.5ex]
\mathbf V^{(k+1)}
=
\LN_k^V\!\left(
\beta_k^M\mathbf V^{(k+\frac12)}
+\gamma_k^M\MLP_k\bigl(\LN_k^M(\mathbf Q_k^M)\bigr)
\right),\\[0.5ex]
\mathbf X^{(k+1)}
=
\mathbf X^{(k+\frac12)}+\mathbf V^{(k+1)}.
\end{cases}
\end{equation}

Now, we introduce the two SympFormer architectures based on the previous discussion.
Both  replace the exact connection
in \eqref{eq:Riemannian_HB_R2} by
\eqref{eq:connection_proxy}, as detailed in Subsection~\ref{subsec:connection-proxy}

\paragraph{4. Coordinate SympFormer $\mathbf{H_1}$.}
Adapting the coordinate variant \eqref{eq:Riemannian_HB_coordinate} yields the following $\mathrm H_1$-architecture:
\begin{equation}\label{eq:H1}
\begin{cases}
\mathbf V^{(k+\frac12)}
=
\LN_k^V\!\left(
\beta_k^A\mathbf V^{(k)}
+\gamma_k^A\Attn_k\bigl(\LN_k^A(\mathbf X^{(k)})\bigr)
-\widehat\Gamma_{\LN_k^A(\mathbf X^{(k)})}
\bigl(\mathbf V^{(k)},\mathbf V^{(k)}\bigr)
\right),\\[0.5ex]
\mathbf X^{(k+\frac12)}
=
\mathbf X^{(k)}+\mathbf V^{(k+\frac12)}.
\end{cases}
\end{equation}

\paragraph{5. $R_2$ SympFormer $\mathbf{R_2}$.}
For the $\mathrm R_2$ variant, we apply velocity
LayerNorm to both corrected displacements:
\begin{equation}\label{eq:R2}
\begin{cases}
\widetilde{\mathbf V}^{(k+\frac12)}
=
\beta_k^A\mathbf V^{(k)}
+\gamma_k^A
\Attn_k\bigl(\LN_k^A(\mathbf X^{(k)})\bigr),\\[0.5ex]
\mathbf C^{(k)}
=
\widehat\Gamma_{\LN_k^A(\mathbf X^{(k)})}
\bigl(
\widetilde{\mathbf V}^{(k+\frac12)},
\widetilde{\mathbf V}^{(k+\frac12)}
\bigr),\\[0.5ex]
\mathbf X^{(k+\frac12)}
=
\mathbf X^{(k)}
+\LN_k^V\!\left(
\widetilde{\mathbf V}^{(k+\frac12)}
-\frac12\mathbf C^{(k)}
\right),\\[0.5ex]
\mathbf V^{(k+\frac12)}
=
\LN_k^V\!\left(
\widetilde{\mathbf V}^{(k+\frac12)}
-\mathbf C^{(k)}
\right).
\end{cases}
\end{equation}

\paragraph{6. Look-ahead SympFormers $\mathrm H_1^{\mathrm{LA}}$ and $\mathrm R_2^{\mathrm{LA}}$.}
These architectures evaluate the attention force and connection proxy
at
\[
\mathbf Z
=
\LN_k^A\bigl(
\mathbf X^{(k)}+\mu_k^A\mathbf V^{(k)}
\bigr).
\]

All SympFormer variants use the MLP update from
\eqref{eq:YuriiFormer}.
Setting the connection correction to zero in $\mathrm{H}_1$ or $\mathrm R_2$ recovers
the attention update of $\mathrm{H_0}$.
Table~\ref{tab:architecture-overview} summarizes the key differences between the architectures we consider.
\begin{table}[h!]
\centering
\begin{tabular}{lcccc}
\toprule
Architecture
& auxiliary state
& attention look-ahead
& geometric correction \\
\midrule
$\mathrm{B}$ \eqref{eq:baseline}
& no & no & no \\
$\mathrm{Y}$ \eqref{eq:YuriiFormer}
& yes & yes & no \\
$\mathrm{H}_0$ \eqref{eq:H0}
& yes & no & no \\ \midrule
$\mathrm{H}_1$ \eqref{eq:H1}
& yes & no & yes \\
$\mathrm{R_2}$ \eqref{eq:R2}
& yes & no & yes \\
$\mathrm H_1^{\mathrm{LA}}$
& yes & yes & yes \\
$\mathrm R_2^{\mathrm{LA}}$
& yes & yes & yes\\
\bottomrule
\end{tabular}
\caption{Architectures for comparison.}
\label{tab:architecture-overview}
\end{table}

The $\mathrm H_1$ and $\mathrm R_2$ architectures differ from the fixed-step discretizations in subsection~\ref{subsec:attention-discretizations} in several respects.
First, both replace the exact Levi--Civita connection by the
computable proxy $\widehat\Gamma$ defined below. Second, LayerNorm alters the updates and the displacement scaling in \eqref{eq:displacement}. In particular, the proxy does not include derivatives of LayerNorm.
Finally, causal multi-head attention, learned layer-dependent update coefficients, and an additional MLP momentum step replace the fixed-force, fixed-step model analyzed above.
Thus, the experiments test the usefulness of
a geometrically motivated correction within a trained architecture,
not the numerical accuracy of a solver for the Newton equation.
Training of the parameters of the architectures minimizes next-token negative log-likelihood (NLL), not the interaction energy appearing in the dynamical analysis.

\subsection{Causal Multi-Head Realization and Connection Proxy}
\label{subsec:connection-proxy}

To introduce the geometric correction $\Gamma$ in \eqref{eq:Riemannian_HB_coordinate} and \eqref{eq:Riemannian_HB_R2}, we adapt the true expression
\eqref{eq:softmax_connection_contracted}
to the attention weights used by the model.
We start with the softmax geometry.

The exact formula concerns the symmetric particle
metric developed above.
For causal multi-head attention, we retain its
algebraic form, set the quadratic form to $B=I_d$,
and obtain a connection proxy $\widehat\Gamma$.

We define this proxy for one attention block,
suppressing the layer index.
Let $\mathbf Z$ have rows $z_i^\top$, and let
$\mathbf U$ have rows $u_i^\top$.
Here $\mathbf Z$ is the normalized token matrix
at which attention is evaluated, and $\mathbf U$
is the direction supplied by the update.
For head $r\in\{1,\ldots,H\}$, define
\[
q_i^{(r)}=W_Q^{(r)}z_i,
\qquad
k_j^{(r)}=W_K^{(r)}z_j,
\qquad
w_j^{(r)}=W_V^{(r)}z_j.
\]
Writing $d_h$ for the head dimension, the scores
and causal attention weights are
\begin{equation}\label{eq:numerical-causal-weights}
\ell_{ij}^{(r)}
=
\frac{\langle q_i^{(r)},k_j^{(r)}\rangle}{\sqrt{d_h}},
\qquad
P_{ij}^{(r)}
=
\frac{\exp(\ell_{ij}^{(r)})}
     {\sum_{m\leq i}\exp(\ell_{im}^{(r)})},
\qquad j\leq i.
\end{equation}
The directional score derivative is
\[
\dot\ell_{ij}^{(r)}[\mathbf U]
=
\frac{
\langle W_Q^{(r)}u_i,k_j^{(r)}\rangle
+
\langle q_i^{(r)},W_K^{(r)}u_j\rangle
}{\sqrt{d_h}}.
\]
We then set
\[
a_i(\mathbf U)
\coloneqq
\frac1H\sum_{r=1}^H\sum_{j\leq i}
P_{ij}^{(r)}\dot\ell_{ij}^{(r)}[\mathbf U],
\qquad
s_i(\mathbf U)\coloneqq\|u_i\|_2^2.
\]
For the output projection $W_O$, the ordinary
attention output and the additional weighted
value aggregation are
\[
\Attn_i(\mathbf Z)
=
W_O\operatorname{Concat}_{r=1}^H
\left(\sum_{j\leq i}P_{ij}^{(r)}w_j^{(r)}\right),
\]
\[
\Attn_i^{(s)}(\mathbf Z,\mathbf U)
\coloneqq
W_O\operatorname{Concat}_{r=1}^H
\left(
\sum_{j\leq i}P_{ij}^{(r)}s_j(\mathbf U)w_j^{(r)}
\right).
\]
With this notation, the proxy is
\begin{equation}\label{eq:connection_proxy}
\widehat\Gamma_{\mathbf Z}(\mathbf U,\mathbf U)_i
=
a_i(\mathbf U)u_i
-\frac12\left(
s_i(\mathbf U)\Attn_i(\mathbf Z)
+\Attn_i^{(s)}(\mathbf Z,\mathbf U)
\right).
\end{equation}
All terms use the weights in
\eqref{eq:numerical-causal-weights}.
The additional computation includes the directional
projections and score contractions, together with the
weighted value aggregation.
Although no further softmax evaluation or matrix inverse
is required, these operations incur additional training
cost, which we report below.

The proxy for the connection for the other types of attention evaluated below is described on the implementation's GitHub page \url{https://github.com/ViktorAJStein/SympFormer}.

\section{Numerical Results}
\label{sec:numerics}
We evaluate the models on TinyStories (TS) and OpenWebText (OWT) to address whether the connection-inspired
correction improves next-token prediction compared to the baseline and an otherwise matched momentum architecture without this correction.
These experiments probe the practical relevance of the geometric derivation.
We report validation NLL (lower is better), together with
computational cost. Throughout, bold marks the lowest observed
mean within each dataset and layout.

\subsection{Softmax Attention}
\label{subsec:softmax-experiments}

The following softmax comparison uses four layers,
four heads, width $64$, context length $128$, and
$99{,}991{,}552$ training tokens.
We report five paired seeds.
Table~\ref{tab:softmax_fiveway_uncertainty}
gives the final validation NLL and computational cost.
Implementation details are given in
Appendix~\ref{section:Implementation}.

\begin{table}[htbp]
\centering
\setlength{\tabcolsep}{6pt}
\renewcommand{\arraystretch}{1.2}
\resizebox{\textwidth}{!}{
\begin{tabular}{lrrrrr}
\toprule
Method
& TS NLL (SD)
& OWT NLL (SD)
& Params (M)
& GPUh (TS/OWT)
& Peak MiB\\
\midrule
$\mathrm{B}$ 
& $2.6593\;(0.0435)$
& $5.3949\;(0.0373)$
& $3.424832$
& $0.2530/0.3560$
& $926.9$\\
$\mathrm{Y}$ 
& $2.5748\;(0.0424)$
& $5.2674\;(0.0407)$
& $6.652760$
& $0.3560/0.4380$
& $982.3$\\
$\mathrm{H}_0$ 
& $2.5782\;(0.0418)$
& $5.2580\;(0.0425)$
& $6.652756$
& $0.3460/0.4010$
& $982.3$\\ \midrule
$\mathrm{H}_1$ 
& $2.5730\;(0.0391)$
& $5.2586\;(0.0407)$
& $6.652756$
& $0.5240/0.5770$
& $999.6$\\
$\mathrm{R}_2$ 
& $\mathbf{2.5720}\;(0.0358)$
& $\mathbf{5.2554}\;(0.0396)$
& $6.652756$
& $0.5420/0.6670$
& $1001.6$\\
\bottomrule
\end{tabular}}
\caption{Final validation NLL for softmax
attention, together with parameter counts, reported
GPU-hours per run, and peak allocated GPU memory.
NLL values are means and sample standard deviations
over five paired seeds.}
\label{tab:softmax_fiveway_uncertainty}
\end{table}

The $\mathrm R_2$ variant has the lowest mean NLL on both
datasets in this comparison. Comparing $\mathrm H_1$ and
$\mathrm R_2$ with the correction-off control $\mathrm H_0$
examines the added value of the connection proxy.
The differences in mean NLL are small relative to the reported
seed-to-seed variability, and Table~\ref{tab:softmax_fiveway_uncertainty} alone does not establish
a statistically significant advantage.

Table~\ref{tab:softmax_fiveway_uncertainty} shows that
the geometry-corrected models have the same parameter count
as $\mathrm H_0$, but require more computation.
The higher parameter count comes primarily from the additional learned velocity embeddings, which add $d(T + \mathsf{V})$ parameters, where $d$ is the embedding dimension, $\mathsf{V}$ is the vocabulary size, and $T$ is the maximal context length.

\subsubsection{Parameter Sharing Across Depth}
\label{subsec:parameter-sharing}

In Table~\ref{tab:separate-vs-shared}, we compare $\mathrm H_1$ and $\mathrm R_2$ with their matched
correction-off control $\mathrm H_0$ and the native residual baseline
$\mathrm B$ when parameters are shared across depth, as in \emph{looped
transformers} \cite{DGVUK2019,YLNP2024,SZGGW2026,geiping2026scaling}.
The untied model uses eight independently parameterized blocks.
The tied model reuses two distinct trainable blocks four times.
Both execute eight logical block applications.
The token-state stream and the velocity stream, when present,
continue across repetitions without resets.

\begin{table}[htbp]
\centering
\setlength{\tabcolsep}{8pt}
\renewcommand{\arraystretch}{1.2}
\resizebox{\textwidth}{!}{
\begin{tabular}{lrrrrrr}
\toprule
& \multicolumn{2}{c}{TinyStories} & \multicolumn{2}{c}{OpenWebText} & \multicolumn{2}{c}{GPUh (TS/OWT)} \\
\cmidrule(lr){2-3} \cmidrule(lr){4-5} \cmidrule(lr){6-7}
Method & Untied & Tied & Untied & Tied & Untied & Tied \\
\midrule
$\mathrm B$ & $\mathbf{1.7940}\;(0.0056)$ & $1.9948\;(0.0063)$ & $4.5078\;(0.0053)$ & $4.8027\;(0.0055)$ & $0.5172/0.5240$ & $0.5117/0.5186$ \\
$\mathrm Y$ & $1.8179\;(0.0056)$ & $1.9886\;(0.0046)$ & $4.5281\;(0.0080)$ & $4.6888\;(0.0027)$ & $0.6029/0.6109$ & $0.5723/0.5801$ \\
$\mathrm H_0$ & $1.8233\;(0.0039)$ & $1.9955\;(0.0109)$ & $4.5297\;(0.0076)$ & $4.6929\;(0.0046)$ & $0.5952/0.6028$ & $0.5661/0.5732$ \\
\midrule
$\mathrm H_1$ & $1.8033\;(0.0075)$ & $\mathbf{1.9514}\;(0.0040)$ & $4.4929\;(0.0102)$ & $\mathbf{4.6536}\;(0.0067)$ & $0.8962/0.9014$ & $0.8647/0.8692$ \\
$\mathrm R_2$ & $1.7969\;(0.0070)$ & $1.9566\;(0.0081)$ & $\mathbf{4.4864}\;(0.0097)$ & $4.6553\;(0.0055)$ & $0.9026/0.9083$ & $0.8732/0.8811$ \\
\bottomrule
\end{tabular}}
\caption{Final validation NLL for untied and tied softmax models.
Values are means (sample SD) over ten paired seeds,
after $99{,}991{,}552$ tokens, with
eight heads, width $256$, and context length $512$.}
\label{tab:separate-vs-shared}
\end{table} 

Both $\mathrm H_1$ and $\mathrm R_2$ have lower mean NLL than the
matched control $\mathrm H_0$ in each layout and on both datasets.
Thus the correction-control improvement is also observed when the same
trainable blocks are applied repeatedly.
Native-baseline utility is more conditional: $\mathrm H_1$ has the
lowest mean NLL in both tied comparisons, whereas the untied minima
are $\mathrm B$ on TinyStories and $\mathrm R_2$ on OpenWebText.
Both corrected variants require more GPU time than $\mathrm H_0$.

\subsection{Linear Attention \& Other Attention Mechanisms}
\label{subsec:linear-experiments}
Next, we investigate linear attention and normalized-sigmoid and native-$L^2$ attention. 
The native-$L^2$ results use an SPD tensor proxy.
Table~\ref{tab:other_attention_endpoints} reports prediction
quality and computational cost using three paired seeds,
four layers, four heads, width $64$, and context
length $512$, and $99{,}991{,}552$ training
tokens.


\begin{table}[htbp]
\centering
\setlength{\tabcolsep}{8pt}
\renewcommand{\arraystretch}{1.2}
\resizebox{\textwidth}{!}{
\begin{tabular}{lrrrrr}
\toprule
Method
& TinyStories NLL (SD)
& OpenWebText NLL (SD)
& Params (M)
& GPUh (TS/OWT)
& Peak MiB\\
\midrule
\multicolumn{6}{l}{\textbf{Linear attention}}\\
\midrule
$\mathrm{B}$ 
& $3.0855\;(0.0139)$
& $5.7018\;(0.0020)$
& $3.449408$
& $0.2004/0.2103$
& $1715.2$\\
$\mathrm{Y}$ 
& $2.9540\;(0.0104)$
& $5.5713\;(0.0007)$
& $6.701912$
& $0.2369/0.2497$
& $1776.8$\\
$\mathrm{H}_0$
& $2.9628\;(0.0306)$
& $5.5746\;(0.0041)$
& $6.701908$
& $0.2336/0.2468$
& $1776.8$\\ \midrule
$\mathrm{H}_1$ 
& $2.9688\;(0.0315)$
& $5.5739\;(0.0237)$
& $6.701908$
& $1.0469/1.0576$
& $3604.9$\\
$\mathrm{R}_2$ 
& $\mathbf{2.9400}\;(0.0355)$
& $\mathbf{5.5539}\;(0.0073)$
& $6.701908$
& $1.0531/1.0643$
& $3608.0$\\
\midrule
\multicolumn{6}{l}{\textbf{Normalized-sigmoid attention}}\\
\midrule
$\mathrm{B}$ 
& $2.6958\;(0.0019)$
& $5.5559\;(0.0112)$
& $3.449408$
& $0.1792/0.2370$
& $1460.8$\\
$\mathrm{Y}$ 
& $\mathbf{2.6238}\;(0.0090)$
& $5.4933\;(0.0254)$
& $6.701912$
& $0.2158/0.2746$
& $1522.6$\\
$\mathrm{H}_0$ 
& $2.6386\;(0.0303)$
& $5.4827\;(0.0395)$
& $6.701908$
& $0.2111/0.2703$
& $1522.6$\\
\midrule
$\mathrm{H}_1$ 
& $2.6655\;(0.0403)$
& $\mathbf{5.4264}\;(0.0140)$
& $6.701908$
& $0.3499/0.4037$
& $1739.6$\\
$\mathrm{R}_2$ 
& $2.6596\;(0.0502)$
& $5.4404\;(0.0156)$
& $6.701908$
& $0.3604/0.4117$
& $1743.6$\\
\midrule
\multicolumn{6}{l}{\textbf{Native-L2 SPD-proxy attention}}\\
\midrule
$\mathrm{B}$
& $2.6698\;(0.0127)$
& $5.5248\;(0.0088)$
& $3.449408$
& $0.1768/0.2435$
& $1403.5$\\
$\mathrm{Y}$ 
& $\mathbf{2.5740}\;(0.0193)$
& $5.4355\;(0.0102)$
& $6.701912$
& $0.2290/0.2835$
& $1466.1$\\
$\mathrm{H}_0$
& $2.5878\;(0.0153)$
& $5.4296\;(0.0171)$
& $6.701908$
& $0.2216/0.2789$
& $1466.1$\\
\midrule
$\mathrm{H}_1$ 
& $2.5923\;(0.0025)$
& $5.4180\;(0.0146)$
& $6.701908$
& $0.5709/0.5861$
& $1909.5$\\
$\mathrm{R}_2$ 
& $2.5892\;(0.0071)$
& $\mathbf{5.4099}\;(0.0138)$
& $6.701908$
& $0.5829/0.5932$
& $1913.6$\\
\bottomrule
\end{tabular}}
\caption{Final validation NLL for the linear, normalized-sigmoid, and native-L2
SPD-proxy attention comparisons, together with parameter counts, mean
allocated GPU-hours per run, and maximum peak allocated
memory over seeds.
}
\label{tab:other_attention_endpoints}
\end{table}

The results depend on the attention mechanism and dataset.
$\mathrm R_2$ has the lowest mean NLL for linear attention
on both datasets and for native-$L^2$ attention on OpenWebText,
whereas other methods have the lowest means in the remaining
comparisons. In particular, the geometric correction is not
uniformly beneficial. Its additional GPU-hour and memory costs
must be considered alongside the observed NLL differences.
\newpage
\section{Related results}\label{sec:Related results}

\emph{From continuum blow-up to discrete depth.}
The authors of~\cite{CACP2025} establish finite-time
blow-up for certain Gaussian solutions of continuum
attention equations.
For the unnormalized, shared-weight inertial
linear-attention model considered here,
Theorem~\ref{thm:lin-transformer-depth} predicts the first
layer reaching a prescribed second-moment threshold,
up to a bounded error in the layer index as the step
size decreases.
To the best of our knowledge, this is the first quantitative
connection between continuum moment blow-up and
threshold-crossing depth in discrete attention networks
with a fixed number of tokens.
Corollary~\ref{cor:lin-transformer-amplification}
translates this prediction into arbitrarily large token
norms.
Moreover, the continuous blow-up time determines a critical
effective depth, defined as the number of layers multiplied
by the step size.
Under refinement at fixed effective depth, the maximum
token norm over all layers remains bounded below this
critical value and diverges at or above it
(Corollaries~\ref{thm:lin-transformer-subcritical-depth}
and~\ref{cor:lin-transformer-critical-depth}).


\emph{Geometry and differential calculus on density spaces.}
In ~\cite{O2001}, Otto introduced
a formal Riemannian calculus for density evolutions on
$\R^d$, interpreting diffusion equations as gradient
flows in Wasserstein geometry; see also \cite{JKO1998}.
Generalizations based on non-linear mobilities were
developed in \cite{DNS2009,CLSS2010,lisini2007characterization}.
Hamiltonian forms on Wasserstein space are introduced in \cite{GKP2011}.
Our local-mobility framework specifies the analytic
domains required by Otto--Wasserstein and
softmax--Wasserstein geometry for smooth densities
with potentially unbounded support.
Theorem~\ref{prop:LM_operator_invertible} identifies
weighted function spaces on which the mobility operator
is invertible, and
Corollary~\ref{prop:LM_metric_well_defined} establishes
the induced metric.

For the geometry associated with Stein variational
gradient descent~\cite{LW2016},
\cite{DNS2023} and \cite{NR2023} construct
restricted tangent and cotangent spaces assuming
the kernel is integrally strictly positive definite.
The bilinear kernel of linear attention falls outside
that setting.
Lemma~\ref{lem:kernel} characterizes the infinite-dimensional
kernel of its mobility operator, leading to the quotient
construction in Theorem~\ref{prop:isomorphism}.
Proposition~\ref{prop:stein-not-surjective} addresses
a separate issue: even an integrally strictly positive
definite kernel need not yield surjectivity onto the
full smooth tangent space.
This obstruction concerns that larger space, rather
than the restricted spaces constructed in \cite{DNS2023}.

In both settings, we develop a differential calculus
without assuming that the admissible density classes
form smooth manifolds.
Lemmas~\ref{lem:regular-hamiltonian-representatives}
and~\ref{lem:regular-stein-hamiltonian-representatives}
identify admissible representatives of Hamiltonian
variations; the calculus is developed in
Appendices~\ref{app:diif_in_in_local_mobility_calculus}
and~\ref{app:diff_stein_geom}.

\emph{Inertial dynamics of probability distributions.}
Acceleration on spaces of probability measures has been
studied through optimal-control formulations
\cite{TM2019}, formal Hamiltonian equations on density
spaces \cite{WL2022}, and phase-space and Hilbert-space
approaches \cite{CLTW2025,TLG2026}.
The formal density-space construction requires an
invertible metric operator and meaningful variations;
the frameworks developed here supply precise conditions
for these operations in the attention setting.

For bilinear Stein kernels, Liu et al.~\cite{LGBP2023}
and Stein and Li~\cite{SL2025} derive and analyze
Gaussian covariance equations, including inertial
dynamics, for sampling problems governed by relative
entropy.
Our inertial linear-attention equations follow the same
formal Hamiltonian construction with the attention energy
in place of relative entropy.
Lemma~\ref{lem:lin-canonical-hamiltonian} and
Proposition~\ref{thm:lin-newton} establish their
Hamiltonian and Newton formulations.
Theorem~\ref{thm:lin-correspondence} proves strong
well-posedness for general admissible initial densities
and quadratic initial potentials.
The resulting matrix correspondence enables the
continuation and global-existence analysis in
Proposition~\ref{prop:lin-continuation} and the
finite-time moment blow-up result under bounded damping
in Theorem~\ref{thm:lin-bounded-damping-blowup}.

For softmax attention,
Theorem~\ref{thm:sm-hamiltonian} identifies the
additional domain and compatibility conditions required
for regular Hamiltonian derivatives.
Both this result and the first-order gradient identity
are statements at admissible states; constructing
trajectories and proving preservation of their domains
require a separate analysis.

\emph{Momentum and the construction of attention layers.}
Dynamical interpretations have guided neural architectures
through numerical discretizations \cite{LLHSDQWL2019}
and modified attention normalizations \cite{SABP2022}.
Momentum has been incorporated into residual networks
\cite{SABP2021}, heavy-ball and Nesterov neural ODEs
\cite{XSJNBOW2021,NNVOV2022}, and linear transformers
\cite{NBKOW2022,WXNO2022}.
Zimin et al.~\cite{ZPR2026} construct YuriiFormer from
a variational interpretation of attention and discrete
Nesterov-style updates.
Our contribution concerns the additional corrections
prescribed by the token-dependent attention metric.
Lemma~\ref{lem:sm-particle-connection} identifies the
connection correction used in the position and momentum
updates
\eqref{eq:softmax_second_order_retraction}
and~\eqref{eq:softmax_second_order_transport}.
Comparisons with matched momentum controls in
Tables~\ref{tab:softmax_fiveway_uncertainty}
and~\ref{tab:other_attention_endpoints} assess their
predictive benefit and computational cost.

\vspace{1cm}

\paragraph{Declaration of generative AI use}
During the research and preparation of this manuscript, the authors used
OpenAI's ChatGPT and Codex to assist with language and \LaTeX{} editing,
software implementation and debugging, and exploratory discussions of
mathematical formulations and numerical methods. AI tools also assisted
with the organization and presentation of computational results.
All AI-assisted material incorporated into the manuscript or the associated
code was reviewed and, where necessary, revised by the authors. The authors
are responsible for the mathematical statements and arguments, the
validation of the numerical experiments, the interpretation of the results,
and the conclusions of the work.

\vspace{5mm}

\noindent\textbf{Acknowledgments.}  The paper was started during a visit of W. Li to the
TU Berlin, supported by the SPP 2298 \enquote{Theoretical Foundations of Deep Learning} in Summer 2025. W. Li is supported by the AFOSR YIP award No. FA9550-23-1-0087, NSF RTG: 2038080, NSF FRG: 2245097, NSF DMS: 2601952, and the McCausland Faculty Fellowship at the University of South Carolina. GS acknowledges funding from the DFG STE 571/23-1 and as a CNRS fellow ambassador. 
V.S. thanks T. Laux for inviting him for a research stay in Heidelberg in May 2026 and acknowledges fruitful conversations with J. Chemseddine and R. Duong.
L.M. gratefully acknowledges support from the Institute
for Advanced Study and an anonymous donor during the
2026--2027 academic year, as well as funding from the
Deutsche Forschungsgemeinschaft (DFG, German Research
Foundation) -- 281869850 (RTG 2229), 390900948
(EXC-2181/1), and 281071066 (TRR 191).

\begin{center}
  \FundingLogos
  
  \vspace{0.5em}
  \begin{tcolorbox}\centering\small   
    V.S. acknowledges funding from the European Research Council (ERC) under the European Union’s Horizon Europe research and innovation programme (grant agreement No. 101198055, project acronym NEITALG).
    
    Views and opinions expressed are, however, those of the author(s) only and do not necessarily reflect those of the European Union or the European Research Council Executive Agency. Neither the European Union nor any of the aforementioned granting authorities can be held responsible for them.
  \end{tcolorbox}
\end{center}

\bibliographystyle{abbrv}
\bibliography{Bibliography}

\clearpage
\appendix
\section{Non-Surjectivity of  Stein's Operator on Compact Domains}\label{app:Non-surjectivity of the Stein operator on compact domains}
This subsection aims to prove that, even on the torus, the Stein
operator is not surjective. Thus, in general, there is no hope that the Stein
calculus performed, for example, in \cite{WL2022} can be made rigorous on the full smooth tangent space. For
more general closed manifolds, additional issues arise, as already mentioned.
In the following, let \(\mathbb T^d \coloneqq  \R^d / \mathbb Z^d\) be the flat
torus, equipped with normalized Lebesgue measure \(\d x\). Recall that the
space of probability densities \(\P^\infty(\mathbb T^d)\) on \(\mathbb T^d\)
is a smooth Fr\'{e}chet manifold modeled on \(C^\infty_0(\mathbb T^d)\), with
corresponding tangent space \( T_\rho\P^\infty(\mathbb T^d) = C^\infty_0(\mathbb T^d)\) and smooth dual \(C^\infty(\mathbb T^d) / \R\), which is the cotangent
space \(T_\rho^*\P^\infty(\mathbb T^d)\). Let
\(k\colon \mathbb T^d \times \mathbb T^d \to \R\) be a smooth, symmetric,
positive-definite kernel, and let \(\mathcal H_k\) denote its reproducing
kernel Hilbert space. Assume that \(k\) is \emph{integrally strictly positive
definite}, i.e.,
\[
    \int_{\mathbb T^d}\int_{\mathbb T^d}
    k(x,y) \d\mu(x) \d\mu(y) > 0
\]
for every nonzero finite signed Borel measure \(\mu\) on $\mathbb T^d$. This gives rise to the
Stein operator
\begin{equation}\label{eq:stein_operator}
        L_\rho^S\colon T_\rho^*\P^\infty(\mathbb T^d)
        \longrightarrow  T_\rho\P^\infty(\mathbb T^d), \qquad
        L_\rho^S\varphi \coloneqq  -\nabla_x\!\cdot\left(
            \rho(x)
            \int_{\mathbb T^d}
            k(x,y)\nabla_y\varphi(y)\rho(y) \d y
        \right).
\end{equation}
Its well-definedness follows from Stokes' theorem and the fact that
\(\mathbb T^d\) has no boundary.

\begin{proposition}[Non-surjectivity of the smooth Stein operator]
\label{prop:stein-not-surjective}
Let \(k \in C^\infty(\mathbb T^d \times \mathbb T^d)\) be symmetric and
integrally strictly positive definite, and let \(\rho \in \P^\infty(\mathbb
T^d)\). Then, the Stein operator \(L_\rho^S \) in \eqref{eq:stein_operator} is
injective, but not surjective.
\end{proposition}

Before delving into the proof, we highlight two implications of Proposition \ref{prop:stein-not-surjective}:
\begin{remark}[Failure of the full tangent-space metric]
By Proposition \ref{prop:stein-not-surjective} there exist smooth tangent vectors
\(\xi \in T_\rho\P^\infty(\mathbb T^d)\), so that there is no smooth potential
\(\varphi\) such that \(\xi = L_\rho^S\varphi\). Therefore the formula
\[
    G_\rho(\xi,\eta)
    =
    \int_{\mathbb T^d}\varphi_\xi(x)\eta(x) \d x,
    \qquad
    \xi = L_\rho^S\varphi_\xi,
\]
cannot define a Riemannian metric on all of
\(T_\rho\P^\infty(\mathbb T^d)\).
\end{remark}

\begin{remark}[The Stein calculus cannot be made rigorous on the full smooth tangent space]
More generally, a smooth Stein calculus that treats \(L_\rho^S\) as an
isomorphism from the cotangent space to the tangent space cannot be made
rigorous on all of \(\P^\infty(\mathbb T^d)\). Indeed, the non-surjectivity
of \(L_\rho^S\) prevents us from assigning a smooth potential to every
tangent vector. Hence, any construction requiring
\((L_\rho^S)^{-1}\xi\) is not well defined on the full tangent space.
\end{remark}

\begin{proof}[\textbf{Proof of Proposition \ref{prop:stein-not-surjective}}]
We use in the following identifications \( T_\rho\P^\infty(\mathbb T^d) \simeq  C^\infty_0(\mathbb T^d)\) and \( T_\rho^*\P^\infty(\mathbb T^d)\simeq C^\infty(\mathbb T^d) / \R\). Endow \(C^\infty_0(\mathbb T^d)\) with its standard Fr\'{e}chet topology. For
\(\varphi \in C^\infty(\mathbb T^d)\), set
\[
    P_0\varphi \coloneqq 
    \varphi -
    \frac{1}{\operatorname{vol}(\mathbb T^d)}
    \int_{\mathbb T^d}\varphi \d x.
\]
Every class \([\varphi] \in C^\infty(\mathbb T^d) / \R\) has the unique
zero-mean representative \(P_0\varphi\). The maps
\([\varphi] \mapsto P_0\varphi\) and \(\psi \mapsto [\psi]\) are continuous
inverses, so they identify the quotient with \(C^\infty_0(\mathbb T^d)\) as
Fr\'{e}chet spaces. Since \(L_\rho^S(\varphi + c) = L_\rho^S\varphi\), the
quotient operator is surjective (respectively, injective) exactly when
\(L_\rho^S \colon C^\infty_0(\mathbb T^d) \to C^\infty_0(\mathbb T^d)\) is. We use
this identification below. 
\\[1ex]
\underline{Injectivity}.  As \(k\) is integrally strictly positive definite and \(\rho>0\), it holds that \(L_\rho^S\varphi=0\) if and only if \(\varphi\in\mathbb{R}\). Thus, $L_\rho^S$ is injective.
\\
\underline{Non Surjectivity}. Periodic integration by parts gives
\[
    L_\rho^S\varphi(x)
    =
    \int_{\mathbb T^d} a_\rho(x,y)\varphi(y) \d y,
    \qquad
    a_\rho(x,y)
    \coloneqq 
    \nabla_x\!\cdot\!
    \left[
        \rho(x)\nabla_y\bigl(k(x,y)\rho(y)\bigr)
    \right].
\]
The kernel \(a_\rho\) is smooth; that is,
\(a_\rho \in C^\infty(\mathbb T^d \times \mathbb T^d)\). Every mixed
derivative \(\partial_x^\alpha\partial_y^\beta a_\rho\) is continuous. This
follows because \(k\) and \(\rho\) are smooth, and the displayed formula for
\(a_\rho\) uses only products and finitely many derivatives. Since
\(\mathbb T^d \times \mathbb T^d\) is compact, all these derivatives are
bounded. For every multi-index \(\alpha\) and
\(\varphi \in C^\infty(\mathbb T^d)\), differentiation under the integral
gives
\[
    \partial_x^\alpha L_\rho^S\varphi(x)
    =
    \int_{\mathbb T^d}
        \partial_x^\alpha a_\rho(x,y)\varphi(y) \d y.
\]
Cauchy--Schwarz in \(y\), followed by integration in \(x\), yields
\[
    \bigl\|\partial_x^\alpha L_\rho^S\varphi\bigr\|_{L^2}
    \leq
    \bigl\|\partial_x^\alpha a_\rho\bigr\|_{L^2(\mathbb T^d \times \mathbb
T^d)}
    \|\varphi\|_{L^2}.
\]
The coefficient on the right is finite because \(a_\rho\) is smooth.
Summing this estimate over \(|\alpha| \leq r+1\) gives a constant \(C_r\)
such that
\[
    \bigl\|L_\rho^S\varphi\bigr\|_{H^{r+1}}
    \leq
    C_r\|\varphi\|_{L^2}.
\]
Since \(L_\rho^S\varphi\) has zero mean (it is an \(x\)-divergence), this
proves, by density, that \(L_\rho^S\) extends to a bounded map
\[
    L^2_0(\mathbb T^d) \longrightarrow H^{r+1}_0(\mathbb T^d),
\]
and, by the compact Sobolev embedding
\(H^{r+1}_0(\mathbb T^d) \Subset H^r_0(\mathbb T^d)\), it is compact as a map
\(L^2_0(\mathbb T^d) \to H^r_0(\mathbb T^d)\).

The estimates above show that \(L_\rho^S\) is continuous as a map from
\(C^\infty_0(\mathbb T^d)\) to \(C^\infty_0(\mathbb T^d)\). Suppose, towards a
contradiction, that it were surjective. By injectivity, it would be a continuous linear bijection of the Fr\'{e}chet space
\(C^\infty_0(\mathbb T^d)\) onto itself. The Fr\'{e}chet open mapping theorem
applied to this map gives a continuous inverse. Applying the continuity of
that inverse to the \(L^2\)-seminorm, and using the Sobolev seminorms that
generate the topology, yields \(r \geq 0\) and \(C > 0\) such that
\begin{equation}\label{eq:inverse-estimate}
    \|\varphi\|_{L^2}
    =
    \bigl\|(L_\rho^S)^{-1}L_\rho^S\varphi\bigr\|_{L^2}
    \leq
    C\|L_\rho^S\varphi\|_{H^r}
    \qquad
    \text{for every } \varphi \in C^\infty_0(\mathbb T^d).
\end{equation}
Choose an \(L^2\)-orthonormal sequence
\((\varphi_n)_{n \geq 1} \subset C^\infty_0(\mathbb T^d)\), for instance
\(\varphi_n(x) = \sqrt{2}\cos(2\pi n x_1)\). It converges weakly to zero in
\(L^2_0(\mathbb T^d)\). Compactness gives
\(\|L_\rho^S\varphi_n\|_{H^r} \to 0\), contradicting
\eqref{eq:inverse-estimate}. Thus \(L_\rho^S\) is not surjective.
\end{proof}

\section{Differentiability in the Local Mobility Calculus}\label{app:diif_in_in_local_mobility_calculus}
Due to the lack of a differentiable structure on the space of densities in the smooth local mobility calculus developed earlier, we provide a workaround in that framework that naturally extends the definition of Fr\'echet differentiability in Otto--Wasserstein geometry on the space of smooth probability densities on closed manifolds. 
\paragraph{Local first variations and the regular gradient.}For \(\rho\in\P_{m_\rho,B}\),  we define the actual admissible variations at
\(\rho\) by
\begin{equation}\label{eq:admi_var_LM}
    \mathcal V_\rho:=
\left\{
\eta\in C^\infty_c(\R^d)\cap T_\rho\P_{m_\rho,B}:
\int_{\R^d}\eta\,\mathrm{d} x=0,\ 
\exists\varepsilon>0\ \text{such that}\
\rho+t\eta\in\P_{m_\rho,B}\ \text{for }|t|<\varepsilon
\right\}.
\end{equation}
Then we say that \(\mathcal F:\P_{m_\rho,B}\to \R \) has a
\emph{local regular first variation} at \(\rho\in\P_{m_\rho,B}\) if there is a class \([f_\rho]\in T_\rho^*\P_{m_\rho,B}\)
such that, for every \(\eta\in\mathcal V_\rho\),
\begin{equation}\label{eq:def_local_regular_variation}
   \lim_{t\to0}
\frac{\mathcal F(\rho+t\eta) - \mathcal F(\rho)}{t}
=
\int_{\R^d}f_\rho\,\eta \d x. 
\end{equation}
We then write \([f_\rho]
=
\left[\frac{\delta \mathcal F}{\delta\rho}(\rho)\right].\) 
Its associated regular extension is the linear map
\begin{equation}\label{eq:tangent_point_extension}
    \ell_{\F,\rho}^{\reg}:T_\rho\P_{m_\rho,B}\to\R,
    \qquad
    \ell_{\F,\rho}^{\reg}(\eta)
    \coloneqq
    \int_{\R^d}f_\rho\,\eta \d x.
\end{equation}
We first give a sufficient condition for this regular extension to be
unique.
\begin{lemma}[Uniqueness of the local regular first variation]
\label{lem:uniqueness_local_regular_variation}
Let \(\rho\in\P_{m_\rho,B}\), and suppose that \(\F\) admits a
local regular first variation at \(\rho\) in the sense of~\eqref{eq:def_local_regular_variation}. If 
\[
    \left\{\eta \in C_c^\infty(\R^d): \, \int_{\R^d} \eta \d x=0\right\}\subseteq \mathcal{V}_\rho
\]
where \(\mathcal{V}_\rho\) is as in \eqref{eq:admi_var_LM}, then the following holds: 
\begin{enumerate}
    \item\label{it_1_uniqueness_lemma} the local regular first variation \(\left[\frac{\delta \mathcal F}{\delta\rho}(\rho)\right]\) is unique, 
    \item\label{it_2_uniqueness_lemma} its associated regular extension
\(\ell_{\F,\rho}^{\reg}\) in \eqref{eq:tangent_point_extension}
is uniquely determined.
\end{enumerate}
\end{lemma}
\begin{proof}
Clearly \ref{it_2_uniqueness_lemma} follows directly from \ref{it_1_uniqueness_lemma}, so it suffices to prove \ref{it_1_uniqueness_lemma}. 
Let \([\Phi],[\Psi]\in T_\rho^*\P_{m_\rho,B}\) both represent
a local regular first variation of \(\F\) at \(\rho\).
By the assumption on admissible variations,
\[
    \int_{\R^d}(\Phi-\Psi)\eta \d x=0
\]
for every \(\eta\in C_c^\infty(\R^d)\) with zero integral.
For any \(\eta\in C_c^\infty(\R^d)\) and \(i=1,\ldots,d\),
the function \(\partial_i\eta\) has zero integral. Hence
\[
    0
    =
    \int_{\R^d}(\Phi-\Psi)\partial_i\eta \d x
    =
    -\int_{\R^d}\partial_i(\Phi-\Psi)\eta \d x.
\]
It follows that \(\nabla(\Phi-\Psi)=0\), so \(\Phi-\Psi\)
is constant on \(\R^d\). Thus \([\Phi]=[\Psi]\).
\end{proof}
For the rest of the section we assume always that the assumptions of \ref{lem:uniqueness_local_regular_variation} hold. Under these assumptions we set \[
D\F(\rho)[\eta]:=\ell_{\mathcal F,\rho}^{\reg}[\eta] \qquad \forall\eta \in T_\rho\P_{m_\rho,B}.
\]
For a functional \(\mathcal F\) having a local regular first variation at $\rho$ in the sense of \eqref{eq:def_local_regular_variation} we define its gradient at the point $\rho$ with respect to the local mobility metric \(G_{m_\rho,B}\) as in \ref{prop:LM_metric_well_defined} as the unique tangent vector \(\operatorname{grad}_{G_{m_\rho,B}}\F(\rho)\) so that \begin{equation}\label{eq:def_gradient_loc_mob_calculus}
     G_{m_\rho,B}\left(
         \operatorname{grad}_{G_{m_\rho,B}}\F(\rho),
        \dot\eta
    \right)
    =
 D\F(\rho)[\dot\eta] \qquad\forall \dot\eta \in T_\rho\P_{m_\rho,B}
\end{equation}
holds.
With this notation, we can prove the following.

\begin{lemma}[Regular local-mobility gradient]
\label{lem:regular-local-mobility-gradient}
If \(\F\) has a local regular first variation at \(\rho\) in the sense of \eqref{eq:def_local_regular_variation}, then its gradient with respect to the local mobility metric \(G_{m_\rho,B}\) as in \ref{prop:LM_metric_well_defined} at the point \(\rho\) is given by
\[
    \operatorname{grad}_{G_{m_\rho,B}}\F(\rho)
    =
    L_{m_\rho,B}
    \left[
        \frac{\delta\F}{\delta\rho}(\rho)
    \right]
    =
    -\nabla\cdot\left(
        m_\rho B
        \nabla\frac{\delta\F}{\delta\rho}(\rho)
    \right).
\]
\end{lemma}
\begin{proof}
Let \(
    \dot\eta=L_{m_\rho,B}[\Phi_{\dot\eta}]\).
Then \eqref{eq:LM_metric_bilinear},
\eqref{eq:LM_weak_integration_by_parts}, \eqref{eq:tangent_point_extension} and \eqref{eq:def_gradient_loc_mob_calculus} imply
\[
    G_{m_\rho,B}\left(
        L_{m_\rho,B}
        \left[
            \frac{\delta\F}{\delta\rho}(\rho)
        \right],
        \dot\eta
    \right)=
    \int_{\R^d}
    \frac{\delta\F}{\delta\rho}(\rho)\,
    \dot\eta \d x
    =
    D\F(\rho)[\dot\eta].
\]
Combined with Theorem~\ref{prop:LM_operator_invertible}, this completes the proof.
\end{proof}
As shown in \eqref{set1} and \eqref{set2}, Otto--Wasserstein and
softmax--Wasserstein geometry are instances of local-mobility geometry.

\begin{lemma}\label{lem:asump_differtiability_OW-SM}
For both geometries, the
Otto--Wasserstein and softmax--Wasserstein geometries, the assumption
on admissible variations in
Lemma~\ref{lem:uniqueness_local_regular_variation} is satisfied.
Consequently, the conclusions of Lemma~\ref{lem:uniqueness_local_regular_variation} hold.
\end{lemma}

\begin{proof}
Let \(\rho\in\P_{m_\rho,B}\) and
\(\dot\eta\in C_c^\infty(\R^d)\) satisfy
\(\int_{\R^d}\dot\eta \d x=0\).
Since \(\dot\eta/\rho\) is bounded, for sufficiently small \(|t|\),
\[
    \frac12\rho\leq\rho+t\dot\eta\leq\frac32\rho.
\]
For Otto mobility, these bounds apply directly to the mobility.
For softmax mobility, multiplying by \(\exp(z^\top Ax)>0\) and
integrating in \(z\) gives
\[
    \frac12Z_\rho\leq Z_{\rho+t\dot\eta}\leq\frac32Z_\rho,
    \qquad
    \frac13m_\rho\leq m_{\rho+t\dot\eta}\leq3m_\rho.
\]
Moreover, \(Z_{\rho+t\dot\eta}-Z_\rho\) is smooth by compact
support of \(\dot\eta\). Thus the perturbed mobility remains
smooth and positive, with finite mass.

Uniform comparison also preserves the weighted Poincar\'e inequality:
\begin{align*}
    \inf_{c\in\R}
    \int_{\R^d}|\Phi-c|^2 \d\mathfrak m_{\rho+t\dot\eta}
    \leq
    3\inf_{c\in\R}
    \int_{\R^d}|\Phi-c|^2 \d\mathfrak m_\rho\leq
    9C_\rho
    \int_{\R^d}
    (\nabla\Phi)^\top B\nabla\Phi\,
    \d\mathfrak m_{\rho+t\dot\eta}.
\end{align*}
Normalization follows from \(\int\dot\eta \d x=0\), while
compact support ensures
\(\int|\dot\eta|^2/m_\rho \d x<\infty\).
Hence \(\dot\eta\) is an actual admissible variation.
The conclusion follows from
Lemma~\ref{lem:uniqueness_local_regular_variation}.
\end{proof}
We can relate local regular first variations to ordinary
differentiation in Otto calculus.

\begin{example}[Otto calculus on $\P^\infty(\mathcal M)$]
\label{ex:app_otto_calc_dens}
Let \(\mathcal M\) be a closed, connected Riemannian manifold.
The space \(\P^\infty(\mathcal M)\) is open in the affine Fr\'echet space of
smooth functions with unit integral, and its tangent vectors are the
smooth functions with zero integral.

For a \(C^1\) functional \(\F \colon \P^\infty(\mathcal M)\to\R\) in the Bastiani
sense, a local regular first variation at \(\rho\) exists if and only
if \(D\F(\rho)\) has a smooth representative:
\[
    D\F(\rho)[\dot\eta]
    =
    \int_{\mathcal M}
    \frac{\delta\F}{\delta\rho}(\rho)\,
    \dot\eta \d\mathrm{vol}_{\mathcal M},
    \qquad
    \dot\eta\in T_\rho\P^\infty(\mathcal M).
\]
For \(m_\rho=\rho\) and the identity mobility tensor, the corresponding
regular gradient is the usual Otto gradient,
\[
    \operatorname{grad}_{G_\rho^W}\F(\rho)
    =
    -\operatorname{div}_{\mathcal M}\left(
        \rho\nabla_{\mathcal M}
        \frac{\delta\F}{\delta\rho}(\rho)
    \right).
\]
Thus, \eqref{eq:def_local_regular_variation} agrees with ordinary
differentiation whenever the differential has a smooth representative.
The \(C^1\) assumption alone does not guarantee such a representative:
the differential may instead be distributional.
\end{example}
\paragraph{Regular Hamiltonian variations in local mobility.}
Consider pairs \((\rho,[\Phi])\in \P_{m_\rho,B}\times T_\rho^*\P_{m_\rho,B}\). Before moving on, we want to stress that this is only a state set and we don't assume the existence of a cotangent-bundle structure. Next we want to define a notion of differentiability for the Hamiltonian we are interested in that is:
\begin{equation}\label{eq:def_ham_app}
    \mathcal H(\rho,[\Phi])
    \coloneqq
    \frac12
    \left\langle
        [\Phi],L_{m_\rho,B}[\Phi]
    \right\rangle
    +\F(\rho)
    =
    \frac12
    \int_{\R^d}
    (\nabla\Phi)^\top B\nabla\Phi\,
    \d\mathfrak m_\rho
    +\F(\rho), \qquad (\rho,[\Phi])\in \P_{m_\rho,B}\times T_\rho^*\P_{m_\rho,B}.
\end{equation}
We always assume in the following that \(\F\) has a local regular first variation at \(\rho\) in the sense of \eqref{eq:def_local_regular_variation}. 
For fixed \([\Phi]\in C^\infty(\R^d)/\R\), we define the actual
admissible variations compatible with \([\Phi]\) by
\[
    \mathcal V_{\rho,\Phi}
    :=
    \left\{
        \eta\in\mathcal V_\rho:
        [\Phi]\in T_{\rho+t\eta}^*\P_{m_\rho,B}
        \text{ for all sufficiently small }|t|
    \right\}.
\]
Assume that \(\F\) has a local regular first variation at \(\rho\),
represented by \([f_\rho]
    =
    \left[
        \frac{\delta\F}{\delta\rho}(\rho)
    \right]
    \in T_\rho^*\P_{m_\rho,B}.\)
We say that the kinetic part of \(\mathcal H\) has a
\emph{fixed-ambient regular sectional variation} at \(\rho\) if there
is a class \( [h_{\rho,\Phi}]
    \in T_\rho^*\P_{m_\rho,B}\)
such that, for every \(\eta\in\mathcal V_{\rho,\Phi}\),
\begin{equation}\label{eq:def_fixed-ambient regular sectional variation}
    \left.\frac{\d}{\d t}\right|_{t=0}
    \frac12
    \int_{\R^d}
    (\nabla\Phi)^\top B\nabla\Phi\,
    \d\mathfrak m_{\rho+t\eta}
    =
    \int_{\R^d}
    h_{\rho,\Phi}\, \eta \d x.
\end{equation}
Here, fixed ambient means that \([\Phi]\) is kept fixed as an element
of \(C^\infty(\R^d)/\R\) while \(\rho\) is varied. In this case, the associated regular extensions are the linear maps
\begin{align}
    &\mathfrak d_\Phi^{\reg}\mathcal H(\rho,[\Phi]):T_\rho^*\P_{m_\rho,B} \longrightarrow \R, \qquad 
    \mathfrak d_\Phi^{\reg}\mathcal H(\rho,[\Phi])[\Psi]:=
    \left\langle
        [\Psi],L_{m_\rho,B}[\Phi]
    \right\rangle,\label{eq:def_reg_del_phi_H}
    \\
   & \mathfrak d_\rho^{\reg}\mathcal H(\rho,[\Phi]):T_\rho\P_{m_\rho,B} \longrightarrow \R, \qquad \mathfrak d_\rho^{\reg}\mathcal H(\rho,[\Phi])[\eta]:=
    \int_{\R^d}
    \bigl(f_\rho+h_{\rho,\Phi}\bigr)\eta \d x.\label{eq:def_reg_del_rho_H}
\end{align}
Indeed, one can prove the following lemma similarly to Lemma \ref{lem:uniqueness_local_regular_variation}.

\begin{lemma}[Uniqueness of the regular Hamiltonian variations]
\label{lem:uniqueness_regular_hamiltonian_variations}
Let \(\rho\in\P_{m_\rho,B}\) and
\([\Phi]\in T_\rho^*\P_{m_\rho,B}\).
Suppose that \(\F\) admits a local regular first variation at $\rho$ in the sense of \eqref{eq:def_local_regular_variation}
and that the kinetic part of \(\mathcal H\) admits a
fixed-ambient regular sectional variation at \(\rho\) in the sense of \eqref{eq:def_fixed-ambient regular sectional variation}.
If
\[
    \left\{
        \eta\in C_c^\infty(\R^d):
        \int_{\R^d}\eta \d x=0
    \right\}
    \subseteq\mathcal V_{\rho,\Phi},
\]
then the following holds:
\begin{enumerate}
    \item
    The classes \([f_\rho]\) and \([h_{\rho,\Phi}]\), and hence \( \frac{\delta^{\reg}\mathcal H}{\delta\rho}
        (\rho,[\Phi])\) are uniquely determined.
    \item
    The regular extensions in
    \eqref{eq:def_reg_del_phi_H} and
    \eqref{eq:def_reg_del_rho_H}
    are uniquely determined.
\end{enumerate}
\end{lemma}
Now we can prove the following lemma.

\begin{lemma}[Representatives of the regular Hamiltonian variations]
\label{lem:regular-hamiltonian-representatives} Suppose $\mathcal{H}$ is as in \eqref{eq:def_ham_app} so that \(\mathcal{H}\) has a fixed-ambient regular sectional variation at $\rho$ in the sense of \eqref{eq:def_fixed-ambient regular sectional variation}. Then, we have
\begin{align*}
    \left.\frac{\d}{\d s}\right|_{s=0}
    \mathcal H(\rho,[\Phi+s\Psi])
   & =
    \mathfrak d_\Phi^{\reg}
    \mathcal H(\rho,[\Phi])[\Psi] \qquad \forall[\Psi]\in T_\rho^*\P_{m_\rho,B},\\
   \left.\frac{\d}{\d t}\right|_{t=0}
    \mathcal H(\rho+t\eta,[\Phi])
    &=
    \mathfrak d_\rho^{\reg}
    \mathcal H(\rho,[\Phi])[\eta] \qquad \forall\eta\in\mathcal V_{\rho,\Phi}.
\end{align*}
\end{lemma}

Lemma \ref{lem:regular-hamiltonian-representatives} justifies the notation 
\[
    \frac{\delta^{\reg}\mathcal H}{\delta\Phi}
    (\rho,[\Phi])
    =
    L_{m_\rho,B}[\Phi],
    \qquad
    \frac{\delta^{\reg}\mathcal H}{\delta\rho}
    (\rho,[\Phi])
    =
    [f_\rho+h_{\rho,\Phi}].
\]

\begin{proof}
The first identity follows by expanding the kinetic part of
\(\mathcal H(\rho,[\Phi+s\Psi])\) in \(s\). The second identity follows
by adding the prescribed local regular first variation of \(\F\) to the
fixed-ambient regular sectional variation of the kinetic part of
\(\mathcal H\).
\end{proof}

With these notions of regular variation, we can state
the accelerated Newton equations. A fixed-ambient regular damped
Hamiltonian trajectory is a pair of functions \((\rho_t,\Phi_t)\),
smooth in \((t,x)\), for which the regular variation classes have been
specified at every time and such that
\[
    \rho_t\in\P_{m_\rho,B},
    \qquad
    [\Phi_t]\in T_{\rho_t}^*\P_{m_\rho,B},
\]
and
\begin{equation}\label{eq:newt_app_pde}
    \begin{cases}
    \partial_t\rho_t
    =
    L_{m_{\rho_t},B}[\Phi_t],
    \\[0.5ex]
    [\partial_t\Phi_t]+\alpha_t[\Phi_t]
    =
    -[f_{\rho_t}+h_{\rho_t,\Phi_t}]
    \qquad
    \text{in }C^\infty(\R^d)/\R.
\end{cases}
\end{equation}

\begin{remark}
The system of PDEs in \eqref{eq:newt_app_pde} is well defined once the two regular
variation classes in \eqref{eq:def_reg_del_phi_H} and \eqref{eq:def_reg_del_rho_H} have been specified. Its second equation is an
equality modulo spatially constant functions.

The fixed-ambient convention means that, when \(\rho\) is varied, the
potential \([\Phi]\) is kept fixed in \(C^\infty(\R^d)/\R\). Since
the spaces \(T_\rho^*\P_{m_\rho,B}\) depend on \(\rho\), this convention
does not define an intrinsic differential on a cotangent bundle. A
different rule for comparing cotangent vectors at different densities
could, in general, yield a different variation in the density variable.
\end{remark}
Before recovering the ordinary damped Hamiltonian equations in
Otto--Wasserstein geometry, we examine the analog of
Lemma~\ref{lem:asump_differtiability_OW-SM} for variations compatible
with a fixed potential. For the geometries in \eqref{set1} and \eqref{set2},
 the following lemma shows that the
admissibility assumption of
Lemma~\ref{lem:uniqueness_regular_hamiltonian_variations}
holds automatically in Otto--Wasserstein geometry and characterizes
the additional compatibility condition in softmax--Wasserstein geometry.

\begin{lemma}[Compatible variations in Otto and softmax geometry]
\label{lem:compatible_variations_otto_softmax}
For Otto--Wasserstein geometry, the admissibility assumption in
Lemma~\ref{lem:uniqueness_regular_hamiltonian_variations}
holds for every \(\rho\in\P_{m_\rho,B}\) and
\([\Phi]\in T_\rho^*\P_{m_\rho,B}\).

For softmax--Wasserstein geometry, a compactly supported smooth
zero-integral variation \(\eta\) belongs to
\(\mathcal V_{\rho,\Phi}\) if and only if
\begin{equation}\label{eq:softmax_hamiltonian_compatibility}
    \left(
        \nabla\frac{Z_\eta}{Z_\rho}
    \right)^\top B\nabla\Phi
    \in L^2(m_\rho),
    \qquad
    Z_\eta(x)
    \coloneqq
    \int_{\R^d}\exp(z^\top Ax)\eta(z) \d z.
\end{equation}
\end{lemma}

Note that
   in the softmax--Wasserstein geometry a sufficient condition is \(\nabla\!\left(\frac{Z_\eta}{Z_\rho}\right)
    \in L^\infty(\R^d)\).

\begin{proof}
The Otto--Wasserstein case follows along the lines of
Lemma~\ref{lem:asump_differtiability_OW-SM}, using the compact
support of \(\eta/\rho\). We therefore consider softmax mobility.

By the density-admissibility argument,
\(\eta\in\mathcal V_\rho\), and \(m_{\rho+t\eta}\) and \(m_\rho\)
are uniformly comparable for sufficiently small \(|t|\).
Thus the weighted Sobolev spaces coincide, and it remains to
verify the operator-domain condition.

The product rule gives
\[
    \frac{L_{m_{\rho+t\eta},B}[\Phi]}{m_{\rho+t\eta}}
    =
    \frac{L_{m_\rho,B}[\Phi]}{m_\rho}
    -
    \left(
        \nabla\log\frac{m_{\rho+t\eta}}{m_\rho}
    \right)^\top B\nabla\Phi,
\]
where
\[
    \nabla\log\frac{m_{\rho+t\eta}}{m_\rho}
    =
    \frac{t\nabla(\eta/\rho)}{1+t\eta/\rho}
    -
    \frac{t\nabla(Z_\eta/Z_\rho)}{1+tZ_\eta/Z_\rho}.
\]
Since
\[
    \left|\frac{Z_\eta}{Z_\rho}\right|
    \leq\left\|\frac{\eta}{\rho}\right\|_{L^\infty},
\]
both denominators are uniformly bounded and bounded away from
zero for small \(|t|\). The first term is smooth and compactly
supported, so its pairing with \(B\nabla\Phi\) belongs to
\(L^2(m_\rho)\).
For nonzero such \(t\), uniform weight comparison therefore
makes the operator-domain condition equivalent to
\eqref{eq:softmax_hamiltonian_compatibility}.

Finally, boundedness of \(\nabla(Z_\eta/Z_\rho)\), together
with the finite kinetic energy of \(\Phi\), implies
\eqref{eq:softmax_hamiltonian_compatibility}.
If \(\nabla\Phi\) has compact support, the product in that
condition is itself smooth and compactly supported.
\end{proof}
\begin{example}[Hamiltonian calculus on $\P^\infty(\mathcal M)$]
\label{ex:hamiltonian_otto_calculus}
Let \(\mathcal M\) be a closed, connected Riemannian manifold. Then,
\[
    T_\rho^*\P^\infty(\mathcal M)=C^\infty(\mathcal M)/\R,
    \qquad
    \mathcal V_{\rho,\Phi}=T_\rho\P^\infty(\mathcal M).
\]
Thus, the regular cotangent space is independent of \(\rho\),
and every tangent vector is compatible with every smooth potential.

Suppose that \(\F \colon \P^\infty(\mathcal M)\to\R\) is \(C^1\) in the
Bastiani sense and that its differential has a smooth representative:
\[
    D\F(\rho)[\dot\eta]
    =
    \int_{\mathcal M}
    \frac{\delta\F}{\delta\rho}(\rho)\,
    \dot\eta \d\mathrm{vol}_{\mathcal M},
    \qquad
    \dot\eta\in T_\rho\P^\infty(\mathcal M).
\]
Then the Hamiltonian
\[
    \mathcal H(\rho,[\Phi])
    \coloneqq
    \frac12\int_{\mathcal M}
    \rho|\nabla_{\mathcal M}\Phi|^2 \d\mathrm{vol}_{\mathcal M}
    +\F(\rho)
\]
is \(C^1\) on
\(\P^\infty({\mathcal M})\times C^\infty(\mathcal M)/\R\) with
\begin{align*}
    D_\Phi\mathcal H(\rho,[\Phi])[\Psi]
    &=
    \int_{\mathcal M}
    L_\rho^W[\Phi]\,\Psi \d\mathrm{vol}_{\mathcal M}
    =
    \int_{\mathcal M}
    \rho\,\nabla_{\mathcal M}\Phi\cdot\nabla_{\mathcal M}\Psi\,
    \d\mathrm{vol}_{\mathcal M},
    \\
    D_\rho\mathcal H(\rho,[\Phi])[\dot\eta]
    &=
    \int_{\mathcal M}
    \left(
        \frac12|\nabla_{\mathcal M}\Phi|^2
        +
        \frac{\delta\F}{\delta\rho}(\rho)
    \right)
    \dot\eta \d\mathrm{vol}_{\mathcal M}.
\end{align*}
Hence the regular Hamiltonian variations coincide with the ordinary
partial differentials. In parti\-cular, we have
\[
    \frac{\delta^{\reg}\mathcal H}{\delta\Phi}(\rho,[\Phi])
    =
    L_\rho^W[\Phi],
    \qquad
    \frac{\delta^{\reg}\mathcal H}{\delta\rho}(\rho,[\Phi])
    =
    \left[
        \frac12|\nabla_{\mathcal M}\Phi|^2
        +
        \frac{\delta\F}{\delta\rho}(\rho)
    \right].
\]
Consequently, \eqref{eq:newt_app_pde} reduces to the damped
Otto--Wasserstein Hamiltonian system
\[
    \begin{cases}
        \partial_t\rho_t
        +
        \operatorname{div}_{\mathcal M}(\rho_t\nabla_{\mathcal M}\Phi_t)
        =0,
        \\[0.75ex]
        [\partial_t\Phi_t]+\alpha_t[\Phi_t]
        +
        \left[
            \displaystyle\frac12|\nabla_{\mathcal M}\Phi_t|^2
            +
            \frac{\delta\F}{\delta\rho}(\rho_t)
        \right]
        =0
        \quad\text{in }C^\infty({\mathcal M})/\R.
    \end{cases}
\]
The second equation holds modulo spatial constants; a
time-dependent additive gauge for \(\Phi_t\) gives an equivalent
pointwise equation.
\end{example}
\section{Differentiability in the Quotient Stein Geometry}\label{app:diff_stein_geom}
\paragraph{Local first variations and the regular Stein gradient.}
For \(\rho\in\P_{K_A}\), we define the actual admissible variations
at \(\rho\) by
\[
    \mathcal V_\rho
    :=
    \left\{
        \eta\in C_c^\infty(\R^d)\cap T_\rho\P_{K_A}:
        \exists\varepsilon>0
        \text{ such that }
        \rho+t\eta\in\P_{K_A}
        \text{ for }|t|<\varepsilon
    \right\}.
\]
In the present setting this simplifies to:
\[
    \mathcal V_\rho
    =
    \left\{
        \eta\in C_c^\infty(\R^d):
        \int_{\R^d}\eta \d x=0
    \right\}.
\]
Indeed, the constraints in the definition of \(\P_{K_A}\) as in \eqref{eq:density-class} are preserved under small perturbations. Since \(\rho\) is smooth and strictly positive, it has a
positive minimum on the compact support of \(\eta\). Consequently,
\(\eta/\rho\in L^2(\rho)\), and \(\rho+t\eta\) remains strictly
positive for all sufficiently small \(|t|\). Its integral remains
equal to one. Moreover, the second-moment condition and the condition
\[
    \int_{\R^d}
    |x|^2
    \frac{|\nabla(\rho+t\eta)|^2}{\rho+t\eta}\,
    \d x
    <\infty\]
are preserved: outside the support of \(\eta\), the density is
unchanged, while on its support the denominator is bounded below
by a positive constant.

We say that \(\F \colon \P_{K_A}\to\R\) has a
\emph{local regular first variation} at \(\rho\) if there is a class \( [f_\rho] \in T_\rho^*\P_{K_A}\) such that, for every \(\eta\in\mathcal V_\rho\),
\begin{equation}
\label{eq:def_local_regular_stein_variation}
    \lim_{t\to0}
    \frac{\F(\rho+t\eta)-\F(\rho)}{t}
    =
    \int_{\R^d}
    f_\rho\eta \d x.
\end{equation}
We then write \([f_\rho]=\left[
        \frac{\delta\F}{\delta\rho}(\rho)\right].\) Its associated regular extension is the linear map
\[
    \ell_{\F,\rho}^{\mathrm{reg}}
    :
    T_\rho\P_{K_A}\longrightarrow\R,
    \qquad
    \ell_{\F,\rho}^{\mathrm{reg}}(\dot\eta)
    :=
    \int_{\R^d}
    f_\rho\dot\eta \d x.
\]
This integral is finite, since
\[
    \left|
        \int_{\R^d}
        f_\rho\dot\eta \d x
    \right|
    \le
    \|f_\rho\|_{L^2(\rho)}
    \left\|
        \frac{\dot\eta}{\rho}
    \right\|_{L^2(\rho)}.
\]
Moreover, it is independent of the representative of
\([f_\rho]\), because every
\(\dot\eta\in T_\rho\P_{K_A}\) has zero integral.

\begin{remark}[Uniqueness of the local regular first variation]
The local regular first variation is unique modulo constants.
Indeed, if \([\Phi],[\Psi]\in T_\rho^*\P_{K_A}\)
both represent the first variation, then
\[
    \int_{\R^d}
    (\Phi-\Psi)\eta \d x=0
\]
for every \(\eta\in C_c^\infty(\R^d)\) with $\int\eta=0$.
Hence \(\Phi-\Psi\) is constant. Consequently,
\([f_\rho]\) and
\(\ell_{\F,\rho}^{\mathrm{reg}}\) are uniquely determined.
\end{remark}
To define the Stein gradient, we restrict this regular extension to
\(T_\rho^{\mathrm{eff}}\P_{K_A}\). By using \eqref{eq:effective-spaces} we obtain the corresponding effective cotangent class \( \bigl[[f_\rho]\bigr]\in T_\rho^{*,\mathrm{eff}}\P_{K_A}.\)   
This class determines the restriction of
\(\ell_{\F,\rho}^{\mathrm{reg}}\) to the effective tangent space.
Indeed, if
\([\Psi]\in\ker L_\rho^{K_A}\) and
\(\dot\eta=L_\rho^{K_A}[\Phi]\), the Green identity gives
\[
    \int_{\R^d}
    \Psi\dot\eta \d x
    =
    \operatorname{tr}
    \left(
        M_\rho(\Psi)^\top A M_\rho(\Phi)
    \right)
    =
    0.
\]

With this notation, we can prove the following.

\begin{lemma}[Regular Stein gradient]
\label{lem:regular-stein-gradient}
If \(\F\) has a local regular first variation at \(\rho\) in the
sense of \eqref{eq:def_local_regular_stein_variation}, then there is
a unique vector
\[
    \operatorname{grad}_{G_\rho^{K_A}}\F(\rho)
    \in T_\rho^{\mathrm{eff}}\P_{K_A}
\]
such that
\[
    G_\rho^{K_A}
    \left(
        \operatorname{grad}_{G_\rho^{K_A}}\F(\rho),
        \dot\eta
    \right)
    =
    \ell_{\F,\rho}^{\mathrm{reg}}(\dot\eta)
    \qquad
    \text{for every }
    \dot\eta\in T_\rho^{\mathrm{eff}}\P_{K_A}.
\]
This regular Stein gradient is given by
\begin{align*}
    \operatorname{grad}_{G_\rho^{K_A}}\F(\rho)=
    -\nabla\cdot
    \left(
        \rho
        M_\rho\left(
            \frac{\delta\F}{\delta\rho}(\rho)
        \right)^\top
        Ax
    \right).
\end{align*}
\end{lemma}

\begin{proof}
Since \([f_\rho]\in T_\rho^*\P_{K_A}\), we have \(L_\rho^{K_A}[f_\rho]
    \in T_\rho^{\mathrm{eff}}\P_{K_A}.\) Let \(\dot\eta\in T_\rho^{\mathrm{eff}}\P_{K_A}\), and choose
\([\Phi]\in T_\rho^*\P_{K_A}\) such that \( \dot\eta=L_\rho^{K_A}[\Phi].\)
By using the polarization of the Stein metric and the Green identity, we obtain
\begin{align*}
    G_\rho^{K_A}
    \left(
        L_\rho^{K_A}[f_\rho],
        \dot\eta
    \right)
    =
    \operatorname{tr}
    \left(
        M_\rho(f_\rho)^\top A M_\rho(\Phi)
    \right)
   =
    \int_{\R^d}
    f_\rho\dot\eta \d x
    =
    \ell_{\F,\rho}^{\mathrm{reg}}(\dot\eta).
\end{align*}
Thus \(L_\rho^{K_A}[f_\rho]\) satisfies the defining gradient
identity. Uniqueness follows from the positive definiteness of
\(G_\rho^{K_A}\) on \(T_\rho^{\mathrm{eff}}\P_{K_A}\).
\end{proof}

\begin{remark}
The regular Stein gradient is defined through the regular extension
\(\ell_{\F,\rho}^{\mathrm{reg}}\). The directional limits in
\eqref{eq:def_local_regular_stein_variation} are required only for
\(\eta\in\mathcal V_\rho\). Their existence does not by itself
imply a chain rule along arbitrary curves in \(\P_{K_A}\).
\end{remark}
\paragraph{Regular Hamiltonian variations in the bilinear Stein setting.}
Consider pairs \((\rho,[\Phi])\) satisfying
\[
    \rho\in\P_{K_A},
    \qquad
    [\Phi]\in T_\rho^*\P_{K_A}.
\]
We regard this collection only as a state set and do not assume
a cotangent-bundle structure. For such a pair, define
\begin{equation}
\label{eq:def_stein_ham_app}
\begin{aligned}
    \mathcal H(\rho,[\Phi])
    \coloneqq
    \frac12
    \left\langle
        [\Phi],L_\rho^{K_A}[\Phi]
    \right\rangle_\rho
    +\F(\rho)
   =
    \frac12
    \operatorname{tr}
    \left(
        M_\rho(\Phi)^\top A M_\rho(\Phi)
    \right)
    +\F(\rho).
\end{aligned}
\end{equation}
Throughout this paragraph, assume that \(\F\) has a local regular
first variation at \(\rho\) in the sense of
\eqref{eq:def_local_regular_stein_variation}, represented by
\[
    [f_\rho]
    =
    \left[
        \frac{\delta\F}{\delta\rho}(\rho)
    \right]
    \in T_\rho^*\P_{K_A}.
\]
For fixed \([\Phi]\in T_\rho^*\P_{K_A}\), define the actual
admissible variations compatible with \([\Phi]\) by
\[
    \mathcal V_{\rho,\Phi}
    :=
    \left\{
        \eta\in\mathcal V_\rho:
        [\Phi]\in T_{\rho+t\eta}^*\P_{K_A}
        \text{ for all sufficiently small }|t|
    \right\}.
\]
In the present setting,
\[
    \mathcal V_{\rho,\Phi}=\mathcal V_\rho.
\]
Indeed, every \(\eta\in\mathcal V_\rho\) is compactly supported.
For sufficiently small \(|t|\), the ratio
\((\rho+t\eta)/\rho\) is bounded above and bounded away from zero.
Consequently, the weighted Sobolev spaces \(H^1(\rho+t\eta)\)
and \(H^1(\rho)\) coincide as sets, so \(\Phi\) remains admissible.

We say that the kinetic part of \(\mathcal H\) has a
\emph{fixed-ambient regular sectional variation} at \(\rho\) if
there is a class
\[
    [h_{\rho,\Phi}]\in T_\rho^*\P_{K_A}
\]
such that, for every \(\eta\in\mathcal V_{\rho,\Phi}\),
\begin{equation}
\label{eq:def_stein_fixed_ambient_variation}
    \left.\frac{\d}{\d t}\right|_{t=0}
    \frac12
    \operatorname{tr}
    \left(
        M_{\rho+t\eta}(\Phi)^\top
        A M_{\rho+t\eta}(\Phi)
    \right)
    =
    \int_{\R^d}
    h_{\rho,\Phi}\eta \d x.
\end{equation}
Here, fixed ambient means that \([\Phi]\) is kept fixed
as an element of \(C^\infty(\R^d)/\R\) while \(\rho\) is varied. In this case, the associated regular extensions are the linear maps
\begin{align}
    \mathfrak d_\Phi^{\reg}
    \mathcal H(\rho,[\Phi])[\Psi]
    &:=
    \left\langle
        [\Psi],L_\rho^{K_A}[\Phi]
    \right\rangle_\rho,
    \qquad
    [\Psi]\in T_\rho^*\P_{K_A},
    \label{eq:def_stein_reg_del_phi_H}
    \\
    \mathfrak d_\rho^{\reg}
    \mathcal H(\rho,[\Phi])[\dot\eta]
    &:=
    \int_{\R^d}
    \bigl(f_\rho+h_{\rho,\Phi}\bigr)
    \dot\eta \d x,
    \qquad
    \dot\eta\in T_\rho\P_{K_A}.
    \label{eq:def_stein_reg_del_rho_H}
\end{align}

With this notation, we can prove the following.

\begin{lemma}[Representatives of the regular Stein Hamiltonian variations]
\label{lem:regular-stein-hamiltonian-representatives}
The kinetic part of \(\mathcal H\) has a fixed-ambient regular
sectional variation at \(\rho\) if and only if
\[
    \left[
        (\nabla\Phi)^\top M_\rho(\Phi)^\top Ax
    \right]
    \in T_\rho^*\P_{K_A}.
\]
In this case, its representative may be chosen as
\[
    h_{\rho,\Phi}(x)
    =
    \nabla\Phi(x)^\top M_\rho(\Phi)^\top Ax,
\]
and
\begin{align*}
    \left.\frac{\d}{\d s}\right|_{s=0}
    \mathcal H(\rho,[\Phi+s\Psi])
    &=
    \mathfrak d_\Phi^{\reg}
    \mathcal H(\rho,[\Phi])[\Psi],
    \qquad
    [\Psi]\in T_\rho^*\P_{K_A},
    \\
    \left.\frac{\d}{\d t}\right|_{t=0}
    \mathcal H(\rho+t\eta,[\Phi])
    &=
    \mathfrak d_\rho^{\reg}
    \mathcal H(\rho,[\Phi])[\eta],
    \qquad
    \eta\in\mathcal V_{\rho,\Phi}.
\end{align*}
\end{lemma}

\begin{proof}
By linearity of the matrix-moment operator in the potential,
\[
    M_\rho(\Phi+s\Psi)
    =
    M_\rho(\Phi)+sM_\rho(\Psi).
\]
Expanding the quadratic kinetic term and using the symmetry of \(A\)
therefore gives
\begin{align*}
    \left.\frac{\d}{\d s}\right|_{s=0}
    \mathcal H(\rho,[\Phi+s\Psi])
    =
    \operatorname{tr}
    \left(
        M_\rho(\Psi)^\top A M_\rho(\Phi)
    \right)
    =
    \left\langle
        [\Psi],L_\rho^{K_A}[\Phi]
    \right\rangle_\rho,
\end{align*}
where the last equality is the Green identity. Similarly, for every \(\eta\in\mathcal V_\rho\),
\[
    M_{\rho+t\eta}(\Phi)
    =
    M_\rho(\Phi)
    +
    t\int_{\R^d}
    x\nabla\Phi(x)^\top\eta(x) \d x.
\]
The additional integral is finite because \(\eta\) is compactly
supported. Differentiating the quadratic expression yields
\begin{align*}
    &\left.\frac{\d}{\d t}\right|_{t=0}
    \frac12
    \operatorname{tr}
    \left(
        M_{\rho+t\eta}(\Phi)^\top
        A M_{\rho+t\eta}(\Phi)
    \right)
 =
    \int_{\R^d}
    \nabla\Phi(x)^\top M_\rho(\Phi)^\top Ax\,
    \eta(x) \d x.
\end{align*}
Since \(\mathcal V_\rho\) contains every compactly supported smooth
function with zero integral, this smooth representative is uniquely
determined modulo constants. Thus a regular sectional variation
exists precisely when its class belongs to \(T_\rho^*\P_{K_A}\).
Adding the prescribed local regular first variation of \(\F\)
proves the second identity.
\end{proof}

\begin{remark}
The condition \(h_{\rho,\Phi}\in H^1(\rho)\) is an additional regularity assumption. It does not follow from
\(\Phi\in C^\infty(\R^d)\cap H^1(\rho)\). Under this assumption, we write
\[
    \frac{\delta^{\reg}\mathcal H}{\delta\Phi}
    (\rho,[\Phi])
    =
    L_\rho^{K_A}[\Phi],
    \qquad
    \frac{\delta^{\reg}\mathcal H}{\delta\rho}
    (\rho,[\Phi])
    =
    [f_\rho+h_{\rho,\Phi}].
\]
The first variation in the potential variable takes values in
\(T_\rho^{\mathrm{eff}}\P_{K_A}\). The restriction of the regular
density variation to \(T_\rho^{\mathrm{eff}}\P_{K_A}\) is represented
by the effective cotangent class
\[
    \bigl[[f_\rho+h_{\rho,\Phi}] \bigr]
    \in T_\rho^{*,\mathrm{eff}}\P_{K_A}.
\]
\end{remark}

With these notions of regular variation in place, we can state
the damped Stein Hamiltonian equations. A fixed-ambient regular
damped Hamiltonian trajectory is a pair of functions
\((\rho_t,\Phi_t)\), smooth in \((t,x)\), such that
\[
    \rho_t\in\P_{K_A},
    \qquad
    [\Phi_t] \in T_{\rho_t}^*\P_{K_A},
\]
the preceding regular variations exist at every time, and
\begin{equation}
\label{eq:stein_newt_app_pde}
\begin{cases}
    \partial_t\rho_t
    =
    L_{\rho_t}^{K_A}[\Phi_t],
    \\[0.75ex]
    [\partial_t\Phi_t]
    +\alpha_t[\Phi_t]
    =
    -[f_{\rho_t}+h_{\rho_t,\Phi_t}]
    \quad\text{in }C^\infty(\R^d)/\R.
\end{cases}
\end{equation}
Equivalently, we get
\[
\begin{cases}
    \partial_t\rho_t
    +
    \nabla\cdot
    \left(
        \rho_t M_{\rho_t}(\Phi_t)^\top Ax
    \right)
    =0,
    \\[0.75ex]
    [\partial_t\Phi_t]
    +\alpha_t[\Phi_t]
    +
    \left[
        \nabla\Phi_t(x)^\top
        M_{\rho_t}(\Phi_t)^\top Ax
        +
        \dfrac{\delta\F}{\delta\rho}(\rho_t)(x)
    \right]
    =0
    \quad\text{in }C^\infty(\R^d)/\R.
\end{cases}
\]
A time-dependent additive gauge for \(\Phi_t\) may be chosen
to write the second equation as a pointwise equality of functions.

\begin{remark}[Dependence on the ambient potential class]
At a fixed density \(\rho\), the Hamiltonian depends only on the
effective cotangent class
\[
    \bigl[[\Phi]\bigr]
    \in T_\rho^{*,\mathrm{eff}}\P_{K_A}.
\]
Its fixed-ambient density variation, however, generally depends
on the full class \([\Phi]\).

Indeed, if \([\Psi]\in\ker L_\rho^{K_A}\), then
\(M_\rho(\Psi)=0\), so
\[
    \mathcal H(\rho,[\Phi+\Psi])
    =
    \mathcal H(\rho,[\Phi]).
\]
Nevertheless, the corresponding smooth kinetic representatives
satisfy
\[
    h_{\rho,\Phi+\Psi}(x)-h_{\rho,\Phi}(x)
    =
    \nabla\Psi(x)^\top M_\rho(\Phi)^\top Ax,
\]
which need not be constant. Thus \eqref{eq:stein_newt_app_pde} is formulated using the full
potential classes \([\Phi_t]\) and the fixed-ambient convention.
The construction does not define a density differential depending
only on the effective cotangent class.
\end{remark}

\section{Proofs of Section \ref{sec3}} \label{app:proofs}
For convenience, we give the proof of the following functional analytic Lemma~\ref{lem:LM_centered_energy_space} in the spirit of \cite[Lemma A.8]{CWDS2025}.

\begin{proof}[\textbf{Proof of Lemma~\ref{lem:LM_centered_energy_space}}]
\ref{it:1_lem:LM_centered_energy_space} 
The mean functional  \(\Phi\longmapsto
    \int_{\R^d}\Phi\,\d\mathfrak m_\rho\) is continuous on \(H^1(m_\rho,B)\). Hence its kernel
\(H^1_0(m_\rho, B)\) is closed and therefore a Hilbert space.

Next, by using the local-mobility
Poincar\'e inequality~\eqref{eq:LM_PI} we have 
\[
    \|\Phi\|_{m_\rho,B}^2\le
    \|\Phi\|_{H^1(m_\rho,B)}^2
    \le
    (1+C_\rho)
    \|\Phi\|_{m_\rho,B}^2\qquad \forall\Phi\in H^1_0(m_\rho,B)
\]
Thus the two norms $\|\cdot\|_{m_\rho,B}$ and $\|\cdot\|_{H^1(m_\rho,B)}$ are equivalent on \(H^1_0(m_\rho,B)\).

\ref{it:2_lem:LM_centered_energy_space}:
Let \(\chi\in C_c^\infty(\R^d)\) satisfy
\(0\leq\chi\leq1\), \(\chi=1\) on \(B_1\), and
\(\operatorname{supp}(\chi)\subset B_2\). Define
\(\chi_R(x)\coloneqq\chi(x/R)\). Then
\begin{equation}\label{eq:eq_chi_R}
     \chi_R=1 \text{ on }B_R,
    \qquad
    \operatorname{supp}(\chi_R)\subset B_{2R},
    \qquad
    |\nabla\chi_R|\leq\frac{C}{R}.
\end{equation}
For \(\Phi\in H^1(m_\rho,B)\), set
\(\Phi_R\coloneqq\chi_R\Phi\). We have
\[
   \Vert\Phi-\Phi_R\Vert_{L^2(\mathfrak m_\rho)}^2 = \int_{\R^d}|(1-\chi_R)\Phi|^2\,\d\mathfrak m_\rho
    \leq
    \int_{\R^d\setminus B_R}|\Phi|^2\,\d\mathfrak m_\rho
    \to0 \qquad\text{as } R\to \infty.
\]
Moreover,
\begin{equation}\label{eq:nabla_phi_Phi_r}
     \nabla(\Phi_R-\Phi)
    =
    (\chi_R-1)\nabla\Phi+\Phi\nabla\chi_R.
\end{equation}
Thus, denoting the largest eigenvalue of \(B\) by
\(\lambda_{\max}(B)\), we obtain by using \eqref{eq:eq_chi_R} and \eqref{eq:nabla_phi_Phi_r}
\begin{align*}
     \Vert \nabla\Phi-\nabla\Phi_R\Vert_{L^2(\mathfrak m_\rho,B,\R^d)}^2=&\int_{\R^d}
    \nabla(\Phi_R-\Phi)^\top B
    \nabla(\Phi_R-\Phi)\,\d\mathfrak m_\rho
    \\
    \leq&
    2\int_{\R^d\setminus B_R}
    \nabla\Phi^\top B\nabla\Phi\,\d\mathfrak m_\rho
    +
    \frac{2C^2\lambda_{\max}(B)}{R^2}
    \int_{\R^d}|\Phi|^2\,\d\mathfrak m_\rho,
\end{align*}
as both integrals are uniformly bounded from above for all $R>0$ and the right-hand side converges to zero as $R\to \infty$. Consequently by definition of the norm on \(H^1(m_\rho,B)\) as in \eqref{eq:norm_on_H_1_B_m_rho},
\begin{equation}\label{eq:phi_R_to_Phi_in_H_1_m_rho}
    \Phi_R\to\Phi
    \qquad\text{in }H^1(m_\rho,B).
\end{equation}
For each fixed \(R\), the function \(\Phi_R\) has compact support.
Since \(m_\rho\) is smooth and strictly positive, for every compact
set \(K\subset\R^d\) there exist constants \(0<c_K\leq C_K<\infty\)
such that
\[
    c_K\leq m_\rho(x)\leq C_K
    \qquad\forall x\in K.
\]
Moreover, since \(B=B^\top\succ0\) and denoting the smallest eigenvalue of \(B\) by
\(\lambda_{\min}(B)\),
\[
    \lambda_{\min}(B)|v|^2
    \leq v^\top Bv
    \leq \lambda_{\max}(B)|v|^2
    \qquad\forall v\in\R^d.
\]
Consequently, for every \(u\) supported in \(K\),
\[
    c_K\min\{1,\lambda_{\min}(B)\}
    \|u\|_{H^1(\R^d)}^2
    \leq
    \|u\|_{H^1(m_\rho,B)}^2
    \leq
    C_K\max\{1,\lambda_{\max}(B)\}
    \|u\|_{H^1(\R^d)}^2.
\]
Thus, on every compact set \(K\subset\R^d\), the weighted Sobolev
norm
\[
    \|u\|_{H^1(m_\rho,B;K)}^2
    \coloneqq
    \int_K |u|^2\,\d\mathfrak m_\rho
    +
    \int_K \nabla u^\top B\nabla u\,\d\mathfrak m_\rho
\]
is equivalent to the usual unweighted Sobolev norm
\[
    \|u\|_{H^1(K)}^2
    \coloneqq
    \int_K |u|^2\d x
    +
    \int_K |\nabla u|^2\d x.
\] Standard mollification therefore provides
\(\psi_R\in C_c^\infty(\R^d)\) such that
\begin{equation}\label{eq:smooth_approx_of_psi_R}
     \|\psi_R-\Phi_R\|_{H^1(m_\rho,B)}
    \leq \frac{1}{R}.
\end{equation}
It follows from~\eqref{eq:phi_R_to_Phi_in_H_1_m_rho} and~\eqref{eq:smooth_approx_of_psi_R} in combination with the triangle inequality that
\[
    \psi_R\to\Phi
    \qquad\text{in }H^1(m_\rho,B).
\]
Hence \(C_c^\infty(\R^d)\) is dense in \(H^1(m_\rho,B)\).

It remains to check that \(\psi- \frac{1}{\mathfrak m_\rho(\R^d)}
    \int_{\R^d}\psi\,\d\mathfrak m_\rho\) for \(\psi\in C_c^\infty(\R^d)\) are dense in \(H^1_0(m_\rho,B)\). By density of \(C_c^\infty(\R^d)\) in \(H^1(m_\rho,B)\) for each \(\Phi\in H^1_0(m_\rho,B)\) we can choose
\(\psi_k\in C_c^\infty(\R^d)\) such that
\(\psi_k\to\Phi\) in \(H^1(m_\rho,B)\). Set
\[
    c_k
    \coloneqq
    \frac{1}{\mathfrak m_\rho(\R^d)}
    \int_{\R^d}\psi_k\,\d\mathfrak m_\rho,
    \qquad
    \widetilde\psi_k\coloneqq\psi_k-c_k.
\]
Then \(\widetilde\psi_k\in H^1_0(m_\rho,B)\). Since \(\psi_k\to\Phi\) in \(H^1(m_\rho,B)\), we have
\(\psi_k\to\Phi\) in \(L^2(\mathfrak m_\rho)\). Therefore, by the
Cauchy--Schwarz inequality,\begin{equation}\label{eq:int_psi_k_min_phi_to_zero}
    \left|
        \int_{\R^d}(\psi_k-\Phi)\,\d\mathfrak m_\rho
    \right|
    \leq
    \mathfrak m_\rho(\R^d)^{1/2}
    \|\psi_k-\Phi\|_{L^2(\mathfrak m_\rho)}
    \to0 \qquad \text{as } k\to \infty.
\end{equation}
Since \(\Phi\in H^1_0(m_\rho,B)\), it has zero
\(\mathfrak m_\rho\)-mean. Thus we obtain from \eqref{eq:int_psi_k_min_phi_to_zero}:
\[
    c_k
    =
    \frac{1}{\mathfrak m_\rho(\R^d)}
    \int_{\R^d}\psi_k\,\d\mathfrak m_\rho
    \to0.
\] We therefore conclude by using $\widetilde{\psi}_k=\psi_k-c_k$ that
\[
    \|\widetilde\psi_k-\Phi\|_{H^1(m_\rho,B)}
    \leq
    \|\psi_k-\Phi\|_{H^1(m_\rho,B)}
    +
    |c_k|\,\mathfrak m_\rho(\R^d)^{1/2}
    \to0 \qquad\text{as }k\to \infty.
\]
This finishes the proof of \ref{it:2_lem:LM_centered_energy_space}.
\end{proof}

\begin{proof}[\textbf{Proof of Theorem \ref{prop:LM_operator_invertible}}]
\ref{it:1_prop:LM_operator_invertible}.
For \(\Psi\in C_c^\infty(\R^d)\),  integration by parts gives
\begin{equation}\label{eq:int_by_parts}
     \int_{\R^d}
    (L_{m_\rho,B}\Phi)\Psi\d x
    =
    \int_{\R^d}
    \nabla\Phi^\top B\nabla\Psi\,
    \d\mathfrak m_\rho.
\end{equation}
This identity extends to every \(\Psi\in H^1(m_\rho,B)\) by the following arguments:  the
left-hand side is continuous with respect to the
\(H^1(m_\rho,B)\)-norm because \(  \frac{L_{m_\rho,B}\Phi}{m_\rho}
    \in L^2(m_\rho)\), indeed,
\begin{align*}
    \left|
        \int_{\R^d}
        (L_{m_\rho,B}\Phi)\Psi\d x
    \right|
    =
    \left|
        \int_{\R^d}
        \frac{L_{m_\rho,B}\Phi}{m_\rho}\Psi\,
        \d\mathfrak m_\rho
    \right|
   \leq
    \left\|
        \frac{L_{m_\rho,B}\Phi}{m_\rho}
    \right\|_{L^2( m_\rho)}
    \|\Psi\|_{L^2( m_\rho)}.
\end{align*}
The right-hand side is also continuous due to the Cauchy--Schwarz inequality 
\[
    \left|
        \int_{\R^d}
        \nabla\Phi^\top B\nabla\Psi\,
        \d\mathfrak m_\rho
    \right|
    \leq
    \|\nabla\Phi\|_{L^2( m_\rho,B;\R^d)}
    \|\nabla\Psi\|_{L^2( m_\rho,B;\R^d)}.
\]
By Lemma \ref{lem:LM_centered_energy_space}, the space
\(C_c^\infty(\R^d)\) is dense in \(H^1(m_\rho,B)\). Thus, for every
\(\Psi\in H^1(m_\rho,B)\), there exists a sequence
\(\Psi_k\in C_c^\infty(\R^d)\) such that
$ \Psi_k\to\Psi$
in $H^1(m_\rho,B)$.
Applying the integration-by-parts identity~\eqref{eq:int_by_parts} to each \(\Psi_k\) and
passing to the limit, proves
\eqref{eq:LM_weak_integration_by_parts} for every
\(\Psi\in H^1(m_\rho,B)\).
\\[1ex]
\ref{it:2_prop:LM_operator_invertible}. The proof is divided into three parts. We first prove that $ L_{m_\rho, B}$ is well-defined, then that it is  surjective and injective.\\

\underline{Well-definedness:}
Since
\(\mathfrak m_\rho(\R^d)<\infty\), the constant function \(1\)
belongs to \(H^1(m_\rho,B)\). Taking \(\Psi=1\) in
\eqref{eq:LM_weak_integration_by_parts} gives
\[
    \int_{\R^d}L_{m_\rho,B}  \Phi\d x=0.
\]
By definition of
\(T_\rho^*\P_{m_\rho,B}\) in \eqref{eq:LM_regular_cotangent} we have 
\(
\frac{L_{m_\rho,B}\Phi}{m_\rho}
    \in L^2( m_\rho).\) Since \(m_\rho\) and \(\Phi\) are smooth,
\[
    L_{m_\rho,B}\Phi
    =
    -\nabla\cdot(m_\rho B\nabla\Phi)
\]
is also smooth. Therefore,
$
    L_{m_\rho,B}\Phi
    \in T_\rho\P_{m_\rho,B}
$.
Furthermore, for every \(c\in\R\),
\begin{equation}\label{eq:L_rho,m;B_well defined on_equivalence_sets}
    L_{m_\rho,B}(\Phi+c)
    =
    L_{m_\rho,B}\Phi.
\end{equation}
Thus, the value of \(L_{m_\rho,B}[\Phi]\) is independent of the
representative of \([\Phi]\), and the operator in
\eqref{eq:LM_operator_map} is well defined. Its linearity follows
directly from its definition. \\

\underline{Surjectivity:} By Lemma \ref{lem:LM_centered_energy_space}, we know that the bilinear form
\begin{equation}\label{eq:def_a_rho}
     a_\rho(\Phi,\Psi)
    \coloneqq
    \int_{\R^d}
    \nabla\Phi^\top B\nabla\Psi\,
    \d\mathfrak m_\rho.
\end{equation}
is bounded and coercive on $\left(H_0^1(m_\rho,B),\| \cdot\|_{H^1(m_\rho,B)} \right)$. 
For every \(\dot\rho\in T_\rho\P_{m_\rho,B}\), the  linear functional 
\begin{equation}\label{eq:ell_rho}
     \ell_{\dot\rho}\colon H^1_0(m_\rho,B)\to \R, \qquad\ell_{\dot\rho}(\Psi)
    \coloneqq
    \int_{\R^d}\dot\rho\,\Psi\d x.
\end{equation}
is continuous on
\(H^1_0(m_\rho,B)\), since
we have by \eqref{eq:LM_PI} that
\begin{align*}
    |\ell_{\dot\rho}(\Psi)|
    =
    \left|
        \int_{\R^d}
        \frac{\dot\rho}{m_\rho}\Psi\,
        \d\mathfrak m_\rho
    \right|
    \leq
    \left\|
        \frac{\dot\rho}{m_\rho}
    \right\|_{L^2( m_\rho)}
    \|\Psi\|_{L^2( m_\rho)}
    \leq
    \sqrt{C_\rho}
    \left\|
        \frac{\dot\rho}{m_\rho}
    \right\|_{L^2( m_\rho)}
    \|\Psi\|_{m_\rho,B}.
\end{align*}
Then, by the Lax--Milgram theorem, there exists  a unique
\(\Phi_{\dot\rho}\in H^1_0(m_\rho,B)\) such that
\begin{equation}\label{eq:LM_weak_potential_equation}
    \int_{\R^d}
    \nabla\Phi_{\dot\rho}^\top B\nabla\Psi\,
    \d\mathfrak m_\rho
    =
    \int_{\R^d}\dot\rho\,\Psi\d x
    \quad
    \text{for all} \quad \Psi\in H^1_0(m_\rho,B).
\end{equation}
Next, let \(\psi\in C_c^\infty(\R^d)\) and define its centered
version $\Psi$ by \eqref{centered}, in particular, \( \nabla\Psi=\nabla\psi.\)
Furthermore, since
\(\dot\rho\in T_\rho\P_{m_\rho,B}\),  it has zero mean and therefore
\begin{align*}
    \int_{\R^d}\dot\rho\,\Psi\d x
    =
    \int_{\R^d}\dot\rho\,\psi\d x
    -
    \frac{1}{\mathfrak m_\rho(\R^d)}
    \int_{\R^d}\psi\d\mathfrak m_\rho
    \int_{\R^d}\dot\rho\d x
    =
    \int_{\R^d}\dot\rho\,\psi\d x.
\end{align*}
In combination with \eqref{eq:LM_weak_potential_equation}, this implies
\[
    \int_{\R^d}
    \nabla\Phi_{\dot\rho}^\top B\nabla\psi\,
    \d \mathfrak m_\rho
    =
    \int_{\R^d}\dot\rho\,\psi\d x
    \quad
    \text{for all}
    \quad\psi\in C_c^\infty(\R^d).
\]
By the definition of distributional derivatives, this is equivalent
to
\begin{equation}
\label{eq:LM_distributional_potential_equation}
    -\nabla\cdot
    \left(
        m_\rho B\nabla\Phi_{\dot\rho}
    \right)
    =
    \dot\rho
    \qquad
    \text{in }\mathcal D'(\R^d).
\end{equation}
Next, we verify that the weak solution
\(\Phi_{\dot\rho}\) is smooth. Let \(K\subset\R^d\) be compact.
Since \(m_\rho\) is smooth and strictly positive, there exists
\(c_K>0\) such that
$
    m_\rho(x)\geq c_K
    $
for all $x\in K$.
Therefore, we obtain for every \(x\in K\) and \(\xi\in\R^d\) that
\[
    \xi^\top m_\rho(x)B\xi
    \geq
    c_K\lambda_{\min}(B)|\xi|^2.
\]
Hence the coefficient matrix \(m_\rho B\) is smooth and uniformly
elliptic on every compact subset of \(\R^d\). 
Since \(m_\rho\) is smooth and strictly positive and \(B\) is
symmetric positive definite, the coefficient matrix \(m_\rho B\)
is smooth and uniformly elliptic on every relatively compact
open subset of \(\R^d\).
These local lower bounds also imply that
\(\Phi_{\dot\rho}\in H^1(m_\rho,B)\) belongs to
\(H^1_{\mathrm{loc}}(\R^d)\).
Thus, since \(\dot\rho\in C^\infty(\R^d)\), standard interior
elliptic regularity applied to
\eqref{eq:LM_distributional_potential_equation} yields
\(\Phi_{\dot\rho}\in C^\infty(\R^d)\).
Moreover, by
\eqref{eq:LM_distributional_potential_equation} it holds
$
    L_{m_\rho,B}\Phi_{\dot\rho}
    =
    \dot\rho
$.
Consequently,  using \(\dot\rho\in T_\rho\P_{m_\rho,B}\), we have
\[
    \frac{L_{m_\rho,B}\Phi_{\dot\rho}}{m_\rho}
    =
    \frac{\dot\rho}{m_\rho}
    \in L^2( m_\rho),
\]
 Together with \( \Phi_{\dot\rho} \in C^\infty(\R^d)\cap H^1_0(m_\rho,B)\), this proves that \([\Phi_{\dot\rho}]
    \in T_\rho^*\P_{m_\rho,B}.\)
Finally, by \eqref{eq:L_rho,m;B_well defined on_equivalence_sets}, we conclude
$
    L_{m_\rho,B}[\Phi_{\dot\rho}]
    =
    \dot\rho
$,
which proves the surjectivity of \(L_{m_\rho,B}\).\\

\underline{Injectivity:}
Suppose that \([\Phi]\in T_\rho^*\P_{m_\rho,B}\) and \(L_{m_\rho,B}[\Phi]=0\).
Applying \eqref{eq:LM_weak_integration_by_parts} with
\(\Psi=\Phi\), we obtain
\[
    0
    =
    \int_{\R^d}
    \nabla\Phi^\top B\nabla\Phi\,
    \d\mathfrak m_\rho.
\]
Since \(m_\rho>0\) and \(B\succ0\), the integrand is nonnegative and
can vanish only if \(\nabla\Phi=0\). Thus, \(\Phi\) is constant on \(\R^d\), and hence
$ [\Phi]=0$.
This proves the injectivity of $L_{m_\rho, B}$ and finishes the proof.
\end{proof}

\begin{proof}[\textbf{Proof of Corollary \ref{cor:LM_gradient}}]
First using \(\left[
        \frac{\delta\F}{\delta\rho}
    \right]
    \in
    T_\rho^*\P_{m_\rho,B}\) in combination with the first part of Theorem \ref{prop:LM_operator_invertible} shows that
    \[
    -\nabla\cdot
    \left(
        m_\rho B
        \nabla\frac{\delta\F}{\delta\rho}
    \right)
    \in
    T_\rho\P_{m_\rho,B}.
    \]
For any \(\dot\eta\in
    T_\rho\P_{m_\rho,B}\) with  associated 
 local-mobility potential \(\Phi_{\dot\eta}\)  we get by \eqref{eq:LM_metric_bilinear} and integration by parts 
\[
G_{m_\rho,B}
    \left(
        L_{m_\rho,B}
        \left[
            \frac{\delta\F}{\delta\rho}
        \right],
        \dot\eta
    \right)
    =
    \int_{\R^d}
    \Big(\nabla\frac{\delta\F}{\delta\rho} \Big)^{\top}
    B\nabla \Phi_{\dot\eta}\,
    \d\mathfrak m_\rho
   =
    \int_{\R^d}
    \frac{\delta\F}{\delta\rho}\,
    \dot\eta\d x
  =
    D\F(\rho)[\dot\eta].
\]
Thus, \eqref{eq:LM_gradient} satisfies the defining gradient identity.
Uniqueness follows from the positive definiteness of
\(G_{m_\rho,B}\).
\end{proof}

\section{Proofs for Section~\ref{sec:stein}} \label{app:5}
\begin{proof}[\textbf{Proof of Lemma \ref{lem:well-defined}}]
\ref{it:1_lem_well_def}. Linearity is immediate, and constants do not change $M_\rho$, so the formula
does not depend on the representative.  Next, a quick computation confirms  for every $C\in\R^{d\times d}$ that
\begin{equation}\label{eq:linear-score}
  \frac{-\nabla\!\cdot(\rho Cx)}{\rho}
  =-\tr(C)-\nabla\log\rho(x)^\top Cx.
\end{equation}
For $C=M_\rho(\Psi)^\top A$, this becomes
\begin{equation}\label{eq:onsager-score}
  \frac{L_\rho^{K_A}[\Psi]}{\rho}
  =-\tr\!\left(M_\rho(\Psi)^\top A\right)
   -\nabla\log\rho(x)^\top M_\rho(\Psi)^\top Ax.
\end{equation}
Using \eqref{eq:onsager-score} in combination with \eqref{eq:density-class}, we obtain
\begin{equation}\label{eq:onsager-bound}
    \left\lVert
    \frac{L_\rho^{K_A}[\Psi]}{\rho}
    \right\rVert_{L^2(\rho)}
    \leq
    \left[
      \sqrt d+
      \left(
        \int_{\R^d}|x|^2|\nabla\log\rho|^2 \d\rho
      \right)^{1/2}
    \right] 
    \lVert A\rVert_{\mathrm{op}}
    \lVert M_\rho(\Psi)\rVert_{\mathrm F}
    <\infty.
\end{equation}
The output is smooth and satisfies the required score condition.

Further, $L_\rho^{K_A}[\Psi]$ belongs to
$T_\rho\P_{K_A}$ for the following reason: let $\chi_R$ be a smooth cutoff with
$0\leq\chi_R\leq1$, equal to $1$ on $B_R$, supported in $B_{2R}$, and satisfying
$|\nabla\chi_R|\leq C/R$.  The score bound makes the output integrable, and
distributional integration by parts gives
\begin{align}
    \left|\int_{\R^d}L_\rho^{K_A}[\Psi] \d x\right|&=\lim_{R\to\infty}
  \left|\int_{\R^d}\rho M_\rho(\Psi)^\top Ax
    \cdot\nabla\chi_R \d x\right|\notag 
    \\
    &\leq
  2C\lVert M_\rho(\Psi)^\top A\rVert_{\mathrm{op}}
  \lim_{R\to\infty}
  \int_{\{|x|>R\}}\rho(x) \d x=0.
  \label{eq:zero-mass}
\end{align}
\\[1ex]
\ref{it:2_lem_well_def}. If $\Phi$ is compactly supported, integration by parts yields
\begin{align}
  \big\langle[\Phi],L_\rho^{K_A}[\Psi]\big\rangle_\rho
=\int_{\R^d}\nabla\Phi^\top
    M_\rho(\Psi)^\top Ax \d\rho=\tr\!\left(M_\rho(\Phi)^\top A M_\rho(\Psi)\right).
  \label{eq:green-compact}
\end{align}
For general $\Phi\in C^\infty\cap H^1(\rho)$, the cutoffs
$\Phi_R\coloneqq \chi_R\Phi$ satisfy
\begin{equation}
     \lVert\Phi_R-\Phi\rVert_{L^2(\rho)}\to 0,\quad \text{and} \quad
  \lVert\nabla\Phi_R-\nabla\Phi\rVert_{L^2(\rho)}
  \leq
  \lVert(1-\chi_R)\nabla\Phi\rVert_{L^2(\rho)}
  +\frac{C}{R}
   \lVert\Phi\mathbf 1_{\{R<|x|<2R\}}\rVert_{L^2(\rho)}
  \to 0.
\end{equation}
Thus $\Phi_R\to\Phi$ in $H^1(\rho)$.  By Cauchy--Schwarz and
\eqref{eq:onsager-bound},
\begin{equation}\label{eq:difference_to_zero}
    \left|
  \big\langle[\Phi_R-\Phi],
    L_\rho^{K_A}[\Psi]\big\rangle_\rho
  \right|\leq
  \lVert\Phi_R-\Phi\rVert_{L^2(\rho)}
  \left\lVert
    \frac{L_\rho^{K_A}[\Psi]}{\rho}
  \right\rVert_{L^2(\rho)}
  \to 0.
\end{equation}
Next, using \[
  \lVert M_\rho(\Psi)^\top Ax\rVert_{L^2(\rho)}
  \leq
  \lVert M_\rho(\Psi)^\top A\rVert_{\mathrm{op}}
  \left(\int_{\R^d}|x|^2 \d\rho\right)^{1/2}<\infty
\]
in combination with \eqref{eq:difference_to_zero}, we conclude that the gradient integral in \eqref{eq:green-compact} satisfies
\begin{align*}
  &\left|
  \int_{\R^d}\nabla(\Phi_R-\Phi)^\top
    M_\rho(\Psi)^\top Ax \d\rho
  \right|\leq
  \lVert\nabla\Phi_R-\nabla\Phi\rVert_{L^2(\rho)}
  \lVert M_\rho(\Psi)^\top A\rVert_{\mathrm{op}}
  \left(\int_{\R^d}|x|^2 \d\rho\right)^{1/2}.
\end{align*}
Thus, passing to the limit in \eqref{eq:green-compact} proves
\eqref{eq:green-identity}.
\end{proof}

\begin{proof}[\textbf{Proof of Lemma \ref{lem:kernel}}]
The inclusion from right to left follows from the definition of
$L_\rho^{K_A}$. Conversely, if $L_\rho^{K_A}[\Phi]=0$, the Green
identity in item~\ref{it:2_lem_well_def} of Lemma~\ref{lem:well-defined}, with $\Psi=\Phi$, gives
\begin{equation}\label{eq:kernel-energy}
  0
  =\tr\!\left(M_\rho(\Phi)^\top A M_\rho(\Phi)\right)
  =\left\lVert A^{1/2}M_\rho(\Phi)\right\rVert_{\mathrm F}^2
\end{equation}
with the Frobenius norm $\| \cdot\|_{\mathrm F}$.
Since $A$ is positive definite, $M_\rho(\Phi)=0$.

For infinite-dimensionality, choose pairwise disjoint open balls
$B_n\subset\R^d$ and nonzero $\varphi_n\in C_c^\infty(B_n)$. Their classes
$[\varphi_n]$ are linearly independent in
$T_\rho^*\P_{K_A}$, so this cotangent space is infinite-dimensional.
On the other hand, $L_\rho^{K_A}$ factors through
$\overline M_\rho$, whose range lies in $\R^{d\times d}$. Hence
$\ran L_\rho^{K_A}$ is finite-dimensional. By the first isomorphism theorem,
\[
  T_\rho^*\P_{K_A}/\ker L_\rho^{K_A}
  \cong\ran L_\rho^{K_A}.
\]
Thus, $\ker L_\rho^{K_A}$ has finite codimension in an infinite-dimensional
space and is itself infinite-dimensional.
\end{proof}

\begin{proof}[\textbf{Proof of Theorem \ref{prop:isomorphism}}]
If two classes in $T_\rho^*\P_{K_A}$ determine the same quotient class, their
difference belongs to $\ker L_\rho^{K_A}$, so
\eqref{eq:induced-isomorphism} is well-defined.  Its kernel is zero by the
definition of the quotient, and it is onto by the definition
$T_\rho^{\mathrm{eff}}\P_{K_A}=\ran L_\rho^{K_A}$.  Hence it is an isomorphism.  The final
identification follows from Lemma~\ref{lem:kernel} and the first isomorphism
theorem applied to $M_\rho$.
\end{proof}

\section{Proofs for Section~\ref{sec:softmax-regular-dynamics}}
\label{app:softmax-proofs}


\begin{proof}[\textbf{Proof of Proposition~\ref{prop:sm-regular-variation}}]
We first compute the variation along compactly supported
density perturbations.
Let $\eta\in C_c^\infty(\R^d)$ satisfy
$\int_{\R^d}\eta \d x=0$.
By Lemma~\ref{lem:asump_differtiability_OW-SM}, we have that
$\rho+\varepsilon\eta$ belongs to $\P_{\SM}$ for
sufficiently small $|\varepsilon|$.
Since $\eta/\rho$ is bounded, these perturbed densities
and their normalizers are uniformly comparable with
$\rho$ and $Z_\rho$, respectively.
Consequently, their interaction energies remain finite.

Symmetry of $A$ in \eqref{eq:sm-energy} gives
\[
\F(\rho+\varepsilon\eta)-\F(\rho)
=
-\varepsilon\int_{\R^d}Z_\rho\eta \d x
-\frac{\varepsilon^2}{2}
\iint_{\R^d\times\R^d}
\exp(y^\top Ax)\eta(x)\eta(y) \d x \d y.
\]
The integrals are absolutely convergent because $\eta$
has compact support and $Z_\rho$ is smooth.
Thus, the first variation is represented by $-Z_\rho$.
Since this identity holds for every compactly supported
zero-integral test, any regular first variation must
equal $[-Z_\rho]$.
By \eqref{eq:LM_softmax_cotangent}, this proves the
necessity of \eqref{eq:sm-gradient-tests}.

For the converse, suppose that
\eqref{eq:sm-gradient-tests} holds.
Finite energy gives
\[
\int_{\R^d}Z_\rho^2m_\rho \d x
=
\int_{\R^d}\rho Z_\rho \d x<\infty.
\]
It remains to verify the weighted gradient condition
in \eqref{eq:LM_softmax_cotangent}.

Choose $\chi_R\in C_c^\infty(\R^d)$ with
$0\leq\chi_R\leq1$, $\chi_R=1$ on $B_R$,
$\operatorname{supp}\chi_R\subset B_{2R}$, and
$|\nabla\chi_R|\leq C/R$.
Multiplying $L_\rho^{\SM}Z_\rho$ by $\chi_R^2Z_\rho$,
integrating by parts, and applying Young's inequality
gives
\begin{equation}\label{eq:sm-variation-cutoff}
\begin{split}
\int_{\R^d}\chi_R^2
\nabla Z_\rho^\top B\nabla Z_\rho\,m_\rho \d x
&\leq
2\left(\int_{\R^d}Z_\rho^2m_\rho \d x\right)^{1/2}
\left(\int_{\R^d}
\frac{|L_\rho^{\SM}Z_\rho|^2}{m_\rho} \d x\right)^{1/2}\\
&\quad+
\frac{C}{R^2}
\int_{\R^d}Z_\rho^2m_\rho \d x,
\end{split}
\end{equation}
where $C$ is independent of $R$.
The right-hand side is bounded uniformly for $R\geq1$.
Letting $R\to\infty$ therefore yields
\[
\int_{\R^d}
\nabla Z_\rho^\top B\nabla Z_\rho\,m_\rho \d x<\infty.
\]
Together with \eqref{eq:sm-gradient-tests}, this proves
$[Z_\rho]\in T_\rho^*\P_{\SM}$.

The variation computed above is therefore a local
regular first variation in the sense of
\eqref{eq:def_local_regular_variation}.
Its extension to the regular tangent space is
\[
\ell_{\F,\rho}^{\reg}(\eta)
=
-\int_{\R^d}Z_\rho\eta \d x,
\qquad \eta\in T_\rho\P_{\SM}.
\]
This integral is finite by weighted Cauchy--Schwarz,
using $Z_\rho\in L^2(m_\rho)$ and the tangent-space
condition in \eqref{eq:LM_softmax_tangent}.
\end{proof}

\begin{proof}[\textbf{Proof of Corollary~\ref{thm:sm-gradient}}]
Proposition~\ref{prop:sm-regular-variation} identifies
the regular first variation of $\F$ with $[-Z_\rho]$.
By the metric definition \eqref{eq:SM_metric},
\[
G_\rho^{\SM}
\bigl(L_\rho^{\SM}[-Z_\rho],\eta\bigr)
=
-\int_{\R^d}Z_\rho\eta \d x
=
\ell_{\F,\rho}^{\reg}(\eta),
\qquad \eta\in T_\rho\P_{\SM}.
\]
Hence
$\grad_{G_\rho^{\SM}}\F(\rho)=L_\rho^{\SM}[-Z_\rho]$.

Differentiating the normalizer in \eqref{eq:def_Z_rho}
and using $BA=V$ gives
\[
B\nabla Z_\rho(x)
=
V\int_{\R^d}y\,\exp(y^\top Ax)\rho(y) \d y
=
Z_\rho(x)\Gamma_\rho(x).
\]
Substituting this identity into
\eqref{eq:LM_softmax_operator} proves
\eqref{eq:sm-gradient-identity}.
The gradient-flow formulation follows by inserting
this identity into \eqref{eq:sm-ordinary} at densities
satisfying the assumptions.
\end{proof}

\begin{proof}[\textbf{Proof of Lemma~\ref{prop:sm-pointwise}}]
By \eqref{eq:LM_softmax_class}, $Z_\rho$ is smooth
and strictly positive.
Moreover, $[\Phi]\in T_\rho^*\P_{\SM}$ gives
\[
\int_{\R^d}
\nabla\Phi^\top B\nabla\Phi\,m_\rho \d x<\infty.
\]
We use this integrability to control the nonlocal term
in the second equation of \eqref{eq:sm-accelerated-pde}.

Strict positivity of $\rho$ and invertibility of $A$
imply that, for every $r>0$, there exists $c_r>0$ such that
\begin{equation}\label{eq:sm-appendix-lower}
Z_\rho(x)\geq c_re^{r|x|}
\qquad(x\in\R^d).
\end{equation}
Indeed, choose $s>1$ such that
$(s-1)\sigma_{\min}(A)>r$.
For $x\ne0$, integrate over the unit ball centered at
$sAx/|Ax|$.
These balls lie in a fixed compact set on which $\rho$
has a positive minimum, and every point $y$ in the
corresponding ball satisfies
\[
y^\top Ax\geq(s-1)|Ax|>r|x|.
\]
This proves \eqref{eq:sm-appendix-lower} for $x\ne0$;
decreasing $c_r$ also includes $x=0$.

The nonlocal integral in
\eqref{eq:sm-accelerated-pde} can be written as
\[
\int_{\R^d}
\frac{\exp(y^\top Ax)}{Z_\rho(y)}
\nabla\Phi(y)^\top B\nabla\Phi(y)
m_\rho(y) \d y.
\]
For $x$ in a fixed compact set, the factor
$\exp(y^\top Ax)/Z_\rho(y)$ and each of its spatial
derivatives are bounded uniformly in $y$ by
\eqref{eq:sm-appendix-lower}.
The weighted gradient integrability of $\Phi$
therefore justifies differentiation under this integral.
Thus, the nonlocal term is smooth and finite.

The remaining spatial terms in
\eqref{eq:sm-accelerated-pde} are smooth because
$\rho$, $\Phi$, and $Z_\rho$ are smooth and
$Z_\rho>0$.
This proves the pointwise formulation whenever
$\alpha(t)$ is finite.
\end{proof}

\begin{proof}[\textbf{Proof of Theorem~\ref{thm:sm-hamiltonian}}]
We first compute the variation of
\eqref{eq:sm-hamiltonian} with respect to the potential.
For $[\Psi]\in T_\rho^*\P_{\SM}$, differentiation and
the Green identity for \eqref{eq:LM_softmax_operator}
give
\[
\left.\frac{\d}{\d\varepsilon}\right|_{\varepsilon=0}
\mathcal H(\rho,[\Phi+\varepsilon\Psi])
=
\int_{\R^d}\nabla\Psi^\top B\nabla\Phi\,m_\rho \d x
=
\left\langle[\Psi],L_\rho^{\SM}[\Phi]\right\rangle.
\]
Consequently, the regular potential derivative in
the sense of \eqref{eq:def_reg_del_phi_H} is
\begin{equation}\label{eq:sm-appendix-potential-derivative}
\frac{\delta^{\reg}\mathcal H}{\delta\Phi}
(\rho,[\Phi])
=
L_\rho^{\SM}[\Phi]
\in T_\rho\P_{\SM}.
\end{equation}

Next, we vary the density while keeping $[\Phi]$ fixed.
Let $\eta\in C_c^\infty(\R^d)$ satisfy
$\int_{\R^d}\eta \d x=0$, and set
\[
Z_\eta(x)\coloneqq
\int_{\R^d}\exp(y^\top Ax)\eta(y) \d y.
\]
By Lemma~\ref{lem:compatible_variations_otto_softmax},
condition~\eqref{eq:sm-compatibility} ensures that
$[\Phi]\in T_{\rho+\varepsilon\eta}^*\P_{\SM}$
for sufficiently small $|\varepsilon|$.
The kinetic energy can therefore be evaluated along
these density variations.

The identity
\[
m_{\rho+\varepsilon\eta}
=
m_\rho
\frac{1+\varepsilon\eta/\rho}
     {1+\varepsilon Z_\eta/Z_\rho},
\qquad
\left|\frac{Z_\eta}{Z_\rho}\right|
\leq\left\|\frac{\eta}{\rho}\right\|_\infty,
\]
justifies differentiation under the kinetic-energy
integral, since $\eta/\rho$ is bounded and the
kinetic energy is finite.
It follows that
\[
\left.\frac{\d}{\d\varepsilon}\right|_{\varepsilon=0}
\frac12\int_{\R^d}
\nabla\Phi^\top B\nabla\Phi\,
m_{\rho+\varepsilon\eta} \d x
=
\frac12\int_{\R^d}
\nabla\Phi^\top B\nabla\Phi\,m_\rho
\left(\frac{\eta}{\rho}
-\frac{Z_\eta}{Z_\rho}\right)\d x.
\]

In the term containing $Z_\eta$, Fubini's theorem
rewrites the right-hand side as the pairing of $\eta$
with the expression inside the brackets in
\eqref{eq:sm-graph-tests}.
The interchange is justified by
\[
\frac{Z_{|\eta|}}{Z_\rho}
\leq\left\|\frac{\eta}{\rho}\right\|_\infty
\]
and the finite kinetic energy.
Condition~\eqref{eq:sm-graph-tests} therefore makes
this density variation a regular cotangent class.
Compactly supported zero-integral tests determine
the class uniquely.

Adding the interaction-energy variation from
Proposition~\ref{prop:sm-regular-variation} gives
\begin{equation}\label{eq:sm-appendix-density-derivative}
\frac{\delta^{\reg}\mathcal H}{\delta\rho}
(\rho,[\Phi])
=
\left[
\frac{\nabla\Phi(x)^\top B\nabla\Phi(x)}{2Z_\rho(x)}
-\frac12\int_{\R^d}\exp(y^\top Ax)
\frac{\nabla\Phi(y)^\top B\nabla\Phi(y)}
     {Z_\rho(y)^2}\rho(y) \d y
-Z_\rho(x)
\right].
\end{equation}
Both contributions belong to $T_\rho^*\P_{\SM}$,
so their pairing with a regular tangent vector
defines the regular density derivative in the sense
of \eqref{eq:def_reg_del_rho_H}.

Finally, substituting
\eqref{eq:sm-appendix-potential-derivative} and
\eqref{eq:sm-appendix-density-derivative} into
\eqref{eq:sm-regular-hamiltonian-system} yields
\eqref{eq:sm-accelerated-pde}, with the potential
equation understood modulo spatial constants.
\end{proof}

\begin{proof}[\textbf{Proof of Lemma~\ref{lem:sm-particle-connection}}]
The proof consists of computing the connection of
\eqref{eq:softmax_metric_particles}.
Fix $\mathbf U,\mathbf Q\in\R^{N\times d}$ and
regard them as constant vector fields.
The Koszul formula then reduces to
\begin{equation}\label{eq:sm-appendix-koszul}
2g_{\mathbf X}
\bigl(\Gamma_{\mathbf X}^g(\mathbf U,\mathbf U),
\mathbf Q\bigr)
=
2Dg_{\mathbf X}[\mathbf U](\mathbf U,\mathbf Q)
-
Dg_{\mathbf X}[\mathbf Q](\mathbf U,\mathbf U).
\end{equation}
We compute the two derivatives by varying the base
point $\mathbf X$ in
\eqref{eq:softmax_metric_particles}, while keeping
the matrix arguments fixed.

For the first derivative,
\eqref{eq:softmax_dz_directional} gives
\[
Dg_{\mathbf X}[\mathbf U](\mathbf U,\mathbf Q)
=
\frac1N
\left\langle
D_{\delta z_{\mathbf X}[\mathbf U]}
\mathbf UB^{-1},\mathbf Q
\right\rangle_{\mathrm F}.
\]
For the second derivative, use
$s_i=U_i^\top B^{-1}U_i$ from the lemma.
Then
\[
Dg_{\mathbf X}[\mathbf Q](\mathbf U,\mathbf U)
=
\frac1N\sum_{i=1}^N
s_i \delta z_{\mathbf X}[\mathbf Q]_i.
\]
Expanding the directional derivative in
\eqref{eq:softmax_dz_directional} and using the
symmetry of $A$ and $\mathbf M$ yields
\[
\sum_{i=1}^N s_i \delta z_{\mathbf X}[\mathbf Q]_i
=
\left\langle
(D_s\mathbf M+\mathbf MD_s)\mathbf XA,
\mathbf Q
\right\rangle_{\mathrm F}.
\]

Substituting these expressions into
\eqref{eq:sm-appendix-koszul} and using
\eqref{eq:softmax_metric_particles} gives
\[
D_z\Gamma_{\mathbf X}^g(\mathbf U,\mathbf U)B^{-1}
=
D_{\delta z_{\mathbf X}[\mathbf U]}\mathbf UB^{-1}
-\frac12(D_s\mathbf M+\mathbf MD_s)\mathbf XA.
\]
Multiplication by $D_z^{-1}$ on the left and $B$
on the right, together with $AB=V^\top$, proves
\eqref{eq:softmax_connection_contracted}.
\end{proof}


\section{Proofs and Supplements of Section \ref{sec:lin-accelerated-linear-attention}}
\label{app:lin-attention-proofs}
This appendix supplies the proofs of the principal results for linear attention in
Section~\ref{sec:lin-accelerated-linear-attention} together with the
first-order analysis and supplementary comparisons.

\subsection{First-Order Linear Attention: Geometry and Analysis}
\label{app:lin-covariance-first-order}
We first verify that the covariance family
\eqref{eq:lin-elliptical-family} belongs to $\P_{K_A}$.
We then prove the three lemmas concerning its covariance
coordinates, induced metric, and gradient-flow formulation.
Throughout, only second moments are required.

Radiality of $\eta$ gives centeredness, and the change of
variables $x=S^{1/2}z$ gives covariance $S$ and the
weighted-score bound
\[
\int_{\R^d}|x|^2|\nabla\log\rho_S(x)|^2 \d\rho_S(x)
\leq
\|S\|_{\mathrm{op}}\|S^{-1}\|_{\mathrm{op}}
\int_{\R^d}|z|^2|\nabla\log\eta(z)|^2 \d\eta(z).
\]
Thus $\rho_S\in\P_{K_A}$.

\begin{proof}[\textbf{Proof of Lemma~\ref{lem:lin-covariance-isomorphism}}]
We first prove that $\Lambda_S$ is an isomorphism.
Diagonalize $SAS$ and denote its eigenvalues by $\lambda_i>0$.
In this eigenbasis,
\[
(\Lambda_S(P))_{ij}=(\lambda_i+\lambda_j)P_{ij},
\]
which proves invertibility.

To establish injectivity of $D\iota(S)$, we first identify
suitable representatives of the effective cotangent classes.
For a regular potential $\Psi$, set
$\psi(z)=\Psi(S^{1/2}z)$.
Integration along the infinitesimal rotations
$z_i\partial_j-z_j\partial_i$ gives
\[
\int_{\R^d}z_i\partial_j\psi \d\eta
=
\int_{\R^d}z_j\partial_i\psi \d\eta.
\]
Radial cutoffs justify this identity, with the limit controlled
by $|z|,|\nabla\psi|\in L^2(\eta)$.
It follows that
\[
S^{-1}M_{\rho_S}(\Psi)
=
S^{-1/2}M_\eta(\psi)S^{-1/2}
\in\Sym(d).
\]

For the reverse inclusion, fix $P\in\Sym(d)$ and take
\[
\Psi_R(x)=\tfrac12\chi_R(x)x^\top Px,
\]
where $\chi_R$ is smooth, equals one on $B_R$, vanishes
outside $B_{2R}$, and satisfies
$|\nabla\chi_R|\leq C/R$.
Then
\[
\|M_{\rho_S}(\Psi_R)-SP\|_{\mathrm F}
\leq
C_P\int_{|x|>R}|x|^2 \d\rho_S(x)
\longrightarrow0.
\]
The range of $M_{\rho_S}$ on $C_c^\infty(\R^d)$ is a
finite-dimensional linear subspace and hence closed.
There is therefore a compactly supported smooth
representative $\Psi_{S,P}$ satisfying
\[
M_{\rho_S}(\Psi_{S,P})=SP.
\]
The cutoff estimate and the closed-range argument use only
the second moment, not radiality. Thus, for every
$\rho\in\P_{K_A}$ and $P\in\Sym(d)$, they also give
\begin{equation}\label{eq:lin-regular-moment-realization}
\Psi_{\rho,P}\in C_c^\infty(\R^d),
\qquad
M_\rho(\Psi_{\rho,P})=S(\rho)P.
\end{equation}

Next, radiality and polar decomposition imply
\begin{equation}\label{eq:lin-elliptical-pushforward}
G_\#\rho_S=\rho_{GSG^\top}
\qquad(G\in\GL(d)).
\end{equation}
Differentiating this identity along a matrix curve with
$G_0=I_d$ and $\dot G_0=PSA$ yields
\[
D\iota(S)[\Lambda_S(P)]
=
-\nabla\cdot(\rho_SPSAx)
=
L_{\rho_S}^{K_A}[\Psi_{S,P}].
\]
If this tangent vector vanishes, the Green identity
\eqref{eq:green-identity} gives
\[
0
=
\tr\bigl((SP)^\top A(SP)\bigr)
=
\|A^{1/2}SP\|_{\mathrm F}^2.
\]
Hence $P=0$. Since $\Lambda_S$ is an isomorphism,
$D\iota(S)$ is injective.
\end{proof}

\begin{proof}[\textbf{Proof of Lemma~\ref{lem:lin-covariance-metric}}]
Fix $S\in\Sym_{++}(d)$ and $H,K\in\Sym(d)$, and set
\[
P=\Lambda_S^{-1}(H),
\qquad
Q=\Lambda_S^{-1}(K).
\]
The tangent representation established in the proof of
Lemma~\ref{lem:lin-covariance-isomorphism}, together with
the Green identity \eqref{eq:green-identity}, gives
\[
\begin{aligned}
G_{\rho_S}^{K_A}
\bigl(D\iota(S)[H],D\iota(S)[K]\bigr)
&=
\tr\bigl(M_{\rho_S}(\Psi_{S,P})^\top
A M_{\rho_S}(\Psi_{S,Q})\bigr)\\
&=
\tr(PSASQ)\\
&=
\tfrac12\tr\bigl(H\Lambda_S^{-1}(K)\bigr).
\end{aligned}
\]
This proves \eqref{eq:lin-covariance-metric-agreement}.

For completeness, symmetry follows by transposition and
cyclicity of the trace, while
\[
g_S^{\lin}(H,H)
=
\|A^{1/2}SP\|_{\mathrm F}^2>0
\qquad(H\neq0).
\]
Finally, smooth dependence of $\Lambda_S^{-1}$ on $S$
shows that \eqref{eq:lin-induced-metric} defines a smooth
Riemannian metric on $\Sym_{++}(d)$.
\end{proof}

\begin{proof}[\textbf{Proof of Proposition~\ref{prop:lin-covariance-geometry}}]
We first compute the gradient of $F^{\lin}$.
For every $H\in\Sym(d)$,
\[
g_S^{\lin}(-\Lambda_S(V),H)
=
-\tfrac12\tr(VH)
=
DF^{\lin}(S)[H].
\]
Thus
\[
\grad_{g^{\lin}}F^{\lin}(S)=-\Lambda_S(V).
\]

Let $S_t\in\Sym_{++}(d)$ be a $C^1$ curve and set
$\rho_t=\rho_{S_t}$.
The tangent identity established in the proof of
Lemma~\ref{lem:lin-covariance-isomorphism} gives
\[
\partial_t\rho_t+\nabla\cdot(\rho_tVS_tAx)
=
D\iota(S_t)
\bigl[\dot S_t-\Lambda_{S_t}(V)\bigr].
\]
Since $D\iota(S_t)$ is injective, the density curve solves
\eqref{eq:lin-ordinary-pde} if and only if
\[
\dot S_t
=
\Lambda_{S_t}(V)
=
-\grad_{g^{\lin}}F^{\lin}(S_t),
\]
which is \eqref{eq:lin-ordinary-ode}.
The initial data correspond under the parametrization $\iota$.
\end{proof}

\begin{proof}[\textbf{Proof of Corollary~\ref{cor:maxtime}}]
If $V\preceq0$, the scalar second moment is nonincreasing, since
\[
\frac{\d}{\d t}\tr(S_t)
=
2\tr(VS_tAS_t)\leq0.
\]
Thus the second moment remains bounded by its initial value.
Proposition~\ref{prop:lin-ordinary-continuation} therefore
implies $T_{\max}=\infty$.
\end{proof}

\subsection{
Accelerated Linear Attention: Geometry and Hamiltonian Dynamics}
\label{app:lin-quadratic-correspondence}

\begin{proof}[\textbf{Proof of Lemma~\ref{lem:lin-quadratic-reduction}}]
Since $\nabla\Phi_P(x)=Px$ and $\rho_S$ has covariance $S$,
the moment map in \eqref{eq:moment-map} satisfies
\begin{equation}\label{eq:lin-quadratic-moment}
M_{\rho_S}(\Phi_P)=SP.
\end{equation}
Differentiating \eqref{eq:lin-elliptical-pushforward}
at $G=I_d$ in the direction $PSA$ gives
\begin{equation}\label{eq:lin-tangent-identification}
D\iota(S)[\Lambda_S(P)]
=
-\nabla\cdot(\rho_SPSAx)
=
L_{\rho_S}^{K_A}[\Phi_P].
\end{equation}
Thus, the first equation in \eqref{eq:lin-accelerated-pde}
is equivalent to
\[
D\iota(S_t)[\dot S_t-\Lambda_{S_t}(P_t)]=0.
\]
The injectivity of $D\iota(S_t)$, established in
Appendix~\ref{app:lin-covariance-first-order},
then gives the first equation in
\eqref{eq:lin-accelerated-ode}.

For the second equation, \eqref{eq:lin-quadratic-moment}
and symmetry of $A,S,P$ give
\begin{equation}\label{eq:lin-quadratic-kinetic}
\nabla\Phi_P(x)^\top M_{\rho_S}(\Phi_P)^\top Ax
=
x^\top P^2SAx
=
\tfrac12x^\top(P^2SA+ASP^2)x.
\end{equation}
Substituting \eqref{eq:lin-quadratic-kinetic} into the
second equation in \eqref{eq:lin-accelerated-pde}
therefore yields
\[
\tfrac12x^\top
\bigl(\dot P_t+\alpha_tP_t+P_t^2S_tA+AS_tP_t^2-V\bigr)x=0
\qquad\text{for all }x\in\R^d.
\]
The matrix in parentheses is symmetric, so this identity
is equivalent to the second equation in
\eqref{eq:lin-accelerated-ode}.
The initial conditions follow from
$\rho_t=\rho_{S_t}$ and $\Phi_t=\Phi_{P_t}$.
\end{proof}

\begin{proof}[\textbf{Proof of Lemma~\ref{lem:lin-canonical-hamiltonian}}]
Differentiating the Hamiltonian in
\eqref{eq:lin-matrix-hamiltonian} gives
\[
\partial_\Pi\mathcal H^{\lin}(S,\Pi)
=
2(SAS\Pi+\Pi SAS),
\qquad
\partial_S\mathcal H^{\lin}(S,\Pi)
=
2(AS\Pi^2+\Pi^2SA)-\tfrac12V.
\]
Substituting $\Pi=P/2$ into the damped Hamiltonian
equations gives precisely
\eqref{eq:lin-accelerated-ode}.
\end{proof}

\begin{proof}[\textbf{Proof of Proposition~\ref{thm:lin-newton}}]
We first identify the Lagrangian associated with
\eqref{eq:lin-matrix-hamiltonian}. For
$S\in\Sym_{++}(d)$ and $H\in\Sym(d)$, set
\[
\mathcal L(S,H)
\coloneqq
\tfrac12g_S^{\lin}(H,H)-F^{\lin}(S).
\]
By \eqref{eq:lin-induced-metric}, its derivative with
respect to $H$ satisfies
\[
D_H\mathcal L(S,H)[K]
=
g_S^{\lin}(H,K)
=
\tfrac12\tr\bigl(\Lambda_S^{-1}(H)K\bigr)
\qquad(K\in\Sym(d)).
\]
Consequently, the momentum matrix associated with
$H=\Lambda_S(P)$ is $\Pi=P/2$.
The Legendre transform of $\mathcal L$ is therefore
\[
\tr(\Pi H)-\mathcal L(S,H)
=
2\tr(\Pi SAS\Pi)-\tfrac12\tr(VS)
=
\mathcal H^{\lin}(S,\Pi).
\]

It follows that the damped Hamiltonian equations in
Lemma~\ref{lem:lin-canonical-hamiltonian} are equivalent
to the damped Euler--Lagrange equation
\[
\frac{\d}{\d t}\partial_H\mathcal L(S_t,\dot S_t)
-
\partial_S\mathcal L(S_t,\dot S_t)
=
-\alpha_t\partial_H\mathcal L(S_t,\dot S_t).
\]
For the kinetic energy
$\tfrac12g_S^{\lin}(H,H)$, the Euler--Lagrange expression
is the metric dual of the covariant acceleration.
Moreover, $\partial_H\mathcal L(S_t,\dot S_t)$ is the
metric dual of $\dot S_t$.
Thus, the displayed equation is equivalent to
\eqref{eq:lin-newton}.

If $(S,P)$ solves \eqref{eq:lin-accelerated-ode},
then its first equation gives $S\in C^2$.
Lemma~\ref{lem:lin-canonical-hamiltonian} and the preceding
equivalence show that $S$ satisfies
\eqref{eq:lin-newton}.

Conversely, let $S\in C^2(I;\Sym_{++}(d))$ satisfy
\eqref{eq:lin-newton}. Since $\Lambda_S^{-1}$ depends
smoothly on $S$, the curve
\[
P_t=\Lambda_{S_t}^{-1}(\dot S_t)
\]
belongs to $C^1(I;\Sym(d))$.
The same equivalence shows that $(S,P/2)$ satisfies
the damped Hamiltonian equations, and hence $(S, P)$
solves \eqref{eq:lin-accelerated-ode}.
\end{proof}

\begin{proof}[\textbf{Proof of Theorem \ref{thm:lin-correspondence}}]
\ref{it:lin-correspondence-matrix}.
The right-hand side of \eqref{eq:lin-accelerated-ode}
is continuous in time and locally Lipschitz in $(S, P)$.
It also preserves symmetry. The Picard--Lindel\"of theorem
therefore gives a unique maximal $C^1$ solution with
initial values $(S_0, P_0)$. Let $G$ solve \eqref{eq:lin-propagator}.
Direct differentiation and uniqueness for the resulting
linear matrix equation give
\[
S_t=G_tS_0G_t^\top\succ0,
\]
since $G_t$ is invertible and $S_0\succ0$.
\\[1ex]
\ref{it:lin-correspondence-lift}.
Since $\det G_t>0$, the pushforward in
\eqref{eq:lin-accelerated-lift} has density
\begin{equation}\label{eq:lin-density-pushforward}
\rho_t(x)
=
\frac{\rho_0(G_t^{-1}x)}{\det G_t}.
\end{equation}
A change of variables gives $S(\rho_t)=G_tS_0G_t^\top=S_t$.
Since $\nabla\Phi_t(x)=P_tx$, it follows 
\[
M_{\rho_t}(\Phi_t)
=
\int_{\R^d}xx^\top P_t\,\rho_t(x) \d x
=
S_tP_t.
\]
The integral defining $M_{\rho_t}(\Phi_t)$ is absolutely convergent by the finite second
moment of $\rho_t$, and $t\mapsto S_tP_t$ is continuous.
This verifies condition~\ref{it:lin-strong-moment} of the strong solution definition. Next differentiating \eqref{eq:lin-density-pushforward} in time yields
\[
\partial_t\rho_t
=
-\tr(P_tS_tA)\rho_t
-
(P_tS_tAx)\cdot\nabla\rho_t
=
-\nabla\cdot(\rho_tP_tS_tAx),
\]
which is the first equation in
\eqref{eq:lin-accelerated-pde}.
For the second equation, a direct computation gives
\[
\partial_t\Phi_t+\alpha_t\Phi_t
+\nabla\Phi_t^\top M_{\rho_t}(\Phi_t)^\top Ax
-\tfrac12x^\top Vx
=
\tfrac12x^\top
\bigl(
\dot P_t+\alpha_tP_t
+P_t^2S_tA+AS_tP_t^2-V
\bigr)x
=0.
\]
Both identities hold in $C(K)$ for every compact
$K\subset\R^d$, proving condition~\ref{it:lin-strong-pde}.

The initial conditions follow from $G_0=I_d$ and
$P|_{t=0}=P_0$, verifying
condition~\ref{it:lin-strong-initial}.
Finally, since $G,P$ are $C^1$ and $\rho_0$ is smooth,
one can check that condition~\ref{it:lin-strong-regularity}
holds.
Thus, $(\rho,\Phi)$ is a strong solution.
\\[1ex]
\ref{it:lin-correspondence-converse}.
Let $(\rho,\Phi)$ be any strong solution with initial values
$(\rho_0,\Phi_0)$, where $\Phi_0=\Phi_{P_0}$.
By condition~\ref{it:lin-strong-moment}, the coefficient
$t\mapsto M_{\rho_t}(\Phi_t)^\top A$ is locally integrable.
Linear ODE theory therefore gives a unique, locally absolutely
continuous and invertible fundamental matrix satisfying
\begin{equation}\label{eq:lin-characteristic-propagator}
\dot G_t=M_{\rho_t}(\Phi_t)^\top AG_t,
\qquad G_0=I_d.
\end{equation}
Applying the chain rule to $t\mapsto\Phi_t(G_ty)$ and using
the second equation in \eqref{eq:lin-accelerated-pde},
we obtain
\[
\frac{\d}{\d t}\Phi_t(G_ty)
+\alpha_t\Phi_t(G_ty)
=
\tfrac12y^\top G_t^\top VG_ty,
\qquad
\Phi_0(G_0y)=\tfrac12y^\top P_0y.
\]
The integrating-factor formula gives the unique solution
of this scalar linear ODE. Returning to $x=G_ty$ yields
\begin{equation}\label{eq:lin-characteristic-potential}
\Phi_t(x)
=
\tfrac12x^\top G_t^{-\top}
\left[
\exp\left(-\int_0^t\alpha_r \d r\right)P_0
+
\int_0^t
\exp\left(-\int_s^t\alpha_r \d r\right)
G_s^\top VG_s \d s
\right]
G_t^{-1}x.
\end{equation}
Thus $\Phi_t=\Phi_{P_t}$ for a symmetric, locally absolutely
continuous matrix $P_t$. Similarly, applying the chain rule and using the density
equation in \eqref{eq:lin-accelerated-pde} gives
\[
\frac{\d}{\d t}
\bigl[\det G_t\,\rho_t(G_ty)\bigr]=0.
\]
The initial condition therefore yields
\eqref{eq:lin-density-pushforward}.
In particular, $\rho_t$ is a probability density with
finite second moment, and a change of variables gives
\[
S_t\coloneqq S(\rho_t)=G_tS_0G_t^\top.
\]
Since $\nabla\Phi_t(x)=P_tx$, we also have
$M_{\rho_t}(\Phi_t)=S_tP_t$.

Next, differentiating the second-moment identity
$S_t=S(\rho_t)=G_tS_0G_t^\top$ and the expression for the
quadratic coefficient $P_t$, here $P_t$ denotes the matrix between $x^\top$ and $x$ on
the right-hand side of \eqref{eq:lin-characteristic-potential}, gives that $(S_t,P_t)$ solves ~\eqref{eq:lin-accelerated-ode} almost everywhere.
Since $S, P, $ and $\alpha$ are continuous, classical bootstrapping implies that the matrix equations in ~\eqref{eq:lin-accelerated-ode} hold at every time.
Substituting $M_{\rho_t}(\Phi_t)=S_tP_t$ into
\eqref{eq:lin-characteristic-propagator} and using the symmetry
of $S_t$ and $P_t$ gives
\[
\dot G_t=P_tS_tAG_t,
\qquad G_0=I_d,
\]
which is precisely \eqref{eq:lin-propagator}.
Uniqueness of the solution to the matrix ODE system
\eqref{eq:lin-accelerated-ode} proves that $(\rho,\Phi)$
has the representation \eqref{eq:lin-accelerated-lift}.
\\[1ex]
\ref{it:lin-correspondence-geometry}.
The representation \eqref{eq:lin-density-pushforward}
shows that $\rho_t$ remains smooth and strictly positive.
Preservation of unit mass also follows from
\eqref{eq:lin-density-pushforward} by a change of variables.
Similarly,
\begin{equation}\label{eq:lin-fourth-propagation}
\int_{\R^d}|x|^4 \d\rho_t(x)
\leq
\|G_t\|_{\mathrm{op}}^4
\int_{\R^d}|y|^4 \d\rho_0(y).
\end{equation}
Applying the chain rule to
\eqref{eq:lin-density-pushforward} also yields
\begin{equation}\label{eq:lin-score-propagation}
\int_{\R^d}|x|^2|\nabla\log\rho_t(x)|^2 \d\rho_t(x)
\leq
\|G_t\|_{\mathrm{op}}^2
\|G_t^{-1}\|_{\mathrm{op}}^2
\int_{\R^d}|y|^2|\nabla\log\rho_0(y)|^2 \d\rho_0(y).
\end{equation}
Since $\rho_0\in\P_{K_A}^{4}$, the right-hand sides of
\eqref{eq:lin-fourth-propagation} and
\eqref{eq:lin-score-propagation} are finite for every
$t\in[0,T_{\max})$.
Together with smoothness, strict positivity, and unit mass,
these estimates prove that $\rho_t\in\P_{K_A}^{4}$
throughout this interval.

Next, we verify that $\partial_t\rho_t\in T_{\rho_t}^{\eff}\P_{K_A}$, with \(T_{\rho_t}^{\eff}\P_{K_A}\) as in~\eqref{eq:effective-spaces}. Using the first equation in \eqref{eq:lin-accelerated-pde}
together with the weighted-score estimate
\eqref{eq:lin-score-propagation}, we obtain
\[
\frac{\partial_t\rho_t}{\rho_t}
=
-\tr(P_tS_tA)
-
(P_tS_tAx)\cdot\nabla\log\rho_t
\in L^2(\rho_t).
\]
By a cutoff argument similar to that in the proof of
Lemma~\ref{lem:well-defined}, we obtain
$\int_{\R^d}\partial_t\rho_t \d x=0$.
Thus $\partial_t\rho_t\in T_{\rho_t}\P_{K_A}$,
where the latter space is defined as in
\eqref{eq:tangent}.
The finite fourth moment of $\rho_t$ gives
$\Phi_{P_t}\in H^1(\rho_t)$, so
$[\Phi_{P_t}]\in T_{\rho_t}^*\P_{K_A}$,
as defined in \eqref{eq:cotangent}.
The first equation in \eqref{eq:lin-accelerated-pde},
together with the definition of $L_{\rho_t}^{K_A}$
in \eqref{eq:onsager}, gives
\[
\partial_t\rho_t
=
L_{\rho_t}^{K_A}[\Phi_{P_t}].
\]
Thus $\partial_t\rho_t\in T_{\rho_t}^{\eff}\P_{K_A}$,
which completes the proof.
\end{proof}

\begin{proof}[\textbf{Proof of Corollary \ref{cor:lin-stability}}]
Fix $0<T<T_{\max}$.
Let $(S^{(n)},P^{(n)})$ and $(S,P)$ solve
\eqref{eq:lin-accelerated-ode} with initial values
$(S(\rho_0^{(n)}),P_0^{(n)})$ and
$(S(\rho_0),P_0)$, respectively.
Let $G^{(n)}$ and $G$ solve \eqref{eq:lin-propagator}
for these matrix pairs with initialization $I_d$.

Convergence of $\rho_0^{(n)}$ to $\rho_0$ in $W_2$ implies
$S(\rho_0^{(n)})\to S(\rho_0)$ in the Frobenius norm.
Continuous dependence for \eqref{eq:lin-accelerated-ode}
and \eqref{eq:lin-propagator}
\cite[Chapter~V]{Hartman2002ODE}
therefore gives existence of $S^{(n)},P^{(n)},G^{(n)}$
on $[0,T]$ for all sufficiently large $n$ and uniform
convergence to $S,P,G$ in the Frobenius norm.
The right-hand sides of these ODEs then also give uniform
convergence of their time derivatives.
By Theorem~\ref{thm:lin-correspondence}, the strong solutions
$(\rho^{(n)},\Phi^{(n)})$ consequently exist on $[0,T]$.

Using the pushforward representation
\eqref{eq:lin-accelerated-lift} and the triangle inequality
for $W_2$, we obtain
\[
W_2(\rho_t^{(n)},\rho_t)
\leq
\|G_t^{(n)}\|_{\mathrm{op}}W_2(\rho_0^{(n)},\rho_0)
+
\|G_t^{(n)}-G_t\|_{\mathrm{op}}
\left(\int_{\R^d}|x|^2 \d\rho_0(x)\right)^{1/2}.
\]
Uniform convergence of $G^{(n)}$ makes the matrices
$G_t^{(n)}$ uniformly bounded on $[0,T]$.
Together with the initial $W_2$ convergence, this estimate
therefore proves convergence in $C([0,T];W_2)$.

Since $G_t$ is invertible throughout the compact interval
$[0,T]$, the inverse matrices $(G_t^{(n)})^{-1}$ also
converge uniformly to $G_t^{-1}$.
Applying \eqref{eq:lin-density-pushforward} and using
$\rho_0^{(n)}\to\rho_0$ in $C^1_{\loc}(\R^d)$ now gives
\[
\rho^{(n)}\longrightarrow\rho
\qquad\text{in }C([0,T];C^1_{\loc}).
\]
The first equation in \eqref{eq:lin-accelerated-pde},
with $M_{\rho_t}(\Phi_t)=S_tP_t$, expresses
$\partial_t\rho_t$ in terms of $S_t,P_t,\rho_t$
and $\nabla\rho_t$.
The same identity holds for $(\rho^{(n)},\Phi^{(n)})$.
The convergence already established for these matrices
and functions therefore gives
\[
\partial_t\rho^{(n)}\longrightarrow\partial_t\rho
\qquad\text{in }C([0,T];C^0_{\loc}).
\]

Finally, the potential formulas in
\eqref{eq:lin-accelerated-lift} read
$\Phi_t^{(n)}(x)=\tfrac12x^\top P_t^{(n)}x$
and $\Phi_t(x)=\tfrac12x^\top P_tx$.
Uniform convergence of $P^{(n)}$ and $\dot P^{(n)}$
therefore gives
\[
\Phi^{(n)}\longrightarrow\Phi
\quad\text{in }
C([0,T];C^1_{\loc})
\cap C^1([0,T];C^0_{\loc}),
\]
completing the proof of \eqref{eq:lin-strong-stability}.
\end{proof}

\begin{proof}[\textbf{Proof of Proposition~\ref{prop:lin-continuation}}]
We first record the energy estimate used in both parts.
Define
\begin{equation}\label{eq:lin-mechanical-energy}
E(S,P)\coloneqq
\mathcal H^{\lin}(S,P/2)
=
\tfrac12\tr(PSASP)-\tfrac12\tr(VS).
\end{equation}
Differentiating along \eqref{eq:lin-accelerated-ode} gives
\begin{equation}\label{eq:lin-dissipation}
\frac{\d}{\d t}E(S_t,P_t)
=
-\alpha_t\|A^{1/2}S_tP_t\|_{\mathrm F}^2
\leq0.
\end{equation}

\par\noindent
\ref{it:lin-continuation-criterion}.
Suppose that $T_{\max}<\infty$ and that $S_t$ remains
bounded on $[0,T_{\max})$.
By \eqref{eq:lin-mechanical-energy} and
\eqref{eq:lin-dissipation},
\[
\|A^{1/2}S_tP_t\|_{\mathrm F}^2
\leq
2E(S_0,P_0)+\tr(VS_t).
\]
Since $A\succ0$ and $S_t$ is bounded, this bounds
$S_tP_t$ and $P_tS_t=(S_tP_t)^\top$.

Differentiating $S_t^{-1}$ using
\eqref{eq:lin-accelerated-ode} yields
\begin{equation}\label{eq:lin-inverse-moment}
\frac{\d}{\d t}S_t^{-1}
=
-S_t^{-1}(P_tS_t)A
-A(S_tP_t)S_t^{-1}.
\end{equation}
Gronwall's inequality therefore bounds $S_t^{-1}$
on $[0,T_{\max})$.
Consequently, $P_t=S_t^{-1}(S_tP_t)$ is also bounded.
The pair $(S_t,P_t)$ thus remains in a compact subset
of $\Sym_{++}(d)\times\Sym(d)$.
The standard continuation theorem for
\eqref{eq:lin-accelerated-ode} extends the solution
beyond $T_{\max}$, contradicting maximality;
see \cite{Hartman2002ODE}.

Finally, Theorem~\ref{thm:lin-correspondence} identifies
$S_t$ with the second-moment matrix of $\rho_t$.
Since $S_t\succ0$,
\[
\|S_t\|_{\mathrm F}
\leq
\tr(S_t)
=
\int_{\R^d}|x|^2 \d\rho_t(x)
\leq
\sqrt d\,\|S_t\|_{\mathrm F}.
\]
This proves both assertions in
\eqref{eq:lin-continuation-criterion}.
\\[1ex]
\ref{it:lin-continuation-global}.
If $V\preceq0$, both terms in
\eqref{eq:lin-mechanical-energy} are nonnegative.
Thus \eqref{eq:lin-dissipation} bounds $S_tP_t$
throughout the existence interval.
Writing the first equation in
\eqref{eq:lin-accelerated-ode} as
\[
\dot S_t=(P_tS_t)AS_t+S_tA(S_tP_t),
\]
Gronwall's inequality gives
\[
\|S_t\|_{\mathrm F}
\leq \|S_0\|_{\mathrm F}e^{Ct}
\]
for a finite constant $C$.
Hence $S_t$ cannot become unbounded in finite time.
Part~\ref{it:lin-continuation-criterion} therefore
implies $T_{\max}=\infty$.
\end{proof}

\begin{proof}[\textbf{Proof of Theorem~\ref{thm:lin-bounded-damping-blowup}}]
Set $y_i(t)\coloneqq\log s_i(t)$.
Eliminating $p_i$ from \eqref{eq:lin-diagonal-ode} gives
\begin{equation}\label{eq:lin-log-moment}
\ddot y_i+\alpha_t\dot y_i
=
2a_iv_i e^{y_i},
\qquad
y_i(0)=\log s_{i,0},
\qquad
\dot y_i(0)=0.
\end{equation}
\ref{it:lin-positive-moment-monotonicity}.
Fix an index $i$ with $v_i>0$ and set $c_i=2a_iv_i>0$.
The integrating-factor formula for
\eqref{eq:lin-log-moment} gives
\[
\dot y_i(t)
=
c_i\int_0^t
\exp\left(-\int_s^t\alpha_r \d r\right)
e^{y_i(s)} \d s
>0
\qquad(t>0).
\]
Therefore $s_i=e^{y_i}$ is strictly increasing.

\par\noindent
\ref{it:lin-bounded-damping-blowup}.
Consider the maximal solution of the $i$th scalar
system in \eqref{eq:lin-diagonal-ode}, and denote
its maximal existence time by $T_i$.
Suppose first that $T_i=\infty$.
Monotonicity of $y_i$ and boundedness of $\alpha$ give,
for $t\geq1$,
\[
\dot y_i(t)
\geq
c_i e^{y_i(0)}
\int_{t-1}^t \exp(-\bar\alpha(t-s)) \d s.
\]
The right-hand side is a positive constant independent
of $t$, so $y_i(t)\to\infty$ as $t\to\infty$.

To compare this growth with the damping, set
$r_i=\dot y_i e^{-y_i/2}$.
Equation \eqref{eq:lin-log-moment} gives
\[
\dot r_i
=
e^{y_i/2}\left(c_i-\frac{r_i^2}{2}\right)
-\alpha_t r_i.
\]
Choose $0<k<\sqrt{2c_i}$.
Once $y_i$ is sufficiently large, whenever
$0\leq r_i\leq k$ we have
\[
\dot r_i
\geq
\frac12\left(c_i-\frac{k^2}{2}\right)e^{y_i/2}
\geq
\frac{c_i-k^2/2}{2k} \dot y_i.
\]
Thus $r_i$ cannot remain below $k$ while $y_i$ tends
to infinity.
Moreover, at $r_i=k$ its derivative is positive,
so it cannot subsequently cross below $k$.
Consequently, for all sufficiently large times,
\[
\frac{\d}{\d t}e^{-y_i/2}
=
-\frac12r_i
\leq-\frac{k}{2}.
\]
This contradicts positivity of $e^{-y_i/2}$ on
$[0,\infty)$.
Hence $T_i<\infty$.

Applying Proposition~\ref{prop:lin-continuation} in
dimension one gives
$\limsup_{t\uparrow T_i}s_i(t)=\infty$.
Since $s_i$ is increasing, it follows that
$s_i(t)\to\infty$ as $t\uparrow T_i$.
The preceding argument also shows that, once $s_i$
is sufficiently large,
$\dot s_i=r_i s_i^{3/2}\geq k s_i^{3/2}$.

Every component with $v_i\leq0$ is global by
Proposition~\ref{prop:lin-continuation} in dimension one.
Since the scalar systems are decoupled,
\[
T_{\max}=\min_{1\leq i\leq d}T_i.
\]
If some $v_i>0$, this minimum is finite and attained
by a positive-forcing component.
For such an index,
\[
s_i(t)\longrightarrow\infty,
\qquad
\int_{\R^d}|x|^2 \d\rho_t(x)
=
\sum_{j=1}^d s_j(t)
\longrightarrow\infty
\]
as $t\uparrow T_{\max}$.
The equality of the matrix and PDE lifespans follows
from Theorem~\ref{thm:lin-correspondence}.
\\[1ex]
\ref{it:lin-commuting-global}.
If every $v_i\leq0$, then $V\preceq0$.
Global existence follows directly from
Proposition~\ref{prop:lin-continuation}.
\end{proof}

\subsection{Finite-Token Layers and Effective Depth}
\label{app:lin-continuation-proofs}

\begin{proof}[\textbf{Proof of Proposition \ref{prop:lin-transformer-lift}}]
\ref{it:lin-transformer-lift}.
Define $X(t)=X_0G_t^\top$ and $Y(t)=X(t)P_t$.
The matrix $G_t$ is invertible, since it solves the linear
system \eqref{eq:lin-propagator} with $G_0=I_d$.
Differentiating $G_t^{-1}S_tG_t^{-\top}$ using
\eqref{eq:lin-accelerated-ode} and
\eqref{eq:lin-propagator} shows that it is constant.
Its initial value is $S_0=X_0^\top X_0/N$, so
\[
\frac1N X(t)^\top X(t)
=
G_tS_0G_t^\top
=
S_t.
\]
Using this identity and \eqref{eq:lin-propagator}, we obtain
\[
\dot X
=
XASP
=
\frac1N XAX^\top Y.
\]
Differentiating $Y=XP$ and substituting
\eqref{eq:lin-accelerated-ode} similarly gives
\[
\dot Y
=
-\alpha Y-XP^2SA+XV
=
-\alpha Y-\frac1N YY^\top XA+XV.
\]
Here, the last equality uses $P=P^\top$ and $X^\top X/N=S$.
The initial values follow from $G_0=I_d$ and $P_0=0$.
Thus $(X,Y)$ solves \eqref{eq:lin-transformer-flow}.
Its right-hand side is continuous in time and locally
Lipschitz in $(X, Y)$, which proves uniqueness.
\\[1ex]
\ref{it:lin-transformer-maximal-time}.
Since $S_t$ is positive definite, the identity
$S_t=X(t)^\top X(t)/N$ in
\eqref{eq:lin-transformer-lift} gives
\begin{equation}\label{eq:lin-moment-token-bound}
\|S_t\|_{\mathrm F}
\leq
\tr(S_t)
=
\frac1N\|X(t)\|_{\mathrm F}^2.
\end{equation}
If $S_t$ is unbounded as $t\uparrow T_{\max}$,
then \eqref{eq:lin-moment-token-bound} implies that
$\|X(t)\|_{\mathrm F}$ is also unbounded.
A solution extending beyond $T_{\max}$ would be continuous
and hence bounded on $[0,T_{\max}]$, contradicting this
unboundedness of $\|X(t)\|_{\mathrm F}$.
Since part~\ref{it:lin-transformer-lift} gives a solution
on $[0,T_{\max})$, its maximal existence time is
exactly $T_{\max}$.
\end{proof}

\begin{proof}[\textbf{Proof of Lemma \ref{prop:lin-transformer-transfer}}]
Fix $0<T<T_{\max}$.
Since $\alpha\in C^1$, the solution $Z=(X,Y)$ of
\eqref{eq:lin-transformer-flow} belongs to $C^2([0,T])$.
On the closed unit neighbourhood of $Z([0,T])$, the
right-hand side of \eqref{eq:lin-transformer-flow}
is Lipschitz in $(X,Y)$, uniformly for $t\in[0,T]$.
The classical Euler estimate \cite[{\S}2.7]{Teschl2012} yields a constant $C_T \ge 0$ such that
\begin{equation}\label{eq:lin-token-euler-bound}
\max_{\ell h\leq T} \|Z_\ell^h-Z(\ell h)\|
\leq C_T
h
\end{equation}
for all sufficiently small $h$.

It remains to estimate the second moments.
Using $S(\ell h)=X(\ell h)^\top X(\ell h)/N$ from
\eqref{eq:lin-transformer-lift} and the definition of
$S_\ell^h$ in \eqref{eq:S_ell_h}, we obtain
\[
\|S_\ell^h-S(\ell h)\|_{\mathrm F}
\leq
\frac{\|X_\ell^h\|_{\mathrm F}
      +\|X(\ell h)\|_{\mathrm F}}{N}
\|X_\ell^h-X(\ell h)\|_{\mathrm F}.
\]
The continuity of $X$ on $[0,T]$ and
\eqref{eq:lin-token-euler-bound} bound both token matrices
uniformly for $\ell h\leq T$ and $0<h<h_T$.
Since
$\|X_\ell^h-X(\ell h)\|_{\mathrm F}
\leq\|Z_\ell^h-Z(\ell h)\|$,
the second-moment error is also bounded by a constant
times $h$.
Combining these estimates proves
\eqref{eq:lin-transformer-euler-error}.
\end{proof}

\begin{proof}[\textbf{Proof of Theorem \ref{thm:lin-transformer-depth}}]
Fix $R>s_*$.
For the chosen positive-forcing component,
\eqref{eq:lin-diagonal-ode} gives $p_{i_*}(t)>0$ for
$0<t<T_{\max}$: at every zero of $p_{i_*}$, its
derivative equals $v_{i_*}>0$.
Consequently, $s'(t_R)>0$.
By continuity of $s'$, we can choose $\delta_R>0$ and
$m_R>0$ with
$[t_R-\delta_R,t_R+\delta_R]\subset(0,T_{\max})$ such that
\begin{equation}\label{eq:lin-threshold-slope}
s'(t)\geq m_R
\qquad
\text{for all }t\in[t_R-\delta_R,t_R+\delta_R].
\end{equation}
Next, we compare the discrete second-moment matrices
$S_\ell^h$ with the continuous matrices $S(\ell h)$.
Applying \eqref{eq:lin-transformer-euler-error} with
$T=t_R+\delta_R$ gives a constant $C_R>0$ such that
\[
\max_{0\leq\ell h\leq t_R+\delta_R}
\|S_\ell^h-S(\ell h)\|_{\mathrm F}
\leq C_Rh
\]
for all sufficiently small $h>0$.
Using $S_\ell^h=(X_\ell^h)^\top X_\ell^h/N$ and
\eqref{eq:lin-directional-moment}, we can express the
error in the chosen direction as
\begin{equation}\label{eq:lin-threshold-moment-error}
\left|
\frac1N\|X_\ell^h e_{i_*}\|_2^2-s(\ell h)
\right|
=
\left|
e_{i_*}^\top\bigl(S_\ell^h-S(\ell h)\bigr)e_{i_*}
\right|
\leq C_Rh.
\end{equation}
The last inequality follows from $\|e_{i_*}\|_2=1$
and the Frobenius norm bound above.

We now locate the first discrete crossing within a
time window of width proportional to $h$ around $t_R$.
Choose $c_R>C_R/m_R$.
By \eqref{eq:lin-threshold-slope}, a time displacement
of $c_Rh$ near $t_R$ changes the continuous moment by
at least $m_Rc_Rh>C_Rh$, exceeding the discretization
error in \eqref{eq:lin-threshold-moment-error}.

First, we observe that, for sufficiently small $h>0$,
both $t_R-c_Rh$ and $t_R+(c_R+1)h$ belong to
$[t_R-\delta_R,t_R+\delta_R]$.
Next, for every layer with $\ell h\leq t_R-c_Rh$,
monotonicity of $s$, the derivative bound
\eqref{eq:lin-threshold-slope}, and the error estimate
\eqref{eq:lin-threshold-moment-error} give
\[
\frac1N\|X_\ell^h e_{i_*}\|_2^2
\leq s(t_R-c_Rh)+C_Rh
\leq R-(m_Rc_R-C_R)h
<R.
\]
Thus no layer of \eqref{eq:lin-transformer-layers}
with $\ell h\leq t_R-c_Rh$ reaches the threshold
defining $L_R(h)$ in \eqref{eq:lin-layer-threshold}.

For the upper bound, take $\ell$ to be the smallest
integer satisfying $\ell h\geq t_R+c_Rh$.
Since consecutive grid times differ by $h$,
\[
t_R+c_Rh\leq\ell h<t_R+(c_R+1)h.
\]
In particular, this grid time lies in the interval
where \eqref{eq:lin-threshold-slope} holds.
Applying \eqref{eq:lin-threshold-slope} and
\eqref{eq:lin-threshold-moment-error} at this layer gives
\[
\frac1N\|X_\ell^h e_{i_*}\|_2^2
\geq s(t_R+c_Rh)-C_Rh
\geq R+(m_Rc_R-C_R)h
>R.
\]
Therefore $L_R(h)$ is finite, and the two bounds yield
\begin{equation}\label{eq:lin-threshold-crossing-bracket}
t_R-c_Rh<hL_R(h)\leq t_R+(c_R+1)h.
\end{equation}
This proves $hL_R(h)=t_R+O_R(h)$.

Finally, the crossing times $t_R$ increase with $R$.
They cannot remain bounded away from $T_{\max}$,
because $s$ is bounded on every compact interval
contained in $[0, T_{\max})$.
Hence $t_R\uparrow T_{\max}$.
Taking the limits in
\eqref{eq:lin-threshold-crossing-bracket}, we obtain
\[
\lim_{R\to\infty}\lim_{h\downarrow0}hL_R(h)
=
\lim_{R\to\infty}t_R
=
T_{\max}.\qedhere
\]
\end{proof}

\begin{proof}[\textbf{Proof of Corollary \ref{cor:lin-transformer-amplification}}]
Choose $R>\max\{s_*,M^2\}$.
By Theorem~\ref{thm:lin-transformer-depth},
$hL_R(h)\to t_R<T_{\max}$.
Thus, for sufficiently small $h$, the layer
$\ell=L_R(h)$ satisfies $\ell h<T_{\max}$.
By \eqref{eq:lin-layer-threshold} and
$\|e_{i_*}\|_2=1$,
\[
M^2<R
\leq\frac1N\|X_\ell^h e_{i_*}\|_2^2
\leq\max_{1\leq j\leq N}\|x_{j,\ell}^h\|_2^2.
\]
This proves the claim.
\end{proof}

\begin{proof}[\textbf{Proof of Corollary \ref{cor:lin-transformer-critical-depth}}]
Fix $M>0$ and set $h=T/L$.
By Corollary~\ref{cor:lin-transformer-amplification},
for every sufficiently large $L$, there is a layer
$\ell$ with $\ell h<T_{\max}\leq T$ at which a token
norm exceeds $M$.
Since $T=Lh$, this layer satisfies $\ell\leq L$.
Hence $B_L(T)>M$ for all sufficiently large $L$.
As $M$ was arbitrary, the conclusion follows.
\end{proof}

\begin{proof}[\textbf{Proof of Corollary \ref{thm:lin-transformer-subcritical-depth}}]
Applying Lemma~\ref{prop:lin-transformer-transfer}
with $h=T/L$ bounds the approximation error by
$C_TT/L$ for all sufficiently large $L$.
Moreover,
\[
B_L(T)
\leq
\max_{0\leq t\leq T}\|X(t)\|_{\mathrm F}
+\frac{C_TT}{L}.
\]
The maximum on the right is finite because $X$ is
continuous on $[0,T]$.
This proves both conclusions.
\end{proof}

\section{Implementation Details} 
\label{section:Implementation}

\paragraph{Implementation Details}
We collect the implementation details for the experiments
in Section~\ref{sec:numerics}. The model configurations
and training protocols are specified separately for
the softmax and linear-attention experiments.

\paragraph{Datasets and tokenization.}
We use fixed GPT-2 BPE tokenizations of
\href{https://huggingface.co/datasets/roneneldan/TinyStories}
{TinyStories} \cite{eldan2023tinystories} and
\href{https://skylion007.github.io/OpenWebTextCorpus/}
{OpenWebText}, using \texttt{tiktoken}'s \texttt{gpt2} encoding
with $50{,}257$ token IDs and a model vocabulary padded to $50{,}304$.
We use the supplied TinyStories training and validation splits.
For OpenWebText, we split the documents of
\texttt{Skylion007/openwebtext} into $99.5\%$ training and $0.5\%$
validation using a shuffled split with preprocessing seed $1337$.
Each document is encoded separately with \texttt{encode\_ordinary},
followed by one end-of-text token (ID $50{,}256$).
The resulting documents are concatenated within each split and stored
as a fixed stream of unsigned 16-bit token IDs.
Input windows and their one-token-shifted targets may cross document
boundaries; the end-of-text token is retained in the prediction task,
and attention uses the causal mask without an additional document mask.
The primary softmax, linear-attention and looping studies draw training
windows from shuffled non-overlapping blocks, with a seeded boundary
offset and reshuffling at each epoch. The normalized-sigmoid and
native-L2 studies instead sample valid window starts uniformly with
replacement, using a deterministic generator keyed by the run seed
and optimizer update. Each comparison pairs the sampling rule and
batch sequence across methods.
The numbers of tokens in the preprocessed training and
validation sets are listed in
Table~\ref{tab:implementation-datasets}.
These dataset sizes are distinct from the training-token
budgets used in the individual experiments.

\paragraph{MLP and optimizer.}
The MLP consists of a linear map from width $d$ to $4d$, a GELU
activation, and a linear map back to $d$. All models use zero dropout
and bias-free linear maps and LayerNorms.
The optimizer is AdamW with peak learning rate
$6\times10^{-4}$, scalar-group learning-rate multiplier
$10$, and momentum coefficients $(0.9,0.95)$.
Weight decay is $0.1$ only for the token and position embeddings,
including the velocity embeddings when present; the tied output weight
is the same parameter as the token-state embedding and is counted once.
Attention and MLP matrices, LayerNorm scales, and learned scalar
coefficients have zero weight decay.
Training uses global gradient clipping at $1$,
$10\%$ warm-up, and cosine decay to $0.1$ of the
peak learning rate.

\begin{table}[htbp]
\centering
\begin{tabular}{lrr}
\toprule
Dataset & Training tokens & Validation tokens\\
\midrule
TinyStories & $473{,}992{,}236$ & $4{,}765{,}918$\\
OpenWebText & $8{,}995{,}237{,}948$ & $44{,}779{,}147$\\
\bottomrule
\end{tabular}
\caption{Numbers of tokens after GPT-2 BPE preprocessing.}
\label{tab:implementation-datasets}
\end{table}

\paragraph{Initialization.}
For an input token $a_i$ at position $i$, the initial token state is
$x_i^{(0)}=E_X(a_i)+P_X(i)$, where $E_X$ and $P_X$ are learned
token and position embedding tables.
All momentum variants additionally use
$v_i^{(0)}=E_V(a_i)+P_V(i)$, with separate learned tables $E_V,P_V$;
these are not tied to $E_X, P_X$ and the initial velocity is not zero.
The ordinary baseline $\mathrm B$ has no velocity tables or stream.
Linear-map and embedding weights are initialized from
$\mathcal N(0,0.02^2)$; LayerNorm scales start at one, with
normalization constant $10^{-5}$ and no bias.
The output-head weight is tied to the token-state embedding $E_X$,
not to the velocity embedding.
For a run with seed $s$, the velocity tables are initialized in the
order $E_V, P_V$ using a separate CPU random generator with seed
$s+10{,}000{,}019$, so paired momentum variants receive identical
initial velocity tables without altering the other model parameters.

\paragraph{Learned update coefficients.}
For each trainable block $k$ and branch $b\in\{A,M\}$ (attention or MLP),
the coefficients are scalar parameters broadcast over tokens and channels.
Writing $\theta_{\beta,k}^b$, $\theta_{\gamma,k}^b$, and
$\theta_{\mu,k}^b$ for their unconstrained trainable parameters, we use
\[
\beta_k^b=\sigma(\theta_{\beta,k}^b),\qquad
\gamma_k^b=\operatorname{softplus}(\theta_{\gamma,k}^b)+10^{-6},
\qquad
\mu_k^b=\sigma(\theta_{\mu,k}^b),
\]
where $\sigma(u)=(1+\exp(-u))^{-1}$ and
$\operatorname{softplus}(u)=\log(1+\exp(u))$.
The initial coefficient values are $\beta_k^b=0.9$,
$\gamma_k^b=1$, and $\mu_k^b=0.9$ wherever look-ahead is enabled.
For $\mathrm H_0$, $\mathrm H_1$, and $\mathrm R_2$,
attention look-ahead is instead fixed at $\mu_k^A=0$, with no
trainable attention look-ahead parameter; MLP look-ahead is retained.
Attention and MLP coefficients are learned separately, as are the
coefficients of distinct trainable blocks.
Their unconstrained parameters use ten times the base learning rate
and no weight decay. These are learned update coefficients, not
coefficients constrained to a prescribed continuous-time damping law.

\paragraph{Looping and parameter sharing.}
In Table~\ref{tab:separate-vs-shared}, both layouts execute eight
logical block applications. The tied layout applies two distinct
trainable blocks in the order $(0,1,0,1,0,1,0,1)$.
Sharing includes all attention and MLP weights, position and velocity
LayerNorms, and learned update coefficients within each block.
For each method and seed, the untied layout starts from eight separate
copies of the same two initialized blocks in this alternating order;
the copies subsequently train independently. Thus the tied and untied
models implement the same function at initialization, rather than
starting from eight independently sampled blocks.
Token and velocity embeddings, the final LayerNorm and the output
head are not duplicated across repetitions.
The token and velocity streams are initialized once per forward pass
and continue through all eight applications without resets or repeated
input injection. Training backpropagates through all eight applications
without detaching the streams; gradients from repeated uses accumulate
on the shared parameters.

\paragraph{Conventions}
Paired methods always share the training batches, tokenizer, data splits, and validation targets.
The reported endpoint is the final evaluation, rather than the best checkpoint.
Peak memory uses only complete dataset/seed blocks.

\paragraph{Hardware and model output.}
All runs use the same RTX A6000 cohort\footnote{As an exception, in the sigmoid-attention and $L^2$-attention experiments, TinyStories runs used A100 GPUs, and OpenWebText runs used  RTX A6000 GPUs. GPU-hours are reported separately by dataset and hardware, not as equal compute across devices. Hence, comparisons remain within each attention family.} and IEEE FP32 arithmetic without TF32, with no method-specific
tuning. All models apply a final LayerNorm to the token states,
followed by a linear output head that produces the
logits for next-token prediction.

\paragraph{Cost conventions.}
GPU-hours are the mean allocated GPU time per completed run, computed
from the scheduler's elapsed allocation time.
Each run uses one GPU. Means are taken over the paired seeds separately
for TinyStories and OpenWebText.
Peak MiB denotes peak PyTorch-allocated GPU tensor memory.
Parameter counts include all trainable embeddings, blocks, and final
normalization, counting shared weights only once.

\end{document}